%% file: article.tex
\documentclass[final,onefignum,onetabnum]{siamart190516}

\input{shared}

\ifpdf
\hypersetup{
  pdftitle={Uncertainty Quantification of Inclusion Boundaries in the Context of X-ray Tomography},
  pdfauthor={B. M. Afkham, Y. Dong, and P. C. Hansen}
}
\fi

\begin{document}

\maketitle

\begin{abstract}
  In this work, we describe a Bayesian framework for \coyq{reconstructing the boundaries of piecewise smooth regions in} the X-ray computed tomography (CT) problem \copc{in} an infinite-dimensional setting. \copc{In addition to} \coyq{the reconstruction, we are also able to quantify the uncertainty of the} \copc{predicted boundaries.}
  \copc{Our approach is goal oriented, meaning that we directly detect}
  the discontinuities \copc{from the data}, instead of reconstructing the entire image.
  \copc{This} drastically reduces the dimension of the problem\coyq{, which makes the application of Markov Chain Monte Carlo (MCMC) methods feasible.} We show that our method provides an excellent platform for challenging X-ray CT scenarios (e.g., in case of noisy data, limited angle, or sparse angle imaging). We investigate the \correct{performance} and accuracy of our method on synthetic data \coyq{as well as \correct{on} real-world data.} The numerical results indicate that our method provides an accurate method in detecting boundaries \coyq{of} piecewise smooth regions and quantifies the uncertainty in the
  prediction\delete{, in the context of X-ray CT}.

\end{abstract}

\begin{keywords}
  X-ray CT, Bayesian framework, inverse problems, Whittle-Mat\'ern field,
  \copc{goal-oriented UQ}
\end{keywords}

\begin{AMS}
  60G10, 60G15, 60G50, 60G60, 60J20, 28-08, 46B99, 65C05, 65C20, 65C40
\end{AMS}

\section{Introduction} \label{sec:intro}
Computed tomography (CT) imaging is the task of reconstructing a
\copc{positive attenuation field (in the form of an image)} from a finite number of projections (e.g., \coyq{sinograms}). CT reconstruction is often formulated as an inverse problem \cite{kaipio2006statistical}. \emph{Filtered back-projection} \cite{natterer2001mathematics} is a classic reconstruction approach to solve this inverse problem for some CT settings.
However, \copc{the quality of the reconstructed image is compromised by}
imaging challenges \copc{such as} noise, sparse angle imaging (to reduce the amount of harmful radiation), or limited angle imaging (to avoid obstacles in the imaging site or due to application set-up, e.g., in mammography). Therefore, finding alternative approaches for CT reconstruction has attracted attention in the past few decades \cite{chan2001active,dahl2017computing,ramlau2007mumford,soussen2004polygonal,sullivan1995reconstructing,wang2017x}.

In many CT reconstruction methods, such as algebraic iterative methods and regularization methods \cite{Hansen2021}, the goal is to identify objects in the image with approximately homogeneous attenuation coefficient distinct from the background attenuation \cite{debreuve1998attenuation,elhalawani2017matched,gibson2018automatic,weston2019automated}. Therefore, CT reconstruction is often followed by an image segmentation step to partition the image into piecewise smooth/constant regions. The boundaries between such regions often carry valuable information \cite{wang2018comprehensive}. Error propagation from reconstruction to segmentation, due to concatenation of such methods, can introduce artifacts in the \coyq{segmentation}. Such effects are also amplified in case of noisy or incomplete data.

To avoid such artifacts many methods attempt to combine the reconstruction and the segmentation (and subsequently boundary extraction) steps. One common approach is to describe the images as a level set of a smooth functions, see, e.g., \cite{alvino2004tomographic,chan2001active,ley2001lower,ramlau2007mumford,wang2017x,yoon2006level,yoon2010simultaneous} for \delete{some} selected examples. \correct{Lambda Tomography \cite{faridani1997local,faridani1992local,vainberg1981reconstruction,webber2020joint} is another approach in identifying boundaries. These methods are based on filtered back-projection method where the filters are chosen to emphasize boundaries.} Another common approach is to construct a deforming/parametric curve which evolves to fit the boundaries between partitions, see, e.g., \cite{dahl2017computing,soussen2004polygonal,sullivan1995reconstructing,thirion1992segmentation} as a non-exhaustive selection. Large parameter spaces, dependency on discretization, noisy data, and limited angle imaging can be challenging for some of the mentioned methods. \copc{Moreover}, all the methods mentioned above lack the quantification of the uncertainties in the reconstructed/segmented images with respect to noisy/perturbed data. Uncertainty quantification could be particularly important in applications where images are used to determine the location and the size of objects. For example, in medical imaging, CT reconstruction is used to \copc{track} the evolution of tumor boundaries over time and make treatment decisions.

A popular approach to characterize uncertainties for inverse problems is within the Bayesian framework. In this approach, all quantities of a model are represented as random variables. The solution to the inverse problem is then the probability distribution, \emph{the posterior distribution}, of the quantity of interest after all given information (e.g., the \emph{prior} belief in the model) is incorporated into the model. Qualities of the posterior distribution can then be interpreted as the degree of confidence in predicting the quantity of interest. \coyq{In general, Markov Chain Monte Carlo (MCMC) methods \cite{mcbook} are applied to explore the posterior distribution.} \coyq{In the context of X-ray CT, the Bayesian \copc{framework} has been successfully applied, especially for reconstructing the attenuation field, see, e.g.,
\delete{but not limited to} \cite{chapdelaine2019variational,chen2009bayesian,2104.06919,wang2017x}.
}

In the past decade, a formulation of a Bayesian inversion theory in an infinite dimensional setting has attracted attention, see \cite{Dashti2017} and the references therein. In this setting, model parameters are modeled as random functions rather than real-valued random variables. This provides a discretization-independent numerical platform for exploring the posterior distribution. \cofinal{Although posterior distributions for such problems are in infinite dimensions, the MCMC methods for exploring such distributions are similar to the more traditional (finite-dimensional) ones.}

Infinite-dimensional Bayesian methods for tomography problems are not new. \coyq{For example,}
\copc{we mention} recent works in the context of electrical impedance tomography (EIT), inverse scattering problem and other partial differential equation based inverse problems, see \cite{doi:10.1137/120894877,Dunlop2016,dunlop2020hyperparameter,liu2018image,huang2020bayesian,1930-8337_2016_4_1007,Roininen2014} and the references therein. The underlying partial differential equation provides a natural platform (a Hilbert space) for the infinite-dimensional random variable to be well defined. Detailed analysis of these inverse problems in a Bayesian setting \cite{Dashti2017} shows that the solution (the posterior distribution) is well-defined and bounded under perturbations \copc{of the} data. Therefore, the Bayesian platform provides an alternative method to investigate the well-posedness of an inverse problem.

However, to the best knowledge of the authors, we still lack a well-established infinite-dimensional framework for the \coyq{CT problem}. \coyq{In this paper, we show} that the well-posedness results, mentioned above, can be extended to the CT problem. This provides an excellent tool to evaluate and quantify the uncertainties in the reconstruction and segmentation of images.

\coyq{Here, we focus on looking for objects/inclusions in an image with a homogeneous attenuation coefficient distinct from the background. Especially, we are interested in reconstructing the boundaries of the objects and quantifying the uncertainties in this reconstruction due to noisy or incomplete data. One approach in evaluating such uncertainties is to consider a Markov random field (MRF) prior. We can estimate uncertainties using, e.g., MCMC methods. The edges are then detected using thresholding-based methods \cite{10.5555/1965423}. We call this approach the sampling-then-threshholding (STT) method. The advantage of STT methods is that advanced methods
\copc{for} reconstruction and edge detection/segmentation have been proposed. But some of the challenges of using STT methods are: 1) error propagation from reconstruction to segmentation; 2) resolution dependency, i.e., the quality of reconstructed boundaries in STT depends on the resolution of discretized attenuation field; 3) \copc{large} computational cost due to the large number of samples of a 2D reconstruction problem. \copc{These challenges make the use of STT methods unattractive.} In \Cref{fig:1.pixel}, we show a numerical comparison between our method, which will be introduced later, and an STT method that combines the two methods in \cite{10.5555/1965423,2104.06919}. The STT method obtains the boundaries of the objects in each sample after applying the edge detection algorithm. The mean of the boundaries can give us some information on the uncertainties in the location of the edges. But note that in this way we cannot really quantify the uncertainties of the boundary curve. In addition, in \Cref{fig:1.pixel}(f) we can clearly see the resolution dependency.}

\begin{figure}[thbp]
    \begin{center}
    \begin{minipage}[t]{4cm}
\includegraphics[totalheight=3.8cm,
keepaspectratio=true]{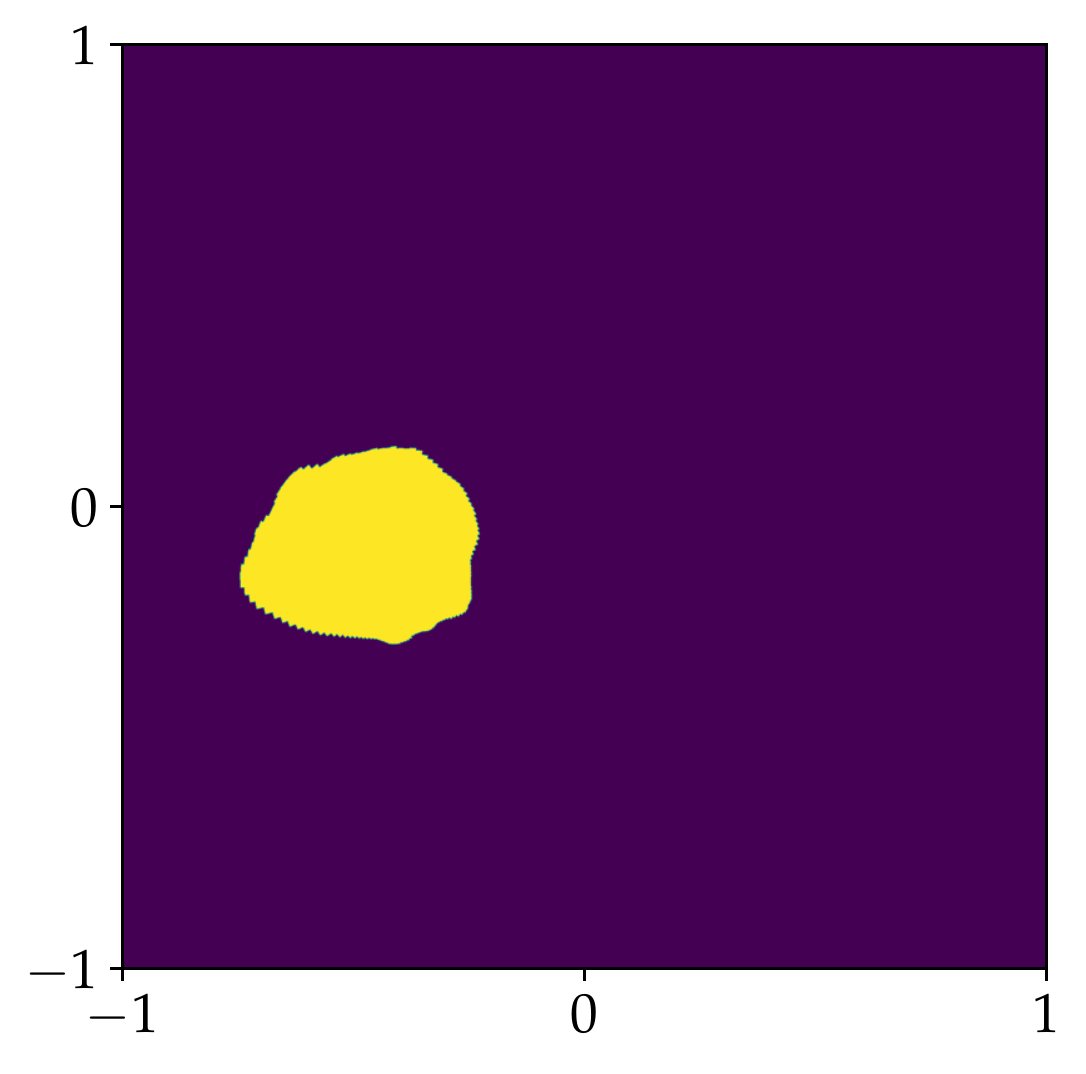}\\
\centering{(a) }
\end{minipage}
\begin{minipage}[t]{4.2cm}
\includegraphics[totalheight=3.8cm,
keepaspectratio=true]{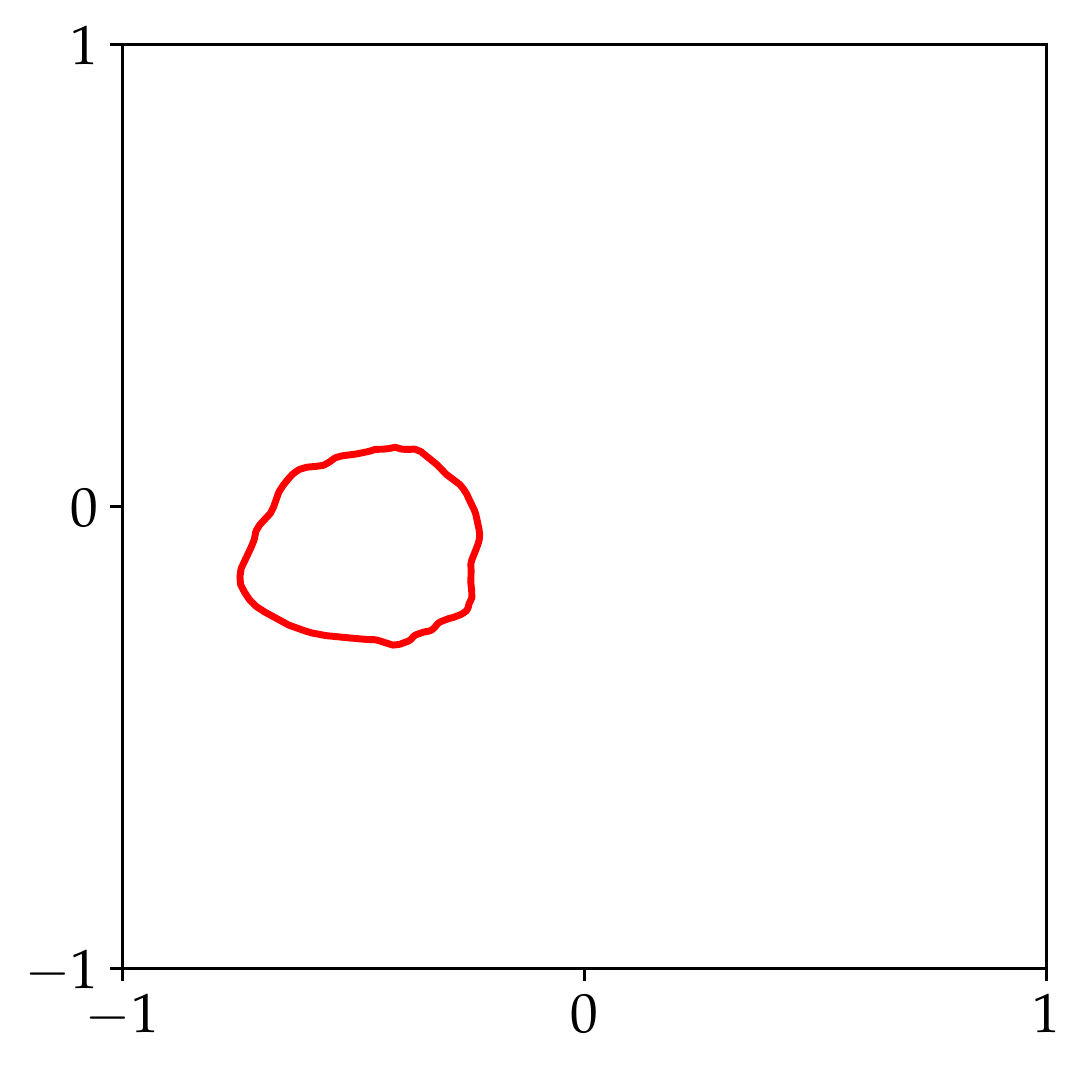}\\
\centering{(b)}
\end{minipage}
\begin{minipage}[t]{4cm}
\includegraphics[totalheight=3.8cm,
keepaspectratio=true]{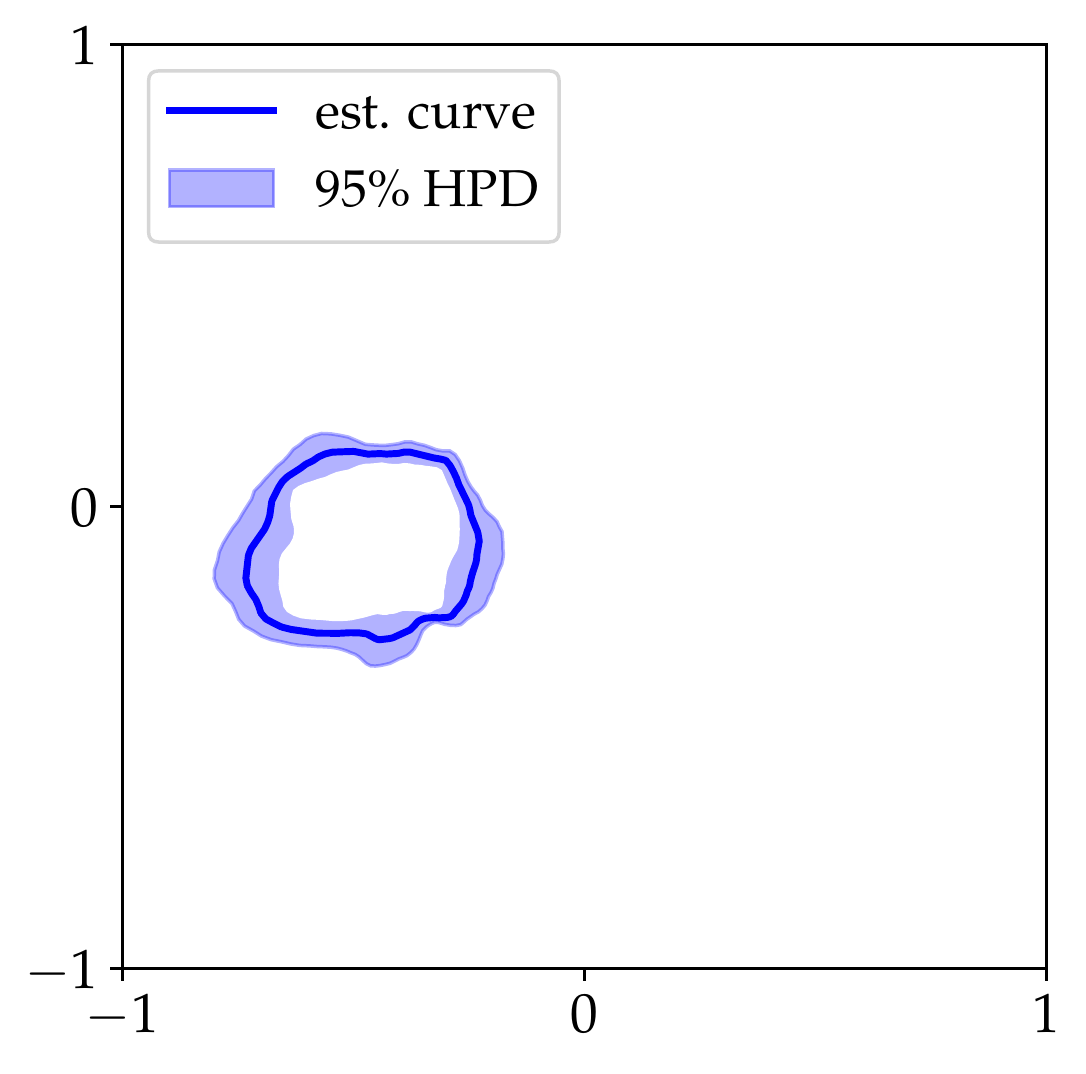}\\
\centering{(c)}
\end{minipage}\\
\begin{minipage}[t]{4cm}
\includegraphics[totalheight=4cm,trim=2.5cm 0cm 2.7cm 0cm,clip]{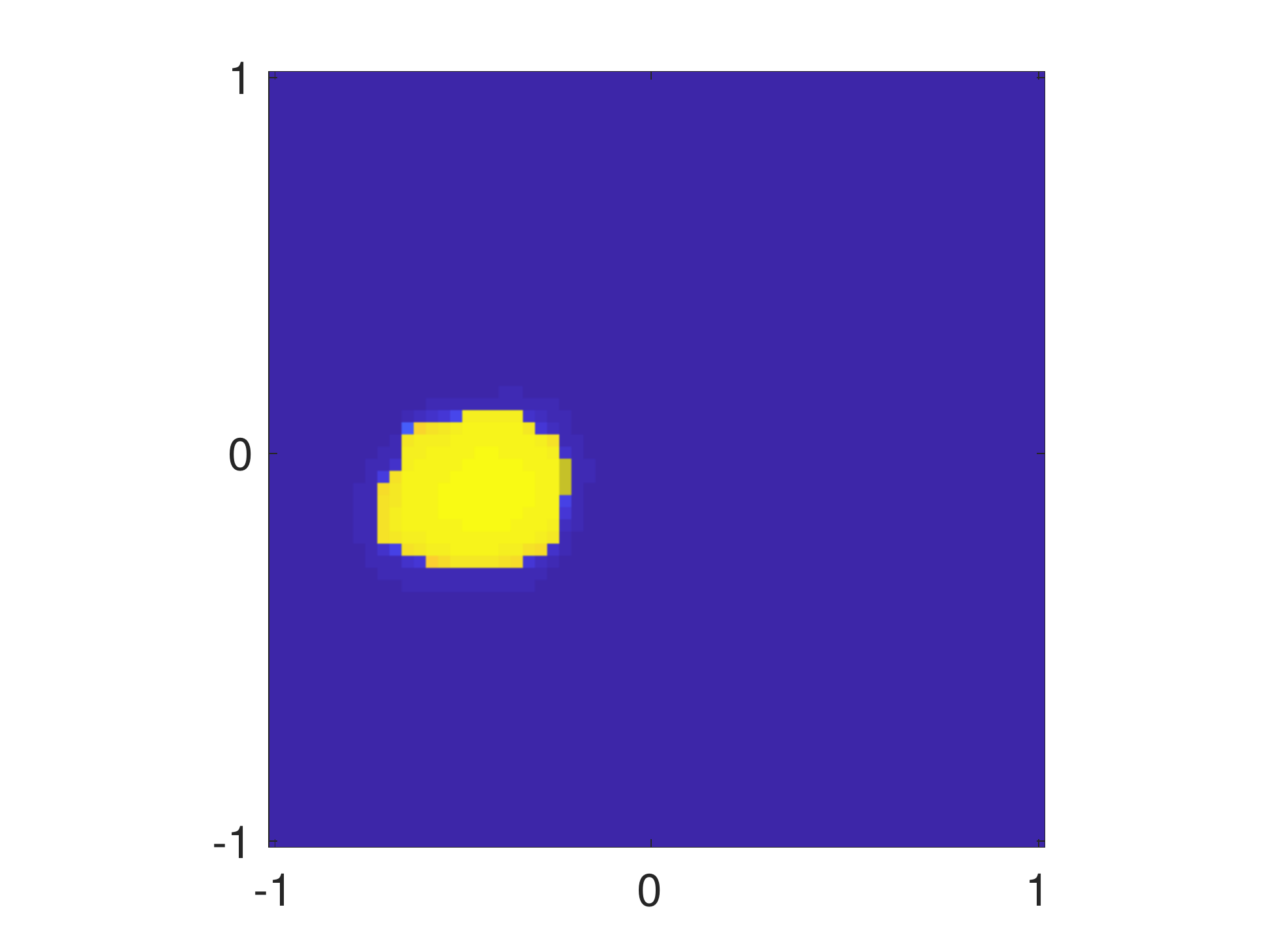}\\
\centering{(d)}
\end{minipage}
\begin{minipage}[t]{4.2cm}
\includegraphics[totalheight=4cm,trim=2.5cm 0cm 2.0cm 0cm,clip]{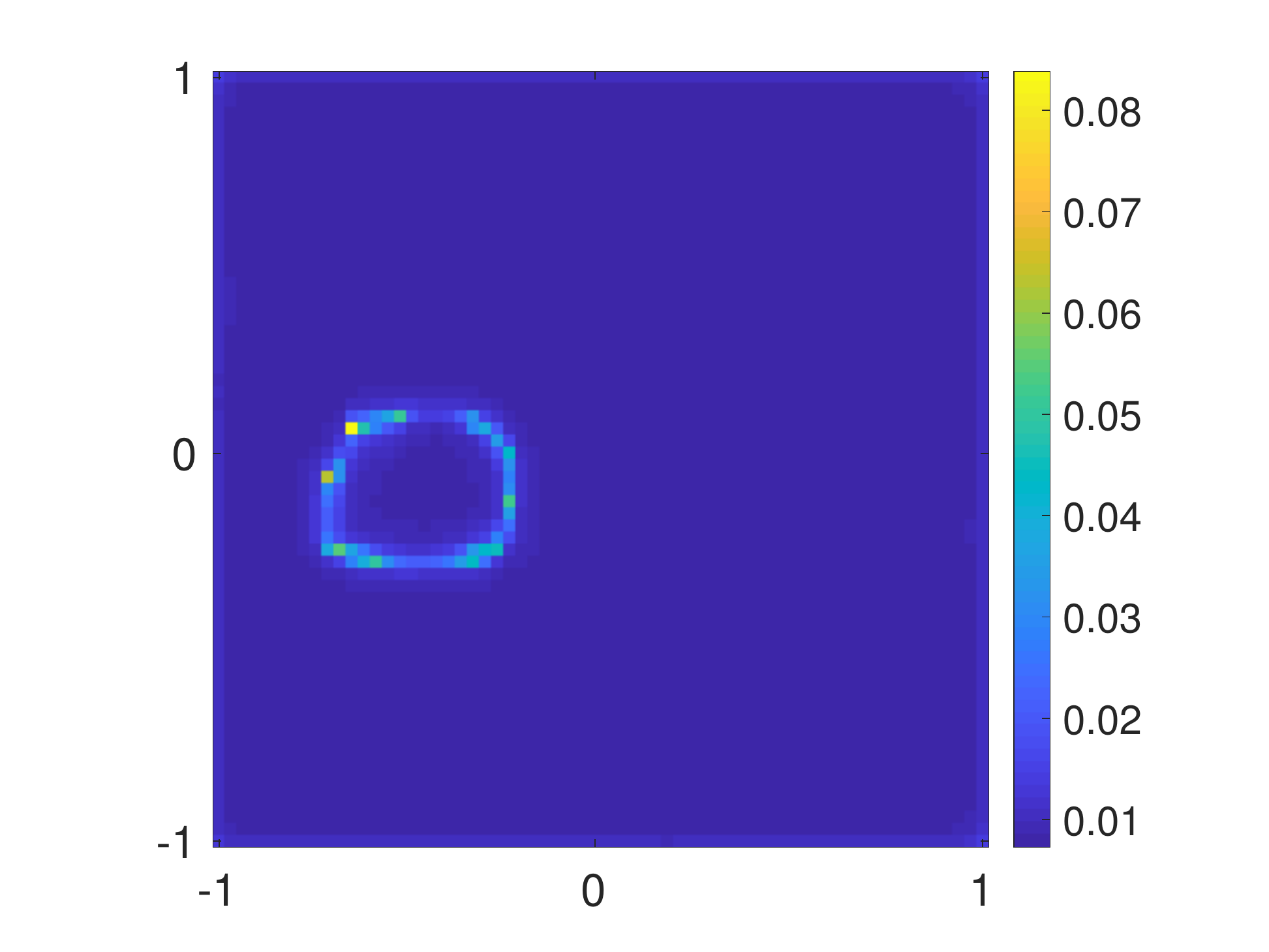} \\
\centering{(e)}
\end{minipage}
\begin{minipage}[t]{4cm}
\includegraphics[totalheight=4cm,trim=2.5cm 0cm 2.7cm 0cm,clip]{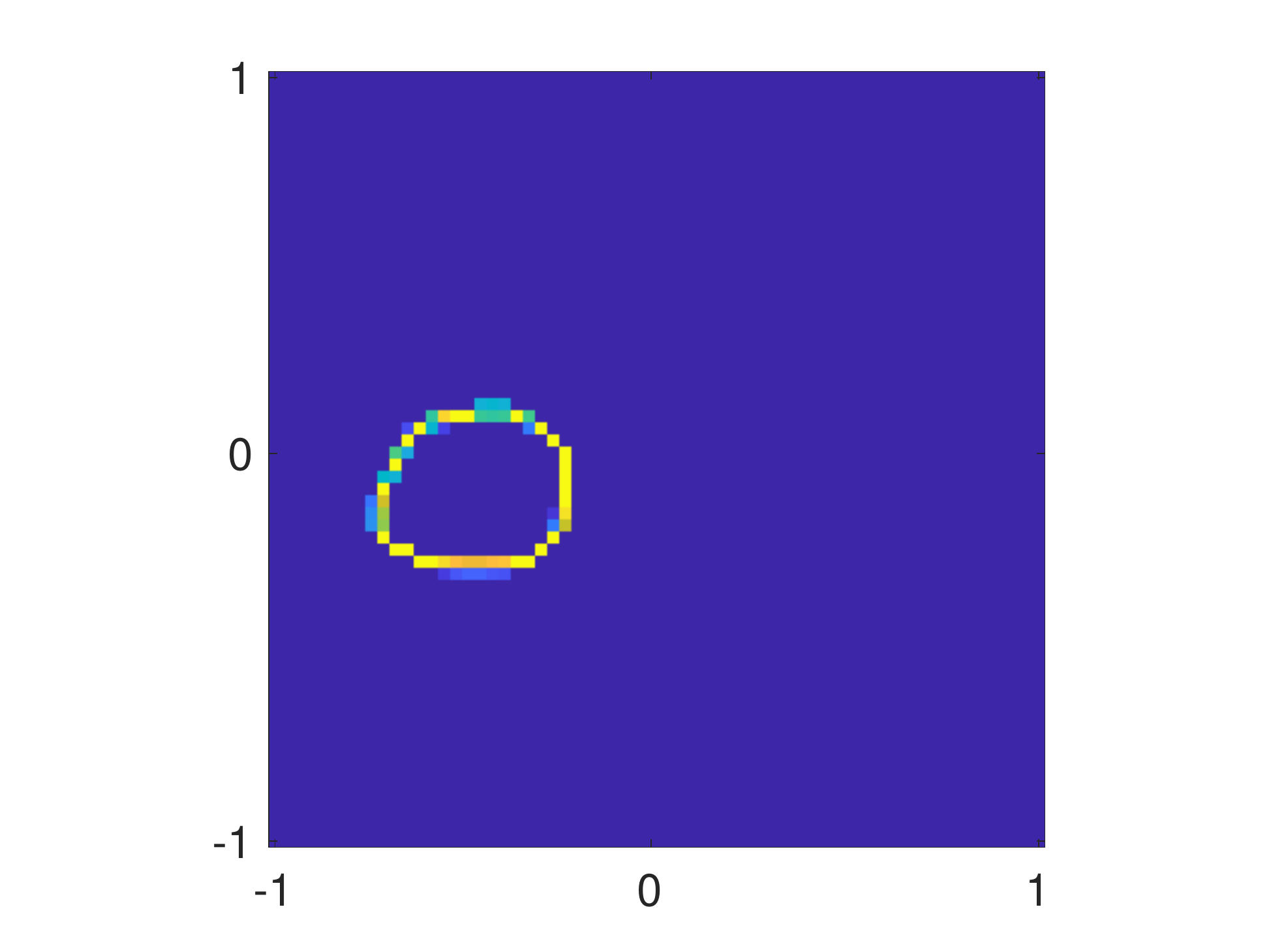}\\
\centering{(f)}
\end{minipage}
    \caption{\coyq{Comparing the performance of our method and the sampling-then-threshholding (STT) method which combines the methods in  \cite{10.5555/1965423,2104.06919}. (a) True attenuation field, (b) true boundary, (c) the estimated boundary by using our method together with the 95\% highest posterior density credibility interval, (d) the mean of the attenuation reconstructions by STT, (e) the standard deviation of the reconstructions by STT, (f) the mean of the boundaries by STT.} }
    \end{center}
\label{fig:1.pixel}
\end{figure}

In this paper, we provide a goal-oriented, infinite-dimensional Bayesian framework for the CT problem. \coyq{In our method, we reconstruct the boundary directly, and without the need for reconstructing the attenuation field. This avoids the error propagation and reduces the dimensionality of the problem from finding a 2D image to a 1D boundary of a region. In addition, we model the boundaries as random functions in order to avoid resolution dependency. Our method consists of two stages: in the stage 1, we approximate locations of all objects in order to maintain the well-posedness of the problem that is only valid for single inclusion; in stage 2, we identify the boundaries of each object, which can be parallelized with respect to the number of inclusions to further improve the efficiency. Our method provides an uncertainty band around the boundary to quantify the reliability in the prediction of the boundaries. We apply the preconditioned Crank-Nicolson (pCN) method \cite{Cotter2013} \correct{as an MCMC method for exploring the posterior}. The pCN ensures that the  reconstruction of the boundaries is independent of the discretization of the random variables. \Cref{fig:1.pixel}(c) shows an example of our method predicting boundaries of the inclusion together with the uncertainty in this prediction. It is \correct{clear} that our result is resolution independent.}

This paper is organized as follows: In \Cref{sec:1} we formulate the \coyq{CT problem} and introduce the Radon transform. The infinite dimensional Bayesian framework \copc{for} inverse problems is presented \coyq{in \Cref{sec:2} and \Cref{sec:2.likefunc}}. \coyq{In \Cref{sec:2},} we introduce the Mat\'ern-Whittle random field and construct the log-Gaussian, level set and the star-shaped priors. \coyq{In \Cref{sec:2.likefunc}, we} \correct{construct a likelihood function} for the CT problem, introduce the posterior, and discuss the \correct{the existence and the} well-posedness of the posterior distribution. In \Cref{sec:3} we introduce \coyq{our two-stage method to identify} inclusion boundaries. We evaluate the performance of the method for simulated images in \Cref{sec:4} for single and multiple inclusions, as well as in sparse and limited angle imaging settings. In addition, we show the performance of \coyq{our} method on a real CT-scan of a tomographic X-ray data of a lotus root slice filled with different chemical elements \cite{bubba_tatiana_a_2016_1254204}. We present conclusive remarks in \Cref{sec:5}.

\section{Radon Transform and Forward Model} \label{sec:1}

Tomography, or slice/volume imaging, comprises methods in reconstructing the internal structure of an object from external measurements. In X-ray computed tomography,
\copc{the X-rays interact with the matter in the object, and we measure the attenuation of
the rays on detectors placed on the other side of the object} \cite{Smith2009}.

A common approach in modeling X-ray interaction with an object is described by line integrals over a density field. Let $D\subset \mathbb R^2$ be a bounded region with Lipschitz boundary.
\copc{Let $0 < \alpha \in L^2(D)$ denote} a density field defined on $D$. This field represents the attenuation of the X-rays \cite{Smith2009}. A measurement can then be described by the line integral
\begin{equation} \label{eq:1.line_integ}
    R_{\theta,s}[\alpha] := \int_{L_{\theta,s}} \coyq{\alpha(x)}\ dl.
\end{equation}
Here, $s\in[-1,1]$ parameterizes points on the line $L^\perp_{\theta}$ passing through the origin and making the angle $\theta\in[0,\theta_{\text{max}})$, $\theta_{\text{max}}\leq \pi$, with the $x$-axis. Furthermore, $L_{\theta,s}$ is the perpendicular line to $L^\perp_{\theta}$ crossing it at point $s$, and $dl$ is an infinitesimally small length on $L_{\theta,s}$. We assume that $\alpha$ has a compact support such that the integration in \eqref{eq:1.line_integ} is carried over a finite section of $L_{\theta,s}$.

 Note that $R_{\cdot, \cdot}$ can be viewed as a linear functional from lines in $\mathbb R^2$ to $\mathbb R$ in which case it is referred to as the \emph{Radon transform} \cite{Hansen2021,natterer2001mathematics,Smith:85}. The simplest set of measurement lines \delete{to carry line integral \eqref{eq:1.line_integ}} are a finite set of equidistant parallel lines perpendicular to $L^\perp_\theta${\copc, for a finite set of angles $\theta$}. This setup is referred to as the \emph{parallel-beam geometry}. The collected measurements of this type is commonly referred to as a \emph{sinogram}.



We discretize $\theta$ into $N_{\theta}$ angles and $s$ into $N_{s}$ detector pixels to construct $N := N_{\theta}\times N_{s}$ measurements. We define an observation vector $y(\alpha) \in \mathbb R^N$ to contain elements of noisy measurement of type
\begin{equation}\label{eq:1.forward}
    y_{i,j}(\alpha) = R_{\theta_i,s_j} \copc{[\alpha]} + \varepsilon_{i,j}, \qquad i = 1,\dots, N_{\theta}, ~ j=1,\dots,N_s,
\end{equation}
where $\varepsilon_{i,j}$ are independent and identically distributed (i.i.d.) following $\mathcal N(0,\sigma_{n}^2)$, for some $\sigma_{\copc{\mathrm{noise}}}>0$. Note that, in general, noise in CT data is non-Gaussian. The choice of observation noise in \eqref{eq:1.forward} is an approximation. Representing \eqref{eq:1.forward} in vector form we obtain
\begin{equation} \label{equation:1.forward_matrix}
    \coyq{\boldsymbol{y}(\alpha) = \coyq{G\alpha} + \boldsymbol{\varepsilon}}.
\end{equation}
Here, $G: L^2(D) \to \mathbb R^N$ contains line-integrals $R_{\theta_j,s_j}$ referred to as the \emph{forward map} and $\coyq{\boldsymbol{\varepsilon}} \sim \mathcal N(\mathbf 0, \copc{\Sigma})$, with $\Sigma = \text{diag}(\sigma_{\mathrm{noise}}^2,\dots,\sigma_{\mathrm{noise}}^2)$. Note that since $R_{\theta_j,s_j}$ is linear then $G$ is locally Lipschitz. The task of finding $\alpha$ from \coyq{$\boldsymbol{y}(\alpha)$} forms an inverse problem. This inverse problem is severely under-determined since $\alpha$ is infinite dimensional and \coyq{$\boldsymbol{y}$} is finite dimensional.

\correct{In many tomography applications we are interested in finding the discontinuities of $\alpha$ and not reconstructing the entire $\alpha$. In this paper, we take a \emph{goal-oriented} approach in finding and quantifying the uncertainty of discontinuities in $\alpha$. Such discontinuities can be parameterized with a significantly lower dimensionality than $\alpha$.}

%

In this \coyq{paper}, we take a Bayesian approach in representing the tomography inverse problem. Severe under-determination of the inverse problem is regularized by a prior knowledge/assumption on $\alpha$ in the form of a probability distribution, \coyq{which is introduced in \Cref{sec:2}}. \coyq{In \Cref{sec:2.likefunc}, we apply the Bayes' theorem to combine the prior distribution with the forward map $G$, then derive a novel posterior distribution which represents the solution to the inverse problem.}

\section{\coyq{Prior Models}} \label{sec:2}
In this section we introduce various prior modelling techniques to construct a density field $\alpha$, \coyq{which is inspired by the works in \cite{Dashti2017,Dunlop2016,1930-8337_2016_4_1007,stuart_2010}. We first review how to construct random functions on a Hilbert space by Gaussian random variables, then present two prior models used in our method.}


\subsection{\coyq{Review of Gaussian Random Variables}} \label{sec:2.prior}

Let $(\Omega, \mathcal A, \mathbb P)$ be a complete probability space with $\Omega$ a measurable space, $\mathcal A$ a $\sigma$-algebra defined on $\Omega$ and $\mathbb P$ a probability measure. Furthermore let $(H,\langle \cdot , \cdot \rangle$) be a Hilbert space. See \cite{nla.cat-vn2458246} for further information on a probability space. A random element $\xi \in H$ is called an $H$-valued Gaussian random variable if $\langle \xi , \zeta \rangle$, for all $\zeta \in H$, is a real-valued Gaussian random variable, i.e., measurable function on the probability space $(\Omega,\mathcal A, \mathbb P)$.

\begin{theorem} \label{thm:2.gaussian}
\cite{nla.cat-vn2458246} Let $\xi$ be an $H$-valued Gaussian random variable. Then we can find $m\in H$, \copc{the} \emph{mean function}, and a symmetric, non-negative and trace-class linear operator $Q:H\to H$, \copc{the} \emph{covariance operator}, such that
\begin{subequations}
\begin{align}
    &\langle m, \zeta \rangle = \mathbb E \langle \xi, \zeta  \rangle, \qquad &\forall \zeta \in H, \\
    &\langle Q\zeta, \eta \rangle = \mathbb E \langle\xi - m,\zeta \rangle \langle \xi - m, \eta \rangle, &\forall \zeta,\eta \in H,
\end{align}
\end{subequations}
where $\mathbb E$ represents expectation. We write $\xi \sim \mathcal N(m,Q)$ and say $\xi$ is Gaussian with measure $\mathbb P \circ \xi^{-1}$.
\end{theorem}

\begin{theorem} \label{thm:2.kl}
\cite{grimmett2001probability} Let $m\in H$ and $Q:H\to H$ be a symmetric, non-negative and trace-class linear covariance operator. Furthermore let $\{ e_i \}_{i\in \mathbb N}$ be eigenfunctions of $Q$ and $\{ \lambda_i \}_{i\in \mathbb N}$ be eigenvalues of $Q$ sorted in decreasing order. An $H$-valued random variable $\xi$ is a Gaussian random variable with the distribution $\mathcal N(m,Q)$ if and only if
\begin{equation}\label{eq:2.KL}
    \xi = m + \sum_{i\in \mathbb N} \sqrt{\lambda_i} \beta_i e_i,
\end{equation}
where $\{\beta_i\}_{i\in \mathbb N}$ is a sequence of i.i.d. real-valued random variables with distribution $\mathcal N(0,1)$. The expansion in \eqref{eq:2.KL} is referred to as the Karhunen-Lo\'eve (KL) expansion of $\xi$.
\end{theorem}

\begin{remark}
In case all random variables are Gaussian random variables, we can construct a probability measure on $H$ using push-forward notation $\mu_0 = \xi_{\star}(\mathbb P)$. Therefore, $(H,\mathcal B(H),\mu_0) )$, with $\mathcal B(H) $ the Borel $\sigma$-algebra, forms a probability space. We refer to $\mu_0$ as the \emph{prior measure}.
\end{remark}

\Cref{thm:2.kl} indicates how we can construct random functions on a Hilbert space $H$ using a covariance operator $Q$. We can approximate the sum in \eqref{eq:2.KL} by truncating it after $N_{\text{KL}}$ terms. We refer the reader to \cite{nla.cat-vn2458246,grimmett2001probability} for convergence properties of the KL expansion.

In the following subsection we introduce a family of covariance operators that we can use to construct Gaussian random functions according to \eqref{eq:2.KL}.

\subsubsection{Mat\'ern-Whittle Covariance} \label{sec:2.matern}
A widely used family of covariance operators are the Mat\'ern-Whittle covariance operators \cite{nla.cat-vn2458246,rasmussen2003gaussian,Roininen2014,zhang2004inconsistent} which allows control over regularity, amplitude and correlation length of samples. Here we briefly introduce this covariance operator but we refer the reader to \Cref{sec:a.1} for a deeper discussion.

A simplified Mat\'ern-Whittle covariance operator \cite{Roininen2014} is given by
\begin{equation} \label{eq:2.cov_simple}
    Q_{\gamma,\tau} = (\tau^2 I- \Delta)^{-\gamma}.
\end{equation}
Here, \correct{$\Delta$ and $I$ are the Laplacian and the identity operators in 1 or 2 dimensions, respectively. Furthermore,} $\tau = 1/\ell >0$ controls the correlation length and $\gamma = \nu + 1$ is the smoothness parameter (see \cite{Lindgren2011} for more detail). For the covariance operator \eqref{eq:2.cov_simple} to be well defined we need to impose proper boundary conditions. See \cite{Roininen2014} for more detail on types of boundary conditions.

We note that in one dimension, i.e., when $D\subset \mathbb R$, and for $\tau=1$ the zero-mean Gaussian random variable $\xi$ distributed according to the covariance $Q_{\gamma,1}$ takes the form
\begin{equation} \label{eq:2.kl_1d}
    \xi(x) = c \sum_{i\in \mathbb N} \left( \frac{1}{k} \right)^\gamma \left( \xi_i^1 \sin(kx) + \xi_i^2 \cos(kx)  \right),
\end{equation}
for some constant $c>0$. Here, $\xi_i^1,\xi_i^2\sim \mathcal N(0,1)$ are real Gaussian random variables. This type of Gaussian random variables will be used in later sections to model the boundaries of inclusions.

\subsection{\correct{Prior as a Push-Forward Measure}} \label{sec:2.likelihood}
In this section we apply a nonlinear transformation on Gaussian random variables to \coyq{represent the image $\alpha$}. \coyq{Assume that $\alpha$ is piecewise constant, and we use the level set \cite{Dunlop2016} and the star-shaped \cite{1930-8337_2016_4_1007} fields to construct it.} Later, in \Cref{sec:2.likefunc}, we discuss how to use these fields to construct a likelihood function.
\subsubsection{Level Set Parameterization} \label{sec:2.levelset}
Let $\xi$ be an $H$-valued Gaussian random variable and take $c\in \mathbb R$. Define $D^{\text{ls}}_{i}\subset D$, for $i=1,2$ as
\begin{equation} \label{eq:2.level_set}
    D^{\text{ls}}_1 := \{ x\in D \mid \xi(x) < 0  \}, \qquad D^{\text{ls}}_2 := D \backslash D^{\text{ls}}_1.
\end{equation}
We define the level set mapping $F_{\text{ls}}:H \to L^p(D)$, with $2 \leq p < \infty$, as
\begin{equation} \label{eq:2.ls_prior}
    F_{\text{ls}}[\xi](x) = a^- \mathbbm{1}_{D^{\text{ls}}_1}(x) + a^+ \mathbbm{1}_{D^{\text{ls}}_2}(x),
\end{equation}
where $\mathbbm{1}_{D^{\text{ls}}_i}(x)$, for $i=1,2$, is the indicator function and $ 0 <a^- < a^+$. We define $D^0 = \overline D^{\text{ls}}_1 \cap \overline D^{\text{ls}}_2$ which contains the points of discontinuity. Throughout this paper we assume that $m(D^0) = 0$ where $m$ is the Lebesgue measure defined on $\overline D$. This assumption is to ensure that the boundary of the inclusions has indeed a lower dimensionality than the image $\alpha$. We refer the reader to \cite{iglesias2015bayesian} for more detail.
\begin{remark}
    It is shown in \cite{iglesias2015bayesian} that the assumption $m(D^0) = 0$ is sufficient for $F_{\text{ls}}$ to be continuous. This means if $\{ \xi^{\epsilon} \}_{\epsilon>0}$ is a sequence of functions such that for any $x\in D$
    \begin{equation}
        \xi(x) - \epsilon \leq \xi^{\epsilon} \leq \xi(x) + \epsilon,
    \end{equation}
    then $\| F_{\text{ls}}[\xi^{\epsilon}] - F_{\text{ls}}[\xi] \|_{L^p(D)} \to 0$, $\mu_0$-almost surely.
\end{remark}

\subsubsection{Star-Shaped Parameterization} \label{sec:2.star}
\correct{Let $\mathbf T = [0,2\pi)$ and $H \subset L^2\left( \mathbf T \right)$ contain period functions on $\mathbf T$.} Furthermore, define $\vartheta: D \to [0,2\pi)$ be the continuous map from Cartesian coordinates to the angular component of polar coordinates. Define \emph{star-shaped inclusions} $D_{i}\subset D$, for $i=1,\dots,N_{\text{inc}}$, as
\begin{equation}
    D_{i}(\xi_i,c_i) := \{ x\in D \mid \|x - c_i\|_2 < \coyq{F_{\text{lg}}} \left( \xi_i( \vartheta(x - c_i) ) \right) \},
\end{equation}
where
\begin{equation} \label{eq:2.log_gaussian}
   \coyq{F_{\text{lg}}[\xi]} = \exp(\xi)
\end{equation}
defines the \emph{log-Gaussian} field $\coyq{F_{\text{lg}}}:H\to L^2(D)$ in order to construct an image with positive attenuation. Here $\xi$ is an $H$-valued Gaussian random variable. Since $D$ is bounded, then \coyq{$F_{\text{lg}}$} is almost surely a continuous map. In addition, $\|\cdot \|_2$ is the Euclidean norm and $c_i\in D$ for $i=1,\dots,N_{\text{inc}}$ are independent stochastic centers of inclusions. Note that $c_i$ and $\xi_i$ are not necessarily i.i.d.
Let $D_0 := (D_1\cup \dots \cup D_{N_{\text{inc}}})^C$, \correct{where the superscript $C$ represents the complement of a set}. We define the star-shaped mapping $F_{\text{star}}:H^{N_{\text{inc}}}\times \mathbb R^{2N_{\text{inc} }} \to L^p(D)$ with $2 \leq p < \infty$ as
\begin{equation} \label{eq:2.star_shape}
    \coyq{F_{\text{star}}[ (\boldsymbol{\xi},\boldsymbol{c})](x)} =  a^{-} \mathbbm{1}_{D_0}(x) +  \sum_{i=1}^{N_{\text{inc}}} a^{+} \mathbbm{1}_{D_i}(x),
\end{equation}
where \coyq{$\boldsymbol{\xi}=[\xi_{1},\cdots,\xi_{N_{\text{inc}}}]^{T}$, $\boldsymbol{c}=[c_{1},\cdots,c_{N_{\text{inc}}}]^{T}$,} $a^-<a^+ \in \mathbb R^+$ and we refer to $a^+$ as the inclusion intensity. The following assumptions are considered when drawing samples of $\xi_i$ and $c_i$:
\begin{enumerate}[(I)]
    \item $D_i$, for $i=1,\dots, N_{\text{inc}}$ are disjoint.
    \item \coyq{$\| x_{1} - x_{2} \|_2 > d^{\partial D}_{\text{min}}$, for all $x_{1}\in D_1 \cup \dots\cup D_{N_{\text{inc}}}$ and $x_{2}\in \partial D$ the boundary of $D$. This insures all inclusions are away from the boundary of $D$.}
    \item \coyq{$\| x_{1} - x_{2} \|_2 > d^{D}_{\text{min}}$, for all $x_{1}\in D_i$ and $x_{2}\in D_j$ for $i\neq j$ and $i,j\neq 0$. This insures that the inclusions are well separated.}
\end{enumerate}
\begin{remark} \label{remark:2.star_shaped}
It is shown in \cite{1930-8337_2016_4_1007} that when $N_{\text{inc}}=1$ and \coyq{$\xi$} is Lipschitz continuous (e.g., $\gamma > 1$ in \eqref{eq:2.cov_simple}) then $F_{\text{star}}$ is a continuous mapping. This means if \coyq{$\{ \xi^{\epsilon} \}_{\epsilon>0}$ and $\{ c^{\epsilon} \}_{\epsilon>0}$ are sequences of $H$-valued random variables and sequence of points in $D$, respectively, such that $\|\xi - \xi^\epsilon \|_\infty \to 0$ and $\| c^{\epsilon} - c \|_2 \to 0$, then $F_{\text{star}}[ (\xi^\epsilon,c^\epsilon)] \to F_{\text{star}}[ (\xi,c) ]$ in measure.} When $N_{\text{inc}} >1$, assumptions (I)-(III) ensures that we can divide $D$ into subregions with only single inclusions.
\end{remark}

\begin{remark}
The background attenuation $a^{-} \mathbbm{1}_{D_0}$ in \eqref{eq:2.star_shape} is constant. However, in many applications $a^-$ varies smoothly in the domain. To account this variation, we modify \eqref{eq:2.star_shape} to
\begin{equation} \label{2.eq.noisy_star}
   \coyq{F^{\text{noisy}}_{\text{star}}[\xi_0,(\boldsymbol{\xi},\boldsymbol{c})](x) = F_{\text{lg}}[\xi_0](x) \mathbbm{1}_{D_1}(x) + \sum_{i=1}^{N_{\text{inc}}} a^{+} \mathbbm{1}_{D_i}(x).}
\end{equation}
We distinguish between the noise term $\varepsilon$ in \eqref{equation:1.forward_matrix} and the variation in the background \coyq{$F_{\text{lg}}[\xi_0](x)$} in our model. The latter can accounts for experiment's systematic error while the former noise models the measurement error.
\end{remark}

\coyq{We need to specify a prior measure $\boldsymbol{\mu}_0$ for $(\boldsymbol{\xi},\boldsymbol{c})$. We define $\boldsymbol{\mu}_0 = \boldsymbol{\mu}^1_0 \otimes \boldsymbol{\mu}^2_0 $, where $\boldsymbol{\mu}^1_0$ is a Gaussian measure on $H^{N_{\text{inc}}}$ and $\boldsymbol{\mu}^2_0$ is a measure on $R^{2N_{\text{inc}}}$ (e.g., Lebesgue measure). We assume that $\xi_i$ and $c_i$ for $i=1,\dots,N_{\text{inc}}$ are independent, then $\boldsymbol{\mu}^1_0$ and $\boldsymbol{\mu}^2_0$ can be factorized further into simpler measures as $\boldsymbol{\mu}^1_0 = \mu^{1,1}_0 \otimes\cdots\otimes \mu^{1,N_{\text{inc}}}_0$ and $\boldsymbol{\mu}^2_0 = \mu^{2,1}_0 \otimes\cdots\otimes \mu^{2,N_{\text{inc}}}_0$. For each $i=1,\cdots,N_{inc}$ the random variable $(\xi_i,c_i)$ is in the separable Hilbert space $X^{i}:=H^{i}\times\mathbb{R}^{2}$ equipped with the norm $ \| (\xi_i,c_i) \|_{X^{i}} := \| \xi_i \|_{H^{i}} + \| c_i \|_2$, and accordingly we define $(\boldsymbol{\xi},\boldsymbol{c})\in X$. Note that the superscript $i$ in $H^i$ indicates the index of the space corresponding to the random variable $\xi_{i}$. This should not be confused with the differentiability order of the Hilbert space.}

\section{\coyq{The posterior distribution for the CT problem}}  \label{sec:2.likefunc}
\coyq{In this section we derive a novel posterior distribution for the CT problem. Before doing so, we first construct the likelihood function based on the prior models introduced in \Cref{sec:2}. Then, we show that this likelihood function fits in the Lipschitz-Hellinger well-posedness framework in \cite{Dashti2017,stuart_2010}. Therefore, the posterior measure exists and is unique.}

\subsection{The Likelihood Function}
\coyq{Define the probability space $(X,\mathcal B(X),\boldsymbol{\mu}_0)$} 
\coyq{and the image $\alpha(\xi)=F_{\text{ls}}[\xi]$ or $\alpha(\boldsymbol{\xi},\boldsymbol{c}):=F_{\text{star}}[(\boldsymbol{\xi},\boldsymbol{c})]$. In this section we only construct the likelihood and the posterior based on the star-shaped field. The formulas regarding the level set field follows immediately by setting $c$ as a constant in the single inclusion case. We refer the reader to \cite{1930-8337_2016_4_1007} for a detailed discussion.}

 Recall that the noise \coyq{$\boldsymbol{\varepsilon}$} is distributed according to $\coyq{\boldsymbol{\varepsilon}} \sim \mathcal N(\mathbf 0, \Sigma)$. \correct{We assume that $\boldsymbol{\varepsilon}$ is independent of $(\boldsymbol{\xi},\boldsymbol{c})$}. 
 \coyq{With the star-shape field} the \emph{negative log-likelihood} function \coyq{$\Phi:X\to \mathbb R$} is formulated as the least squares distance
\begin{equation} \label{eq:2.loglike}
    \coyq{\Phi\left((\boldsymbol{\xi},\boldsymbol{c}); \boldsymbol{y}\right) = \frac 1 2 \| \boldsymbol{y} - \mathcal G(\boldsymbol{\xi},\boldsymbol{c}) \|_{\Sigma}^2,}
\end{equation}
where $\mathcal G = G\circ F_{\text{star}}$, and $\|\cdot\|_{\Sigma} = \|\Sigma^{-1/2} \cdot \|_2$. \coyq{For the CT problem with a single inclusion, it satisfies the following conditions.} \correct{We refer the reader to \Cref{sec:a.2} for the proof.}

\coyq{
\begin{proposition} \label{prop:continuous}
Let $(X,\mathcal B(X), \mu_0)$ be the probability space defined above. The negative log-likelihood $\Phi$ defined in \eqref{eq:2.loglike} with a single inclusion, i.e. $N_{\text{inc}}=1$, satisfies the following conditions:
\begin{enumerate}[(i)]
\item There is a continuous function $K:\mathbb R^+\times \mathbb R^+ \to \mathbb R^+$ such that for every $\rho >0$, $(\xi_1,c_1) \in X$, and bounded observation vector $\boldsymbol{y}$ with $\| \boldsymbol{y} \|_{\Sigma}\leq\rho$,
\begin{equation}
    0\leq \Phi((\xi_1,c_1), \boldsymbol{y}) \leq K(\rho, \| (\xi_1,c_1) \|_X ).
\end{equation}
\item For a fixed observation vector $\boldsymbol{y}\in \mathbb R^N$, $\Phi(\cdot; \boldsymbol{y}): X \to \mathbb R$ is continuous $\mu_0$-a.s. on the probability space $(X,\mathcal B(X), \mu_0)$.
\item There exists a continuous map $M:\mathbb R^+ \times \mathbb R^+ \to \mathbb R^+ $ such that for every pair of observation vectors $\boldsymbol{y}_1, \boldsymbol{y}_2\in \mathbb R^N$ with $\| \boldsymbol{y}_1 \|_{\Sigma},\|\boldsymbol{y}_2 \|_{\Sigma} \leq \rho $, and every $(\xi_1,c_1)\in X$,
\begin{equation}
    \left| \Phi((\xi_1,c_1),\boldsymbol{y}_1) - \Phi((\xi_1,c_1), \boldsymbol{y}_2) \right| \leq M(\|(\xi_1,c_1)\|_X,\rho)\| \boldsymbol{y}_1 - \boldsymbol{y}_2 \|_{\Sigma}.
\end{equation}
\end{enumerate}
\end{proposition}
}


\subsection{Posterior Distribution} \label{sec:2.posterior}
In this section we present \delete{the} Bayes' theorem to connect the prior measure with the likelihood function and construct the posterior measure \coyq{$\boldsymbol{\mu}^{\boldsymbol{y}}$.}
\begin{theorem} \cite{Dashti2017}\label{thm42}
\coyq{Let $\Phi$ be the negative log-likelihood defined in \eqref{eq:2.loglike} satisfying \Cref{prop:continuous} and $\boldsymbol{\mu}_0$ be the prior measure defined on $(X^{N_{\text{inc}}},\mathcal B(X^{N_{\text{inc}}}))$. Then there is a posterior measure $\boldsymbol{\mu}^{\boldsymbol{y}}$ absolutely continuous with respect to $\boldsymbol{\mu}_0$, i.e., $\boldsymbol{\mu}^{\boldsymbol{y}} \ll \boldsymbol{\mu}_0$, and \delete{is} defined through the Radon-Nikodym derivative \cite{nla.cat-vn2458246}
\begin{equation} \label{eq:2.posterior}
    \frac{d \boldsymbol{\mu}^{\boldsymbol{y}}}{d\boldsymbol{\mu}^0} = \frac{1}{Z} \exp \left( -\Phi((\boldsymbol{\xi},\boldsymbol{c});\boldsymbol{y}) \right),
\end{equation}
where $Z$ is the normalization constant and for $\boldsymbol{y}$-almost surely
\begin{equation}
    Z:= \int_X \exp \left( -\Phi((\boldsymbol{\xi},\boldsymbol{c});\boldsymbol{y}) \right) \ \boldsymbol{\mu}_0(d(\boldsymbol{\xi},\boldsymbol{c})) >0.
\end{equation}}
\end{theorem}
\correct{We show in \Cref{sec:a.2} that the posterior measure \coyq{$\boldsymbol{\mu}^{\boldsymbol{y}}$ constructed with the star-shaped prior in a single inclusion case} is well-posed. Existence and well-posedness of the posterior measure constructed with level-set prior \coyq{together with} a linear forward map (e.g., the Radon transform $G$) is discussed in \cite{Dunlop2016,stuart_2010}.} \coyq{In the next section we utlize both priors to develop a two-stage method for identifying the location and boundaries of inclusions in an image $\alpha$.}

\section{Two-Stage Method for Detecting Inclusion Boundaries} \label{sec:3}
\coyq{In \Cref{sec:2} we \copc{presented} two prior models that are used in constructing a posterior distribution. \Cref{sec:2.posterior} shows the Bayesian formula for modelling the CT problem.} In this section we provide a two-stage method in detecting the location of inclusions in an image and also estimate the boundary of the inclusions. \coyq{In the first stage we use the level set prior \eqref{eq:2.ls_prior} to construct a posterior and use the mean to obtain the approximated centers of inclusions. In the second stage we use the star-shape prior to estimate the boundaries of the inclusions.}

We remind the reader that the star-shaped prior for multiple inclusions with \coyq{unknown centers} can be used to construct a posterior. However, constructing a sampling method for such a posterior is challenging (e.g., the Metropolis-within-Gibbs-type method can result in highly correlated samples and increased computational cost, see \cite{SAIBABA2021110391}). The first stage of the method is to ensure that we can decompose the images into regions containing a single inclusion. \Cref{remark:2.star_shaped} then guarantees that we can converge to the right solution for each individual inclusion. This decomposition also allows parallel computation with respect to the number of inclusions for exploring the posterior.

\subsection{Stage 1: Estimating Centers of Inclusions} \label{sec:3.stage1}
In this section we construct the posterior measure \coyq{$\mu^{\boldsymbol{y}}$} in \eqref{eq:2.posterior} using the Mat\'ern covariance \eqref{eq:2.cov_simple} and the level set map $F_{\text{ls}}$ introduced in \eqref{eq:2.ls_prior}.

We assume that the correlation-length parameter $\tau$ and the regularity parameter $\gamma$ \coyq{are} known. Once the covariance $Q_{\gamma,\tau}$ is constructed we discretize \eqref{eq:2.cov_simple} using a finite differences (FD) discretization scheme with a pixel size of $2/N_s$, where $N_s$ is the number of detector pixels. \correct{We truncate the KL expansion after $N_{\text{KL}}$ terms. We propose two \coyq{approaches} to draw samples from a truncated KL expansion.}

In the first approach we consider periodic boundary conditions on the box centered at the origin with the size 2-by-2. Note that $Q_{\gamma,\tau}$ is defined via its Fourier transform. Therefore, this choice of boundary conditions allows efficient sampling from $Q_{\gamma,\tau}$ using the fast Fourier transform. \correct{For a sample $\xi \sim Q_{\gamma,\tau}$ we first compute the Fourier transform $\widehat \xi$ of $\xi$, see \Cref{sec:a.1}.} We then use the inverse fast Fourier transform (IFFT) to obtain $\xi_{\text{ex}}$. To Ensure that $\xi$ is only valid in $D$ we define $\xi := \xi_{\text{ex}}|_{D}$. The image $\alpha$ is then constructed using the map $F_{\text{ls}}[\xi]$ for fixed $a^+,a^-\geq 0$. This method allows efficient generation of Gaussian samples.

\correct{An alternative sampling approach is to discretize $Q_{\gamma,\tau}$, e.g. using FD methods, to obtain a covariance matrix. We can construct efficient sampling methods by computing the Cholesky factor or the principle square root \cite{banerjee2014linear} of this covariance matrix. Note that this computation is carried only once for a set of parameters $\tau$ and $\gamma$. The image $\alpha$ is then construct by directly computing the $KL$ expansion in \eqref{eq:2.KL} and using the map $F_{\text{ls}}[\xi]$ for fixed $a^+,a^-\geq 0$. In our numerical experiments we only used the first approach.}


We construct the posterior measure \coyq{$\mu^{\boldsymbol{y}}$} and explore it by using the preconditioned Crank-Nicolson (pCN) method \cite{Cotter2013}, which is introduced in \Cref{sec:2.mcmc} in detail. The algorithm of Stage 1 is given in \Cref{alg:3.stage_1}.

\begin{algorithm}[th]
\caption{Detecting inclusion centers}
\label{alg:3.stage_1}
\begin{algorithmic}[1]
\STATE{Construct the posterior measure  \coyq{$\mu^{\boldsymbol{y}}$} using the level set prior for an observation vector \coyq{$\boldsymbol{y}$}.}
\STATE{Draw samples $\{\xi^{(j)} \}_{j=1}^{N_{\text{sample}}}$ by using pCN with \Cref{alg:2.pCN}.}
\STATE Compute $\overline \alpha = F_{\text{ls}} [\mathbb E \xi]$.
\STATE \coyq{Estimate $N_{\text{inc}}$, $\boldsymbol{c}$ and bounding boxes $\square^{i}$ with $i=1,\cdots,N_{\mathrm{inc}}$ from $\overline \alpha$.}
\end{algorithmic}
\end{algorithm}

To show the performance of our method in Stage 1, we consider two types of mean for the density field suggested in \cite{1930-8337_2016_4_1007}. The first mean is computed on the space $H$ and then push-forward to the density space using the map $F_{\text{ls}}$, i.e. $\overline \alpha = F_{\text{ls}}[ \mathbb E \xi ]$. This mean conserves the piecewise constant nature of the density field. The other type is the sample mean in the density field, i.e., $\widehat \alpha = \mathbb E F_{\text{ls}}[\xi]$. This mean does not construct a piece-wise constant nature of the density field but provides an \coyq{uncertainty estimate on the boundaries} of the inclusions.

We use $\overline \alpha = F_{\text{ls}}[ \mathbb E \xi ]$ to estimate the approximate location of the inclusions. We can use standard matrix/image segmentation methods (e.g., \cite{Weaver1985}) to identify individual inclusions. This also gives us an estimate of the number of existing inclusions. The center of mass for each inclusion is an approximation of the center for the star-shape inclusion.

\begin{remark}
The center defined in the star-shaped inclusions is, in general, not the center of mass. However, the center is \copc{just} a modelling tool to describe inclusions and
\copc{it is not explicitly needed} in most applications.
\end{remark}

\begin{figure}[tbhp]
    \centering
    \subfloat[True density $\alpha(\boldsymbol{\xi})$]{\label{fig:3.stage1.a}\includegraphics[width=0.3\textwidth]{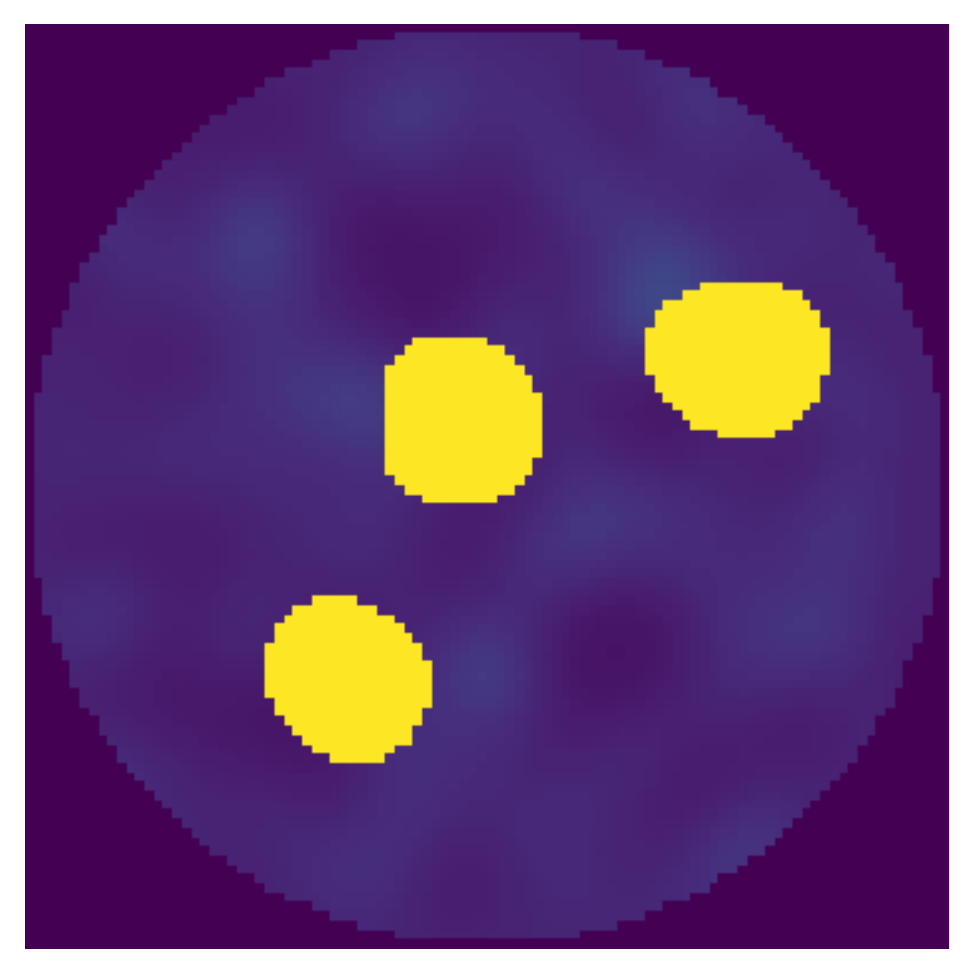}}
    \subfloat[$F_{\text{ls}}(\mathbb E \boldsymbol{\xi})$]{\label{fig:3.stage1.b}\includegraphics[width=0.3\textwidth]{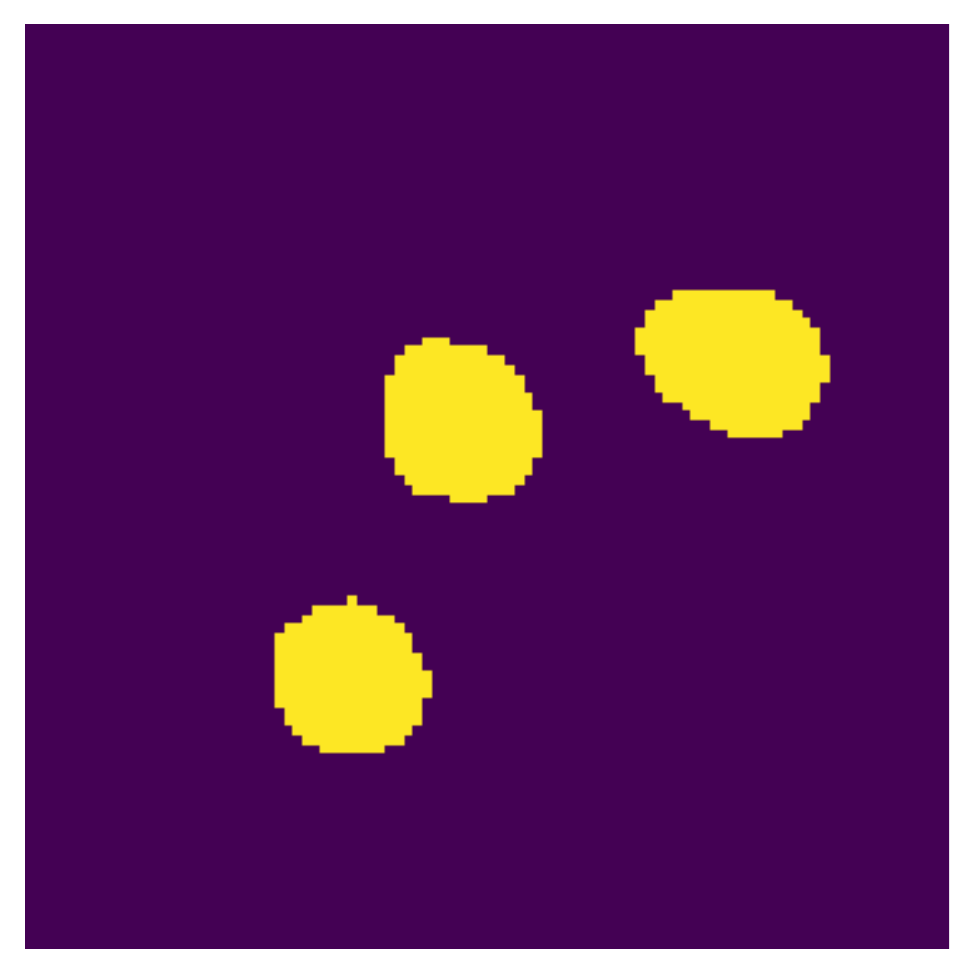}}
    \subfloat[$\mathbb E F_{\text{ls}}(\boldsymbol{\xi})$]{\label{fig:3.stage1.c}\includegraphics[width=0.3\textwidth]{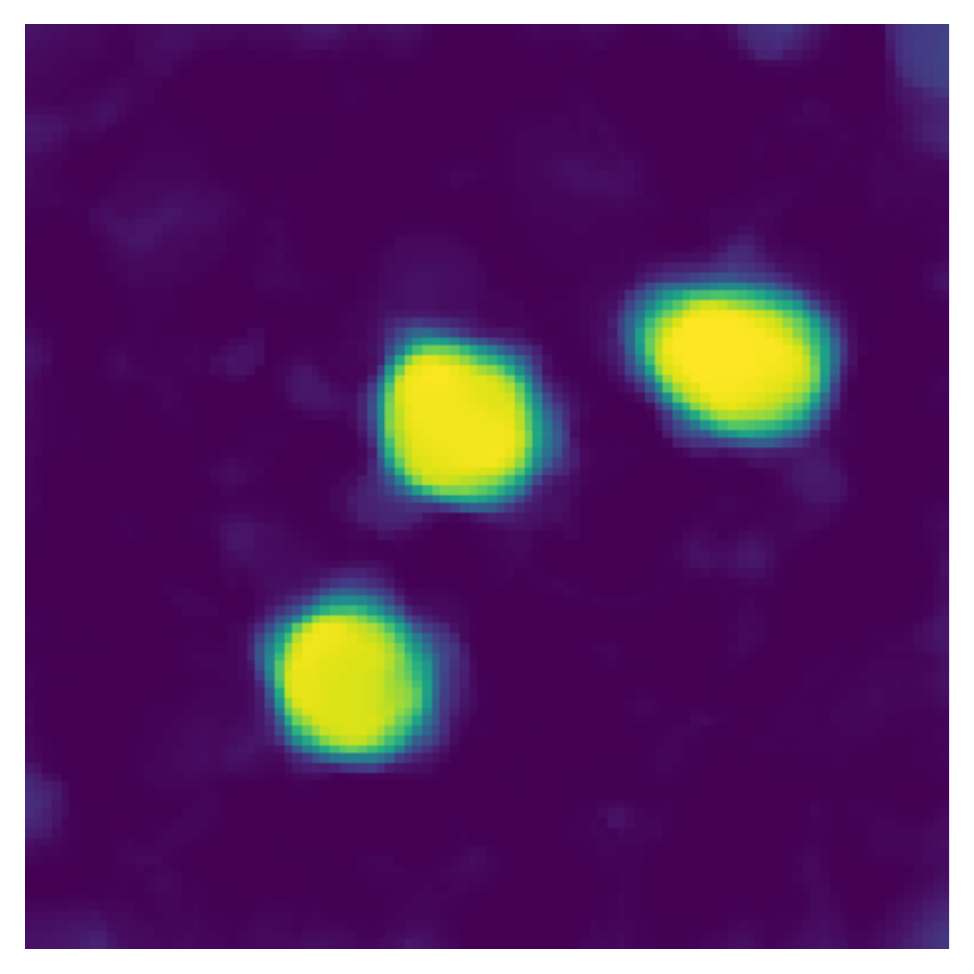}}
    \\
    \subfloat[True density $\alpha(\boldsymbol{\xi})$]{\label{fig:3.stage1.d}\includegraphics[width=0.3\textwidth]{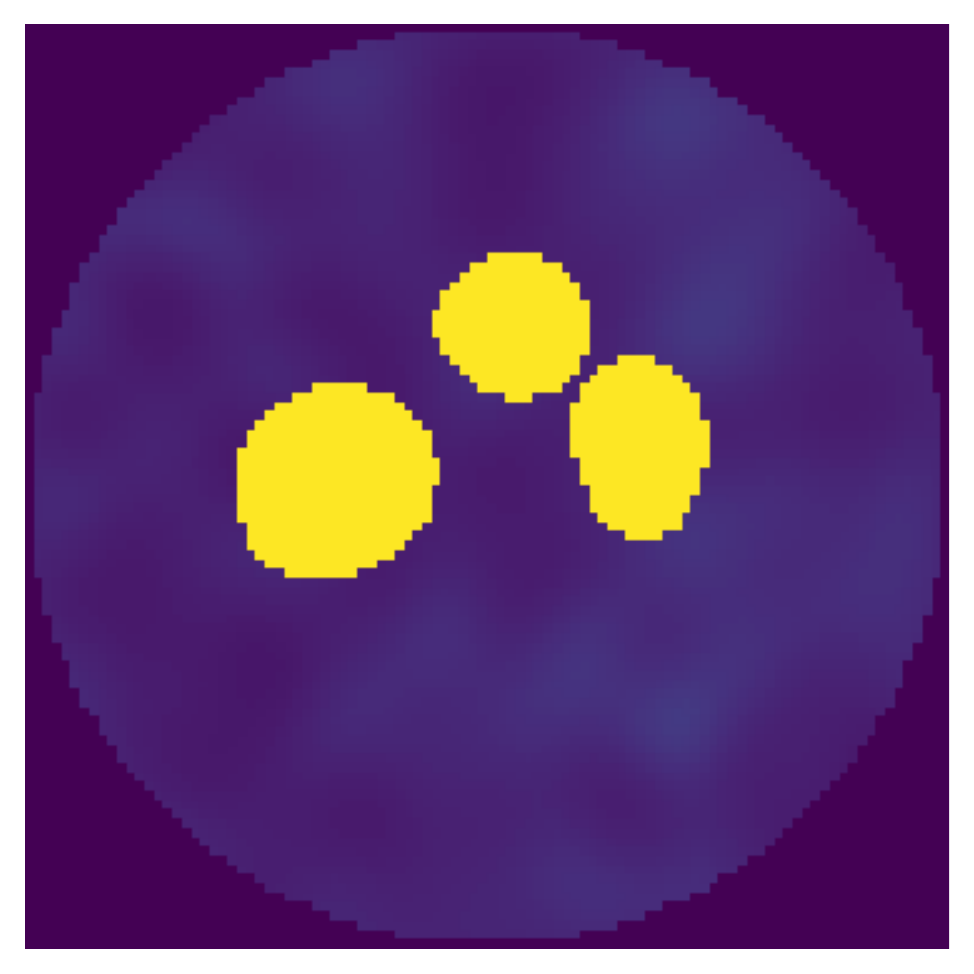}}
    \subfloat[$F_{\text{ls}}(\mathbb E \boldsymbol{\xi})$]{\label{fig:3.stage1.e}\includegraphics[width=0.3\textwidth]{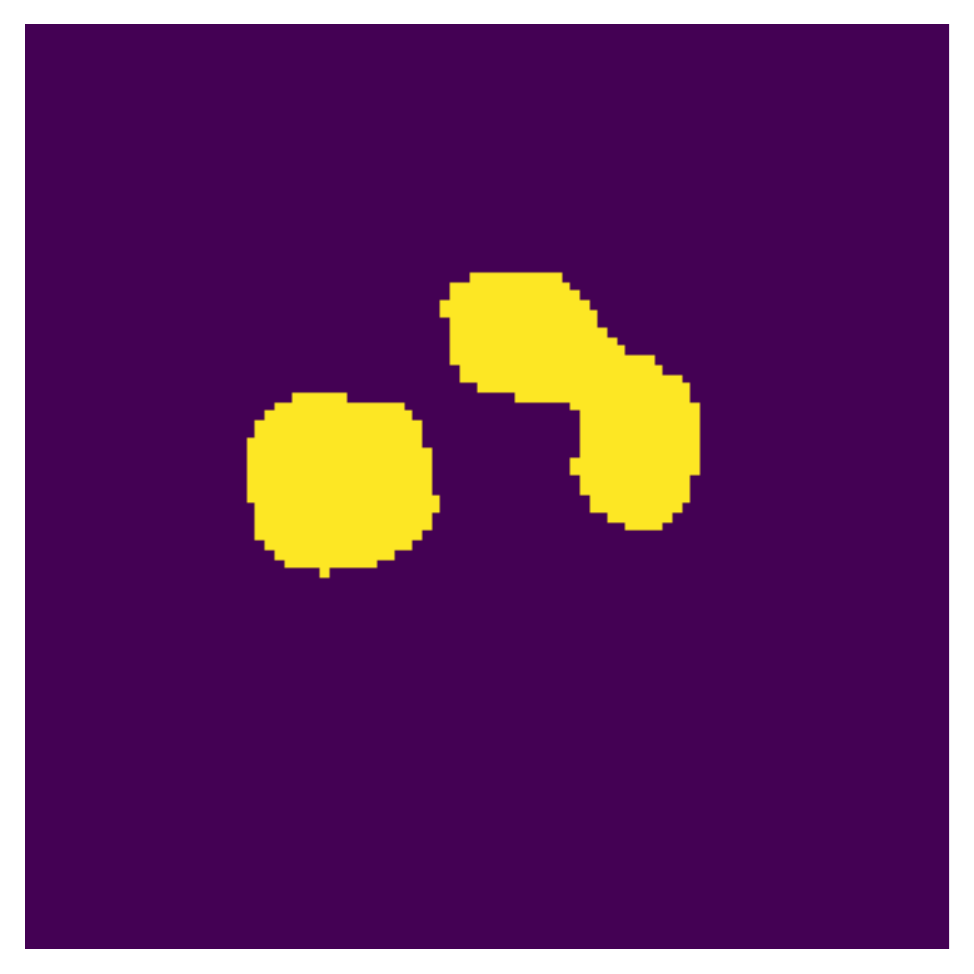}}
    \subfloat[$\mathbb E F_{\text{ls}}(\boldsymbol{\xi})$]{\label{fig:3.stage1.f}\includegraphics[width=0.3\textwidth]{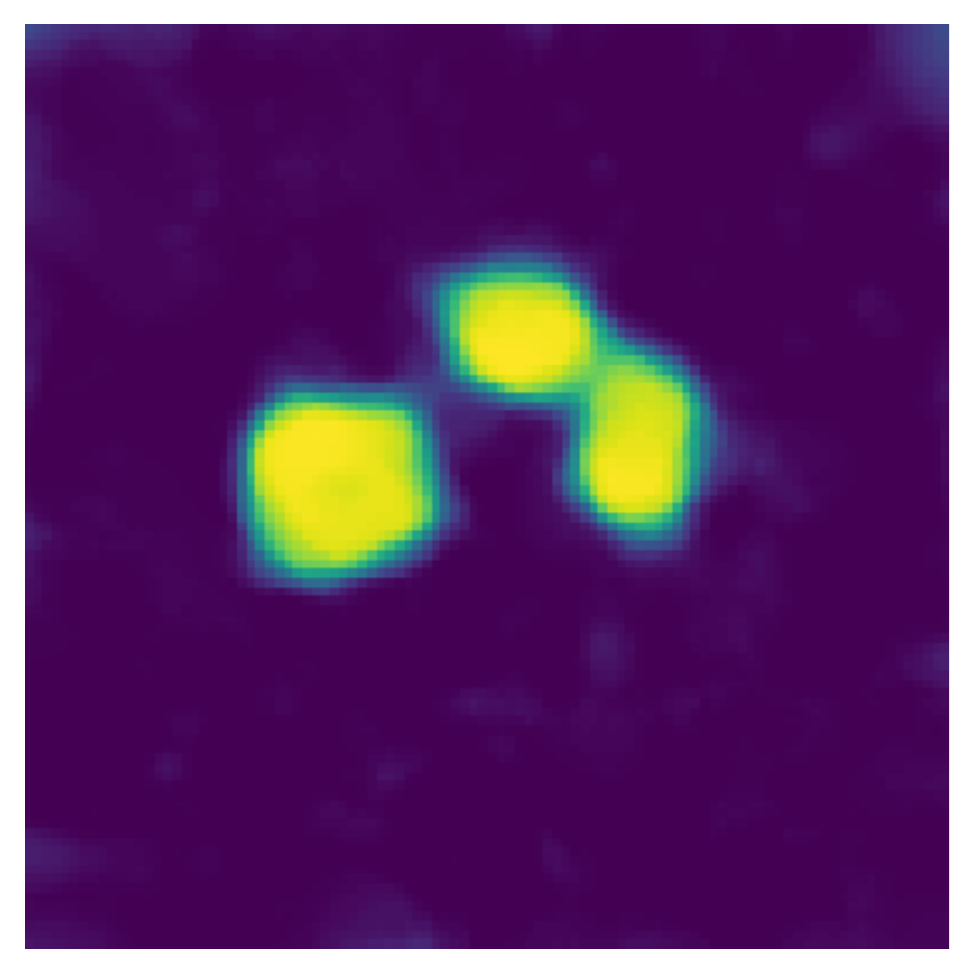} }
    \caption{Example of \coyq{Stage 1} in identifying the location of the inclusions. Each row corresponds to the same experiment. }
\label{fig:3.stage1}
\end{figure}

\coyq{In \Cref{fig:3.stage1} we show an example of Stage 1.} The true densities are generated randomly using the noisy star-shaped prior \eqref{2.eq.noisy_star}. \coyq{We consider a parallel-beam geometry described in \Cref{sec:1} with 100 equidistantly spaced angles in $[0^{\circ}, 180^{\circ})$, i.e. $\theta_\text{max} = 180^\circ$ and $N_{\theta} = 100$. The \copc{detector} length is set to 2 and centered at the origin, thus, the size of a detector pixel is $2/N_s$ with $N_{s} = 100$. The number of inclusions is unknown.}

\coyq{\Cref{tab:3.centers} lists the estimated centers of mass in Stage 1 \copc{compared} with the star-shaped centers used to generate \Cref{fig:3.stage1.a,fig:3.stage1.d}. Note that the estimated centers of mass are far from the centers assigned to the star-shaped inclusions, but they are accurate estimates of the true centers of mass.} We report that the method consistently had similar results for other examples tested by the authors.

\begin{table}[tbhp]
{\footnotesize
  \caption{Prediction of centers}\label{tab:3.centers}
\begin{center}
\setlength{\tabcolsep}{10pt} 
\renewcommand{\arraystretch}{1.5}
\begin{tabular}{|c|c|c|c|} \hline
 & \bf inclusion 1 & \bf inclusion 2 & \bf inclusion 3 \\ \hline
star-shaped centers & $(-0.298, 0.452)$  & $(-0.157, -0.143)$  & $(0.382, -0.252)$ \\ \hline
exact centers of mass & $(-0.279, 0.548)$ & $(-0.139, -0.058)$ & $(0.424, -0.305)$ \\ \hline
estimated centers of mass & $(-0.272, 0.540)$ & $(-0.137, -0.059)$ & $(0.424, -0.300)$ \\ \hline \end{tabular}
\end{center}
}
\end{table}

We notice that if assumption (III) in \Cref{sec:2.star} is satisfied the inclusions are well approximated. In \coyq{the case of \Cref{fig:3.stage1.d} we see that $F_{\text{ls}}[\mathbb E \xi]$ falsely identifies} the two close inclusions as one. However, \Cref{fig:3.stage1.f} suggests that there is uncertainty in detecting the inclusions and \copc{this} can be interpreted as the violation of assumption (III).

\coyq{We assign a bounding box $\square^{i}$ with $i=1,\dots,N_{\text{inc}}$ for each inclusion. We first find the \copc{leftmost} pixel of the $i$th inclusion in $\overline{\alpha}$, then set the \copc{leftmost} position of $\square^{i}$ by
\copc{subtracting $d_{\text{min}}^D/2$ pixels}
in order to ensure assumption (III) in \Cref{sec:2.star} \copc{is} satisfied.
\copc{We define} the right, top and bottom bounds of $\square^{i}$ \copc{in a similar manner}. These bounding boxes are essential for Stage 2,} since \copc{they decompose} the domain into regions with a single inclusion \copc{each}. \Cref{remark:2.star_shaped} ensures that there is a well-defined posterior distribution on such regions for the star-shaped prior.

\coyq{
\begin{remark}
The main purpose of Stage 1 is to estimate the bounding boxes and decompose the image into regions with a single inclusions. This can be achieved with other methods, e.g., Bayesian methods with MRF-type priors \cite{MarkkanenRoininen} or Besov-type priors \cite{1930-8337_2012_2_183,Lassas2009}. Although the sampling method used in Stage 1 is robust to discretization refinement, the forward mapping $G$ is still dimension-dependent. We can speed up Stage 1 by truncating the KL expansion with only $10\sim20$ modes to evaluate the bounding boxes. In our experiments the computational costs of Stage 1 is negligible compared to Stage 2.
\end{remark}

}

\cofinal{
\begin{remark} We emphasize that the performance of Stage 2 depends on the correct detection of number of inclusions in Stage 1. If Stage 1 of the approach fails to detect the correct number of bounding boxes, the true inclusions will not be supported by the posterior in Stage 2.
\end{remark}
}

\subsection{Stage 2: Estimation and Uncertainty Quantification of the Boundaries of Inclusions} \label{sec:3.stage2}
In this stage we construct the posterior measure \coyq{$\boldsymbol{\mu}^{\boldsymbol{y}}$} using the star-shaped prior. The number of inclusions $N_{\text{inc} }$ is estimated in Stage 1 together with the bounding boxes that decompose the domain into regions with a single inclusion \copc{each}. \correct{Note that in Stage 2 we only reconstruct the boundaries of the inlcusions provided by Stage 1.} We let $\xi_i$, for $i=1,\dots,N_{\text{inc}}$ be the $H$-valued random variables with $H = \mathbf T([0,2\pi))$.

We assume that the correlation parameter $\tau=1$ and $\gamma$ is known in \eqref{eq:2.cov_simple}. This choice of $\tau$ indicates that points on the boundary are highly \copc{correlated}. \coyq{Since $\xi_i$ models the boundary of an inclusion,} we may use the periodic boundary conditions to construct $Q_{\gamma,1}$. Furthermore, \coyq{since $\xi_i$ is 1D,} we can assemble them directly from \eqref{eq:2.kl_1d}. Furthermore, \coyq{we assume $m$ in \eqref{eq:2.KL} is \copc{a known constant} for all $\xi_i$. The centers of mass estimated in Stage 1 are used as the centers of the star-shaped inclusions.} 



The posterior measure \coyq{$\boldsymbol{\mu}^{\boldsymbol{y}}$} is constructed following \Cref{sec:2.posterior}. We draw samples from the posterior distribution using the Gibbs sampling method, which is introduced in \Cref{sec:2.mcmc}, to estimate the two types of mean for the posterior, i.e. \coyq{$\overline \alpha_{\mathbb E c} = F_{\text{star}}[ (\mathbb E \boldsymbol{\xi},\mathbb E \boldsymbol{c}) ]$ and $\widehat \alpha_{\mathbb E c} = \mathbb E F_{\text{star}}[(\boldsymbol{\xi},\boldsymbol{c})]$,} introduced in \Cref{sec:3.stage1}. Here, the subscript $\mathbb E c$ means that the expected image is drawn around the expected centers.

\correct{We consider the highest posterior density (HPD) credibility interval \cite{gelman2013bayesian} to estimate the uncertainties in detecting the boundaries of the inclusions. A 95\% HPD interval for \coyq{an 1D real-valued random variable with probability density $\pi(x)$ is defined by
\[
I_{\text{HPD}}=\{x\in \mathbb R | \pi(x) > c_{\text{HPD}}\}
\]
where $c_{\mathrm{HPD}} > 0$ is the largest number such that
\begin{equation}
    \int_{\{ x\in \mathbb R| \pi(x) > c_{\text{HPD}} \}} \pi(x) \ dx = 0.95,
\end{equation}}
i.e., the smallest interval with the highest credibility. Note that there are other credibility intervals that can be used, e.g., the commonly used equal-tailed credibility interval \cite{gelman2013bayesian}. We choose the HPD credibility interval since it reveals the asymmetry and multi-modal properties of the posterior. Furthermore, it contains the maximum a posterior (MAP) estimate \cite{gelman2013bayesian}.
}

\coyq{In \Cref{alg:3.stage_2}, we summarize how to use the HPD to quantify the uncertainties in the estimated boundaries. Note that we compute the HPD intervals only with respect to $\boldsymbol{c}$, since we find that in the context of the experiments in this paper the center of the star-shaped inclusion reveals the full modality of the posterior. Furthermore, the knowledge of the exact location of the center results in a posterior with a single mode. The results for the latter is omitted for brevity.}


\begin{algorithm}[th]
\caption{Visualizing the uncertainty in the estimated boundaries}
\label{alg:3.stage_2}
\begin{algorithmic}[1]
\STATE{Input $N_{\text{inc}}$, \coyq{$\boldsymbol{c}$ and $\square^{i}$ for $i=1,\cdots,N_{\mathrm{inc}}$ obtained in Stage 1 by using \Cref{alg:3.stage_1}.}}
\FOR{$i=1, \cdots, N_{\text{inc}}$}
\STATE{Construct the posterior measure \coyq{$\mu^{\boldsymbol{y}}$ using the star-shaped prior with a single inclusion.}}
\STATE{Draw samples $\{ (\xi^{(j)}_i, c^{(j)}_i) \}_{j=1}^{N_{\text{sample}}}$ by using the Gibbs sampling method with \Cref{alg:3.gibbs}.}
\STATE{\coyq{Using the samples of $\{c_i^{(j)}\}_{j=1}^{N_{\text{sample}}}$, compute the HPD intervals $I^{i}_{\mathrm{HPD}}$, which includes $N^{i}_{\text{modes}}$ disjoint subintervals, $\{I^{i,k}_{\mathrm{HPD}}\}_{k=1}^{N^{i}_{\text{modes}}}$, according to $N^{i}_{\text{modes}}$ modes in the posterior}.}
\STATE{\coyq{Compute the mean $\{ (\bar \xi^{k}_i, \bar c^{k}_i) \}$ of the samples $\{ (\xi^{(j)}_i, c^{(j)}_i) \}$ with $j\in I^{i,k}_{\mathrm{HPD}}$ for $k=1,\cdots,N^{i}_{\text{modes}}$.}}
\STATE{\correct{Collect the HPD intervals based on the samples of $\{ \xi^{(j)}_i(\vartheta) \}_{j\in I_{\text{HPD}}^{i,k}}$ for each angle $\vartheta$ to represent the uncertainty in $\bar \xi^{k}_i(\vartheta)$.}}
\ENDFOR
\end{algorithmic}
\end{algorithm}

\correct{\coyq{In the last step of \Cref{alg:3.stage_2} we quantify the local uncertainties for each angle $\vartheta$, then we} construct a radial credibility band by interpolating the radial HPD intervals for each inclusion. This estimate provides a \coyq{local} uncertainty quantification of the shape of the inclusion. \coyq{In addition, we also provide the global variance $\mathbb E \| \xi_i - \mathbb E \xi_i \|^2_{H^{i}}$ as a global uncertainty estimator for each inclusion.}}

\correct{
\begin{remark}
The samples from the radial random variables $\xi_i$ are generally \copc{coupled} with the samples of the centers $c^{(j)}_i$, due to the structure of the the posterior. However, by ignoring the center component of the samples we can approximately achieve independent samples of $\xi_i$. We plot the credibility band with respect to $\bar c^{k}_i$, for $i=1,\dots,N_{\text{inc}}$. This is only for illustration purposes and should not be confused with samples of the random variable $\xi|(c_i=\bar c^k_i)$.
\end{remark}
}

\section{Numerical Results} \label{sec:4}
In this section we test the presented method on synthetic images with inclusions. \copc{The test} images contain single and multiple inclusions with boundaries with various regularities. We also test the method on a tomographic X-ray data of a lotus root filled with attenuating objects.

\subsection{Generated Inclusions} \label{sec:4.1.gen}
In this section, we consider synthetic images with inclusions. We construct an image (\coyq{attenuation} field $\alpha$) using the star-shaped prior.
We then apply \coyq{our two-stage} method to estimate the center of inclusions and quantify the uncertainty in estimating the boundaries.

We consider the domain $D$ to be the unit disk and construct the random \coyq{attenuation} fields using the noisy star-shaped prior \eqref{2.eq.noisy_star}
\begin{equation}
    \alpha^{\text{noisy}} (x) = \coyq{F^{\text{noisy}}_{\text{star}}[\xi_0,(\boldsymbol{\xi},\boldsymbol{c}) ](x),}
\end{equation}
where foreground and background attenuation is set to $a^+=1$ and $a^-=0.1$, respectively. Recall that $\xi_0$ is the random variable that controls the fluctuations in the background, while \coyq{$\boldsymbol{\xi}$ determines the boundaries of the inclusions}. We draw samples from the noisy star-shaped density by first randomly choosing a center in the unit disk and then taking \coyq{$\xi_0\sim \mathcal N(0,Q_{\gamma^0,\tau^0})$ with $\gamma^0 = 2.5$, $\tau^0 = 50$} as the parameters for the background. We use a sampling-and-elimination method to ensure that the inclusion appears inside the unit disk. We truncate the KL expansion after 100 terms.

\coyq{In our numerical tests, we use the scikit-image package \cite{vanderWalt2014} (version 0.19.0.dev0),} a Python image processing toolkit, to carry out the forward calculations, i.e., 2D Radon transforms \copc{of} images. The package uses the parallel beam geometry described in \Cref{sec:1}. We set $\theta_\text{max} = 180^\circ$ and $N_{\theta} = 100$ for a full set of images. The \copc{detector} length is set to 2 and centered at the origin. This means that $L^\perp_{\theta}$ is centered at the origin and  $s\in[-1,1]$.


\coyq{In Stage 2 (\Cref{alg:3.gibbs}) of our method, we first run a warm-up phase of 20 Gibbs iterations with $N_{\text{pCN}} = N_{\text{MH}} = 500$. We tune the parameters in the proposal such that the loops within the Gibbs sampler provide an acceptance rate of 15\% to 25\%. Then, we fix the parameters, set $N_{\text{pCN}} = N_{\text{MH}} = 20$, and calculate $2\times 10^4$ samples from $\boldsymbol{\mu}^{\boldsymbol{y}}$. We note that designing efficient MCMC methods for such posteriors is an area of active research, see, e.g., \cite{Dunlop2016,1930-8337_2016_4_1007,SAIBABA2021110391}.
}


\subsubsection{Single Inclusion} \label{sec:4.1.1.single}
In this section we consider \coyq{attenuation} fields $\alpha$ with a single inclusion. We investigate \coyq{the effect of the observation noise on the estimations in our method.} 
The results in this section are illustrated on randomly generated test examples. The authors found similar results for all other test examples with same setting.

\coyq{To create ground truth images we use the noisy start-shaped prior \eqref{2.eq.noisy_star}, and set the inclusion $\xi_{1}\sim \mathcal N(0,Q_{\gamma^{1},\tau^{1}})$ with $\gamma^{1} \in \{2, 3\}$ and $\tau^{1}=1$}. 
We define the \emph{signal to noise ratio (SNR)} in an observation to be
\begin{equation}
    \coyq{\text{SNR} = \frac{\| \boldsymbol{y} \|_2}{\|\boldsymbol{\varepsilon}\|_2},
}\end{equation}
and define the \emph{noise level} percentage in an observation to be $1/\text{SNR}\times 100$.

\coyq{In \Cref{fig:4.1inc_multi_mode} we show the test problem with $\gamma^1=3$ and noise level $1\%$. Here, the true star-shaped center of the inclusion is $c = (0.154, -0.215)$. We provide some samples from the prior distribution and the posterior distribution according to Stage 2 of our method in \Cref{fig:4.10inc.f,fig:4.10inc.g}. In addition, in \Cref{fig:4.10inc.a,fig:4.10inc.b,fig:4.10inc.c} we illustrate the performance of our method. By using HPD we are able to quantify the uncertainties of the boundaries according to each mode, which are given in \Cref{fig:4.10inc.b,fig:4.10inc.c}. For each mode although the estimated center does not match the true star-shaped center, the boundary of the inclusion is still well reconstructed. The reason is that the star-shaped center is not uniquely defined and is not consistent with the center of the mass that is used in our method. Moreover, we would like to remind that the main goal of our method is to reconstruct the boundary of the inclusion. \Cref{fig:4.10inc.a} shows the posterior mean with the $99\%$ HPD interval.

\correct{The posterior mean is not a good representation of the boundary due to the multi-modal nature of the posterior.} This is clearly \copc{seen in} \Cref{fig:4.10inc.a}.} \correct{Therefore, in the rest of this section we let $(\bar \xi_i^1, \bar c_i^1)$, computed from the first HPD interval $I^{i,1}_{\text{HPD}}$, \copc{represent} the boundary. The choice of the first HPD is arbitrary. The authors confirm that in the experiments in this work the quality of the reconstructed boundaries $(\bar \xi_i^k, \bar c_i^k)$ is comparable for all HPD intervals $I^{i,k}_{\text{HPD}}$.}

\begin{figure}[tbhp]
    \centering
    \subfloat[\coyq{true $\alpha(\xi_{1})$}]{\label{fig:4.10inc.d}\includegraphics[width=0.3\textwidth]{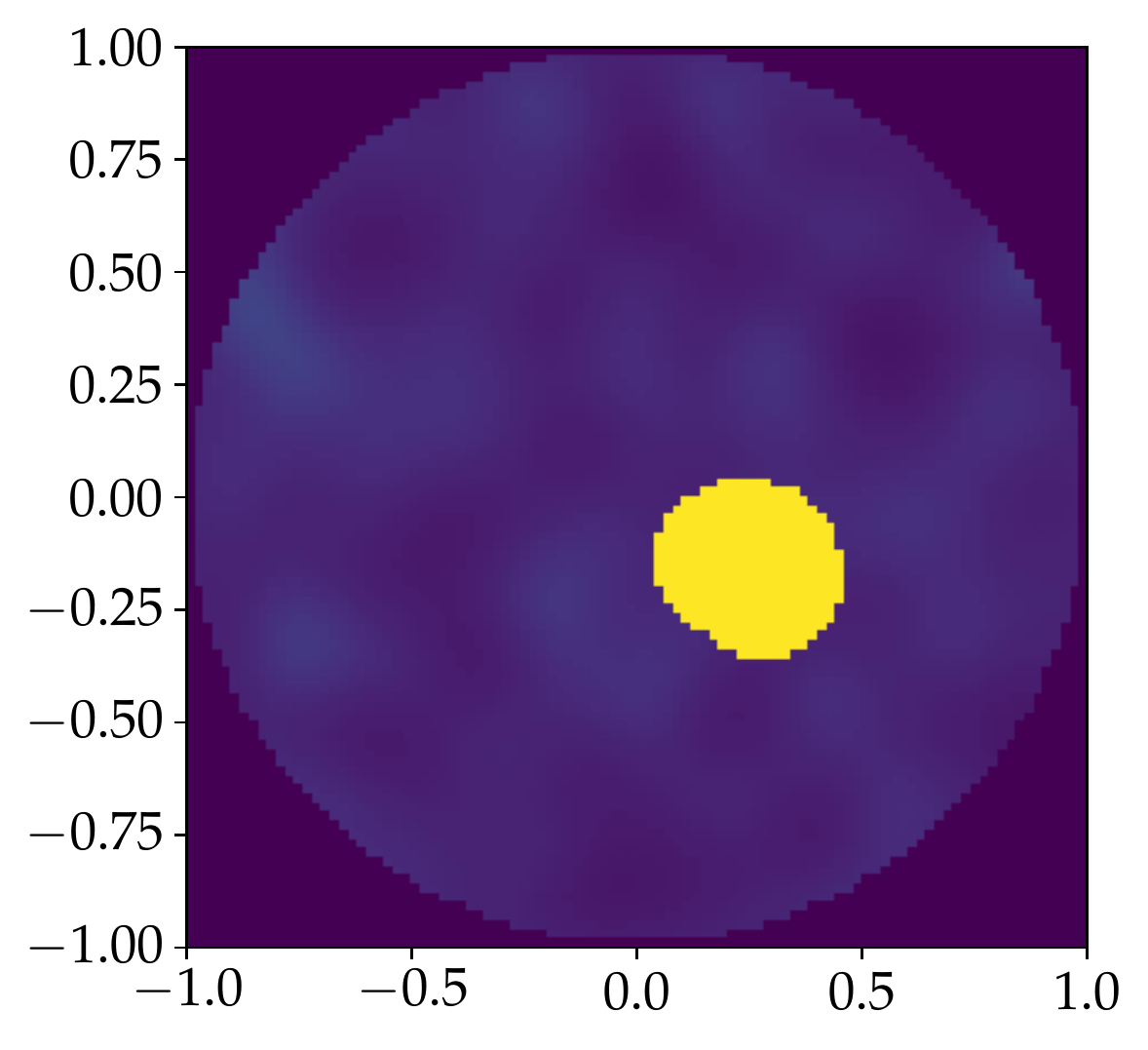} }\hspace{5mm}
    \subfloat[\coyq{Sinogram $\boldsymbol{y}$}]{\label{fig:4.10inc.e}\includegraphics[width=0.3\textwidth]{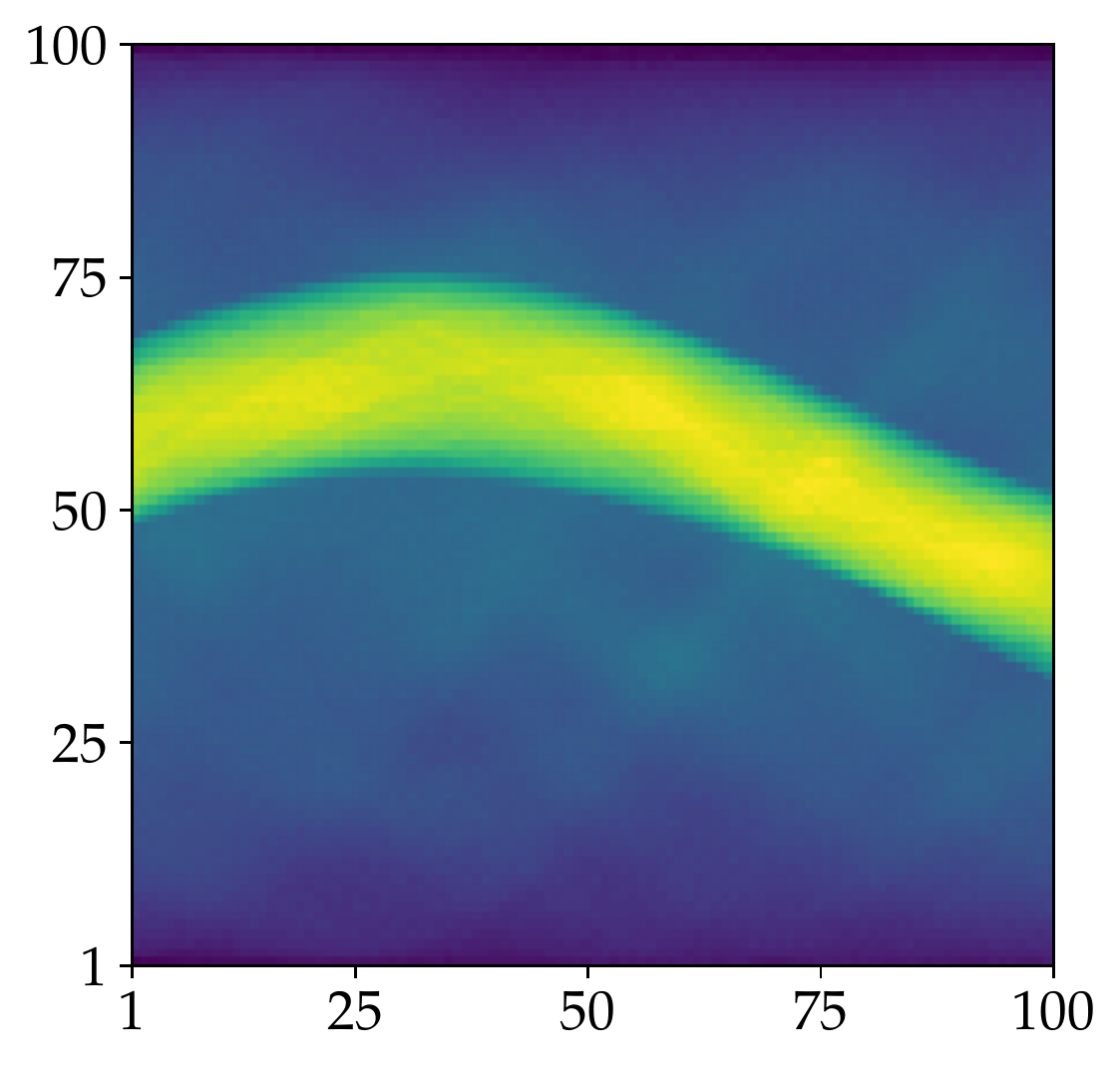} } \\
     \caption{\coyq{Test problem for \copc{the single-inclusion} case with regularity parameter $\gamma^1 = 3$ and noise level $1\%$.}}
\label{fig:4.1inc_multi_mode}
\end{figure}

\begin{figure}[tbhp]
    \centering
    \subfloat[prior samples]{\label{fig:4.10inc.f}\includegraphics[width=0.45\textwidth]{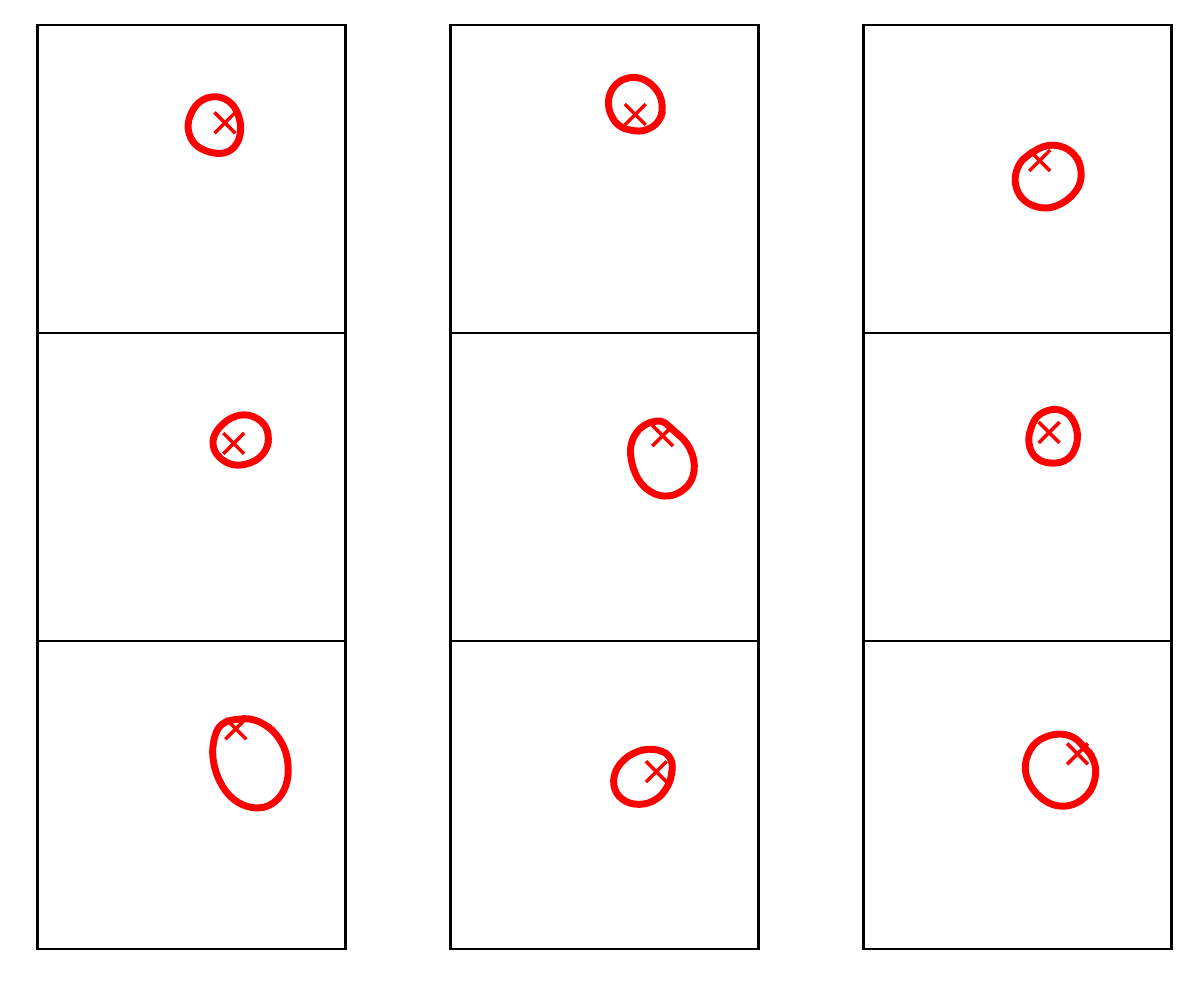} }
    \subfloat[posterior samples]{\label{fig:4.10inc.g}\includegraphics[width=0.45\textwidth]{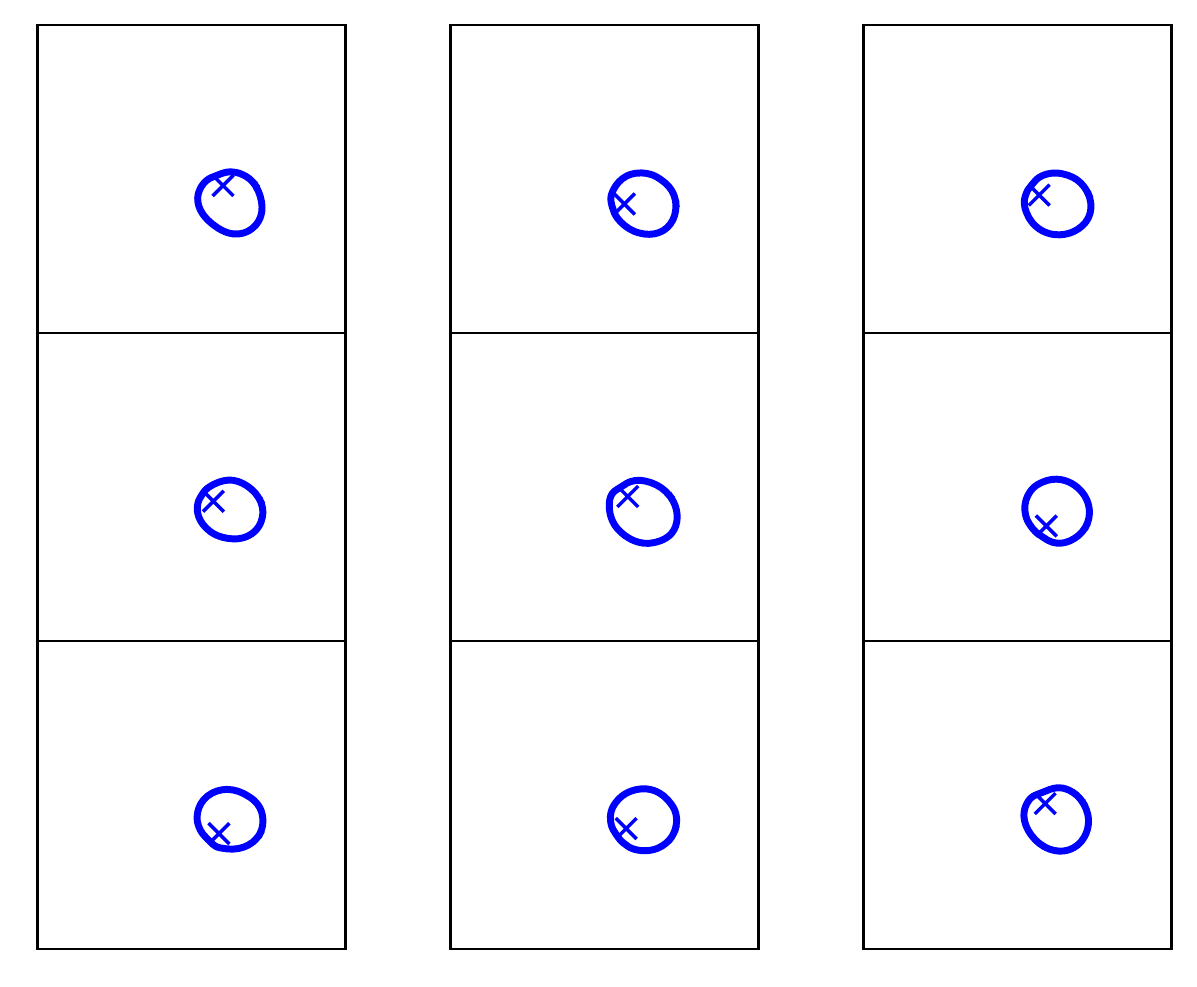} } \\
    \subfloat[1st posterior mode]{\label{fig:4.10inc.b}\includegraphics[width=0.3\textwidth]{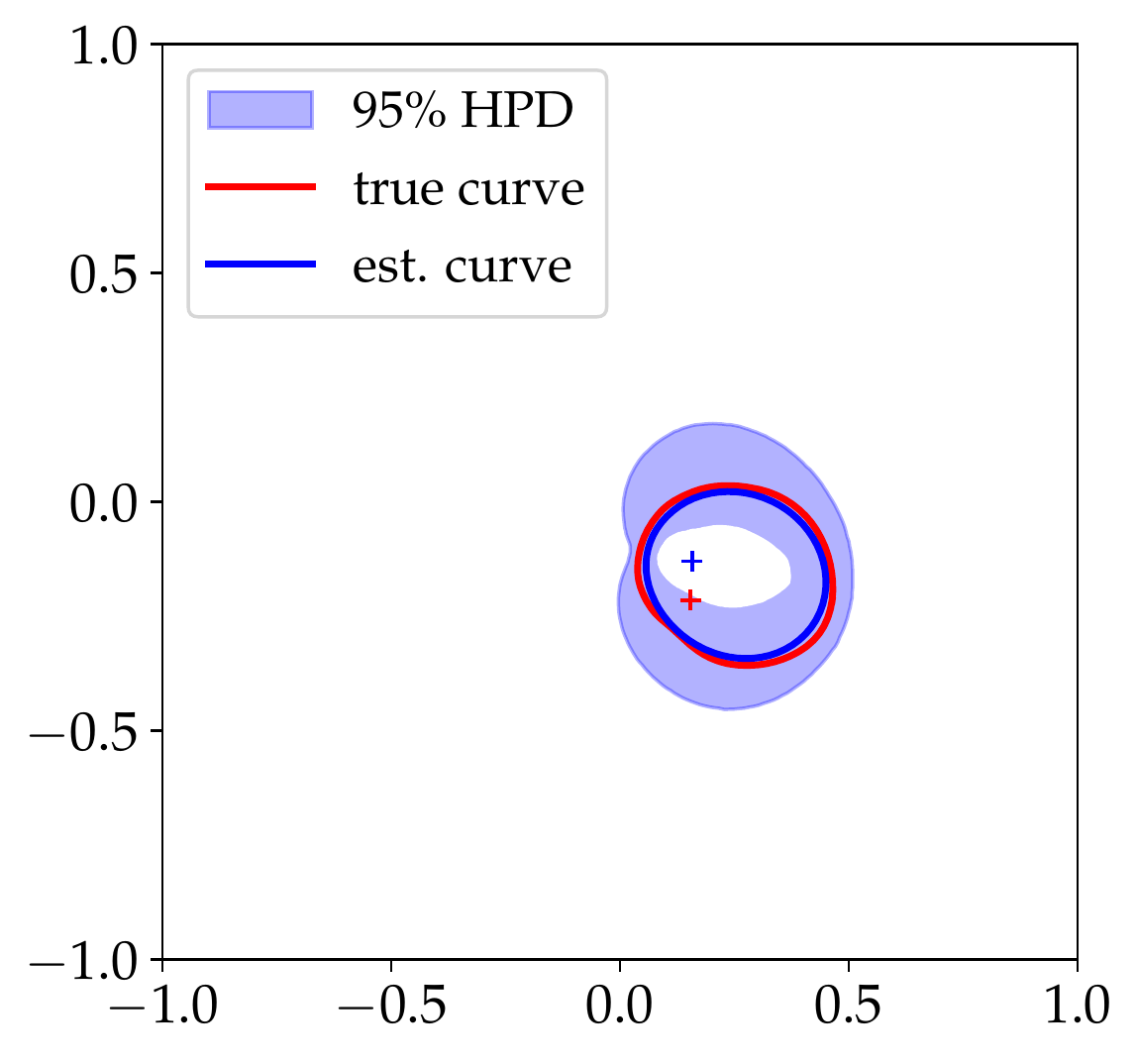} }
    \subfloat[2nd posterior mode]{\label{fig:4.10inc.c}\includegraphics[width=0.3\textwidth]{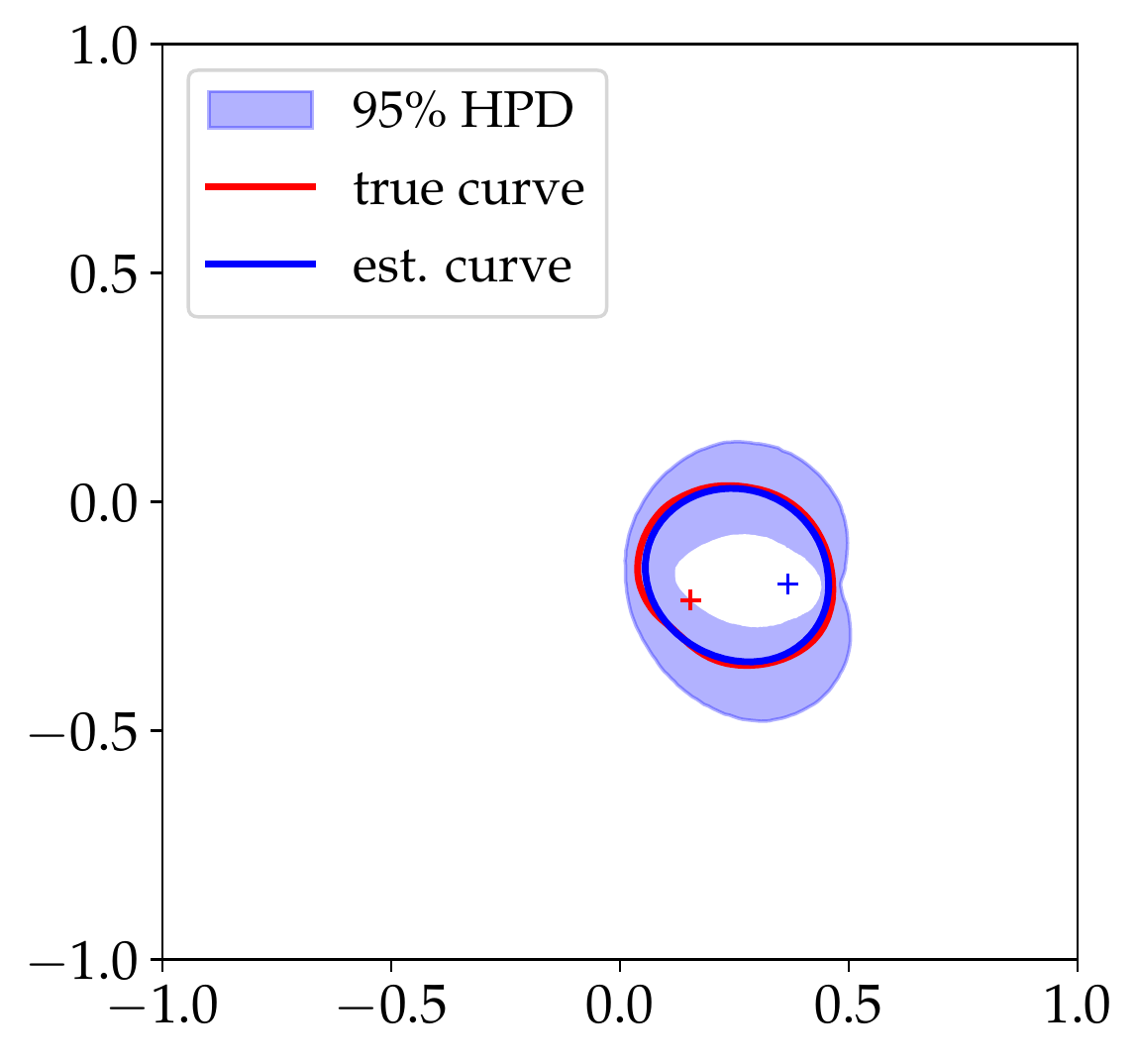} }
    \subfloat[posterior mean]{\label{fig:4.10inc.a}\includegraphics[width=0.3\textwidth]{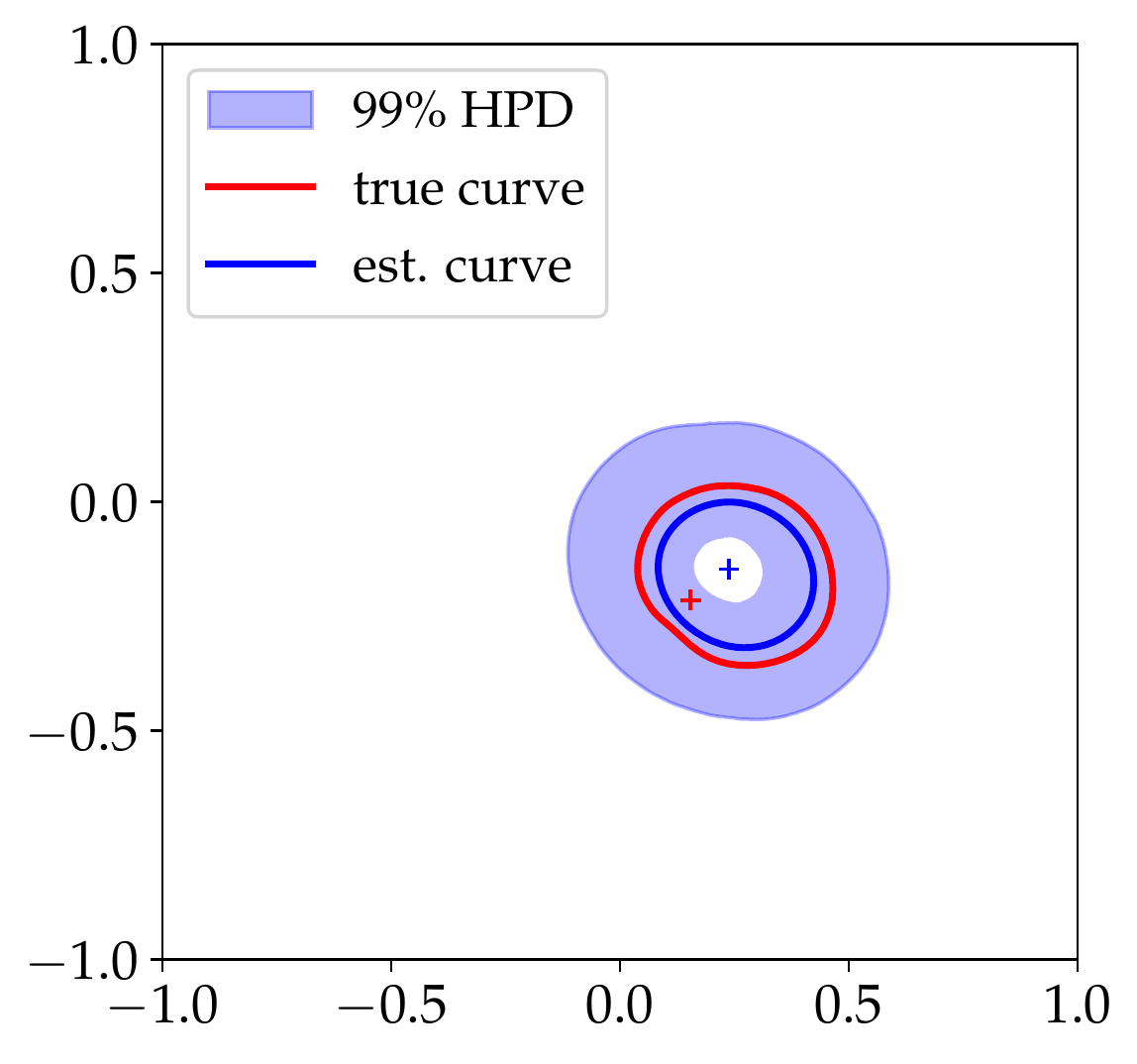} }
    \caption{\coyq{Results of our method to the \copc{single-inclusion} test problem shown in \Cref{fig:4.1inc_multi_mode}. In \Cref{fig:4.10inc.b,fig:4.10inc.c,fig:4.10inc.a}, the red curve and the red cross represent the true boundary and the true star-shaped center. The blue cross and the blue curve represent $\mathbb E c_1$ and $F_{\text{star}}[(\mathbb E \xi_1,\mathbb E c_1)]$, respectively. }}
\label{fig:4.prior_posterior}
\end{figure}

\coyq{In \Cref{fig:4.diagnostics} we present the diagnostics for the Gibbs sampler in \Cref{alg:3.gibbs} for 4 MCMC chains. The first chain is constructed by using the mass center obtained in Stage 1 as the initial center for star-shaped inclusion. The other chains \copc{use} a uniformly distributed random initial guesses from the bounding box. In all our numerical results, we only focus on samples in $I^1_{\text{HPD}}$. We notice that the estimated boundaries detected for other HPD intervals are comparable with the ones in $I^1_{\text{HPD}}$ in all our experiments. \Cref{fig:4.diagnostics.b} shows the auto-correlation function (ACF) of these samples for each chain as well as the mean of all chains. We can see that ACF drops to $0$ after a reasonable number of samples. 

\cofinal{We compute the effective sample size (ESS) of the MCMC samples following \cite{10.1214/20-BA1221}. The ESS and the auto-correlation function is computed using the ArviZ \cite{arviz_2019} package in Python. We refer the reader to \cite{mcbook,10.1214/20-BA1221} for more detail on ESS.
}

The ESS for all chains are between 200 and 1000 samples. In \Cref{fig:4.diagnostics.a} we can see that the mean curves of samples in $I^1_{\text{HPD}}$ for all chains match. This is a necessary condition for the MCMC method to be converging \cite{gelman1992inference}.}

\begin{figure}[tbhp]
    \centering
    \subfloat[auto-correlation function]{\label{fig:4.diagnostics.b}\includegraphics[height=0.425\textwidth]{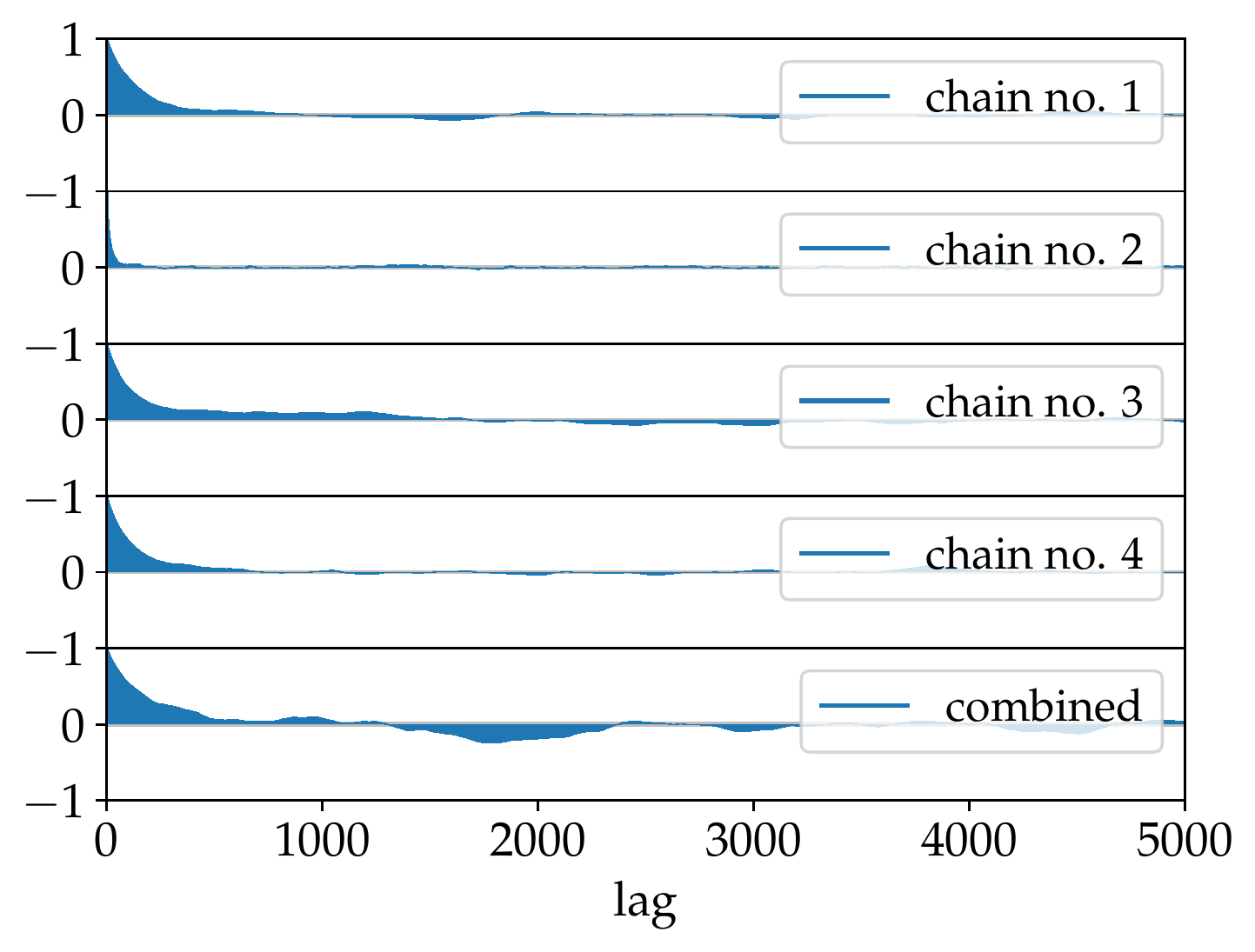} }\hspace{5mm}
    \subfloat[mean curves]{\label{fig:4.diagnostics.a}\includegraphics[height=0.35\textwidth]{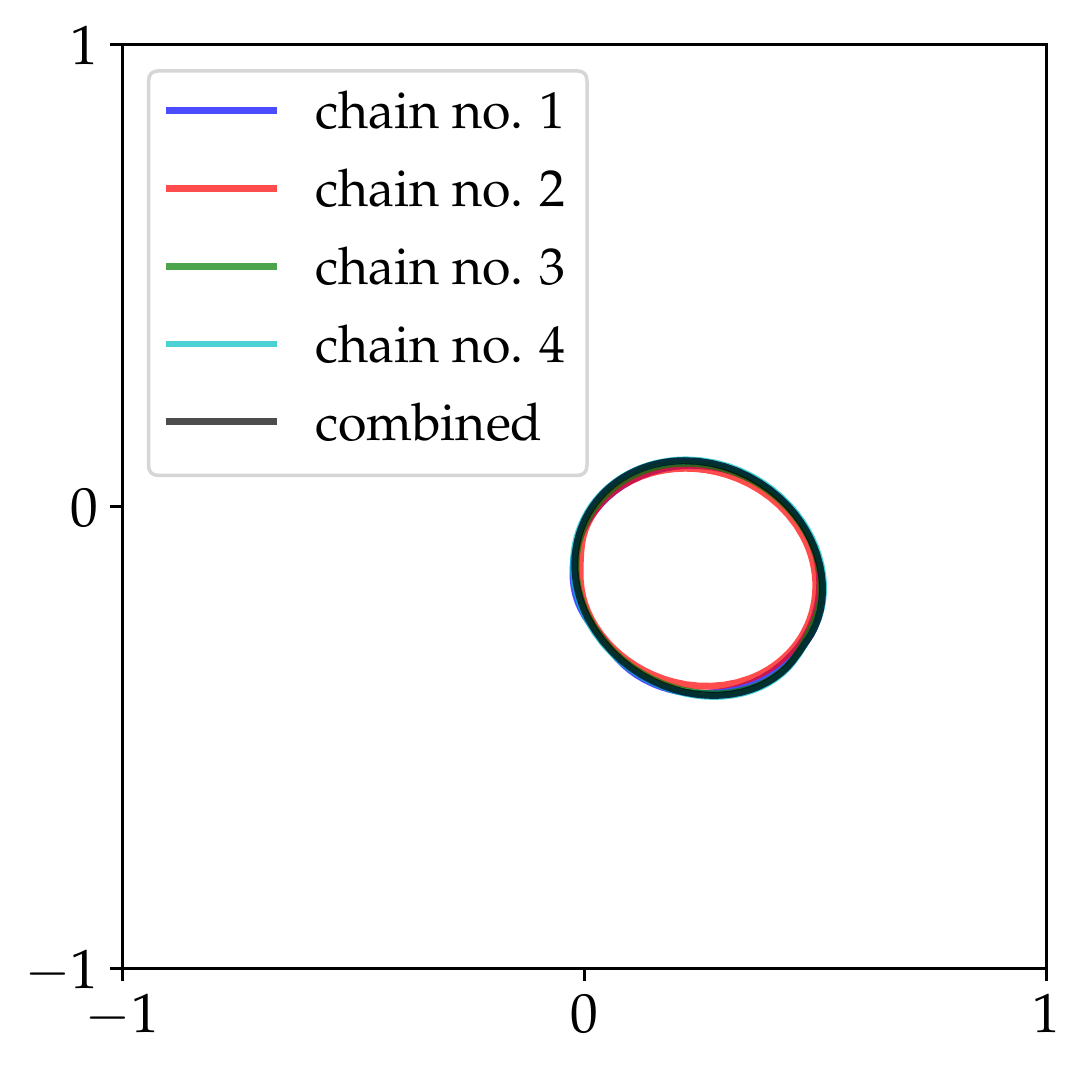} }
    \caption{\coyq{Performance of the Gibbs sampler in \Cref{alg:3.gibbs}. Here we only consider the samples in $I^1_{\text{HPD}}$.}}
\label{fig:4.diagnostics}
\end{figure}

\coyq{To illustrate the effect of the noise level, we apply our method to reconstruct the boundary from \delete{the} sinograms including $2.5\%$ and $5\%$ noise, respectively. The true $\alpha$ is the same as shown in \Cref{fig:4.10inc.d}, and the results are show in \Cref{fig:4.1inc}. We can see that our method can provide a good estimated boundary of the inclusion, although the estimated star-shaped center increases the uncertainty in estimating the boundary. It is clear that the accuracy of estimation increases as the noise level is reduced.} In \Cref{fig:4.1inc} the 99\% HPD band represents the uncertainty in estimating the boundary of the inclusion. We notice that the HPD band in all test cases completely covers the true outline of the inclusions. \coyq{Furthermore, it is significantly smaller for the test case with smaller noise level.}
The HPD band does not provide a uniform uncertainty around the boundary, as it is suggested in \Cref{fig:4.1inc}. The variation in the uncertainty around the boundary is a result of the estimated center and the choice of the Whittle-Mat\'ern covariance, \correct{i.e., the prior parameters}.


\begin{figure}[tbhp]
    \centering
    \subfloat[$2.5\%$ noise level]{\label{fig:4.1inc.d}\includegraphics[width=0.35\textwidth]{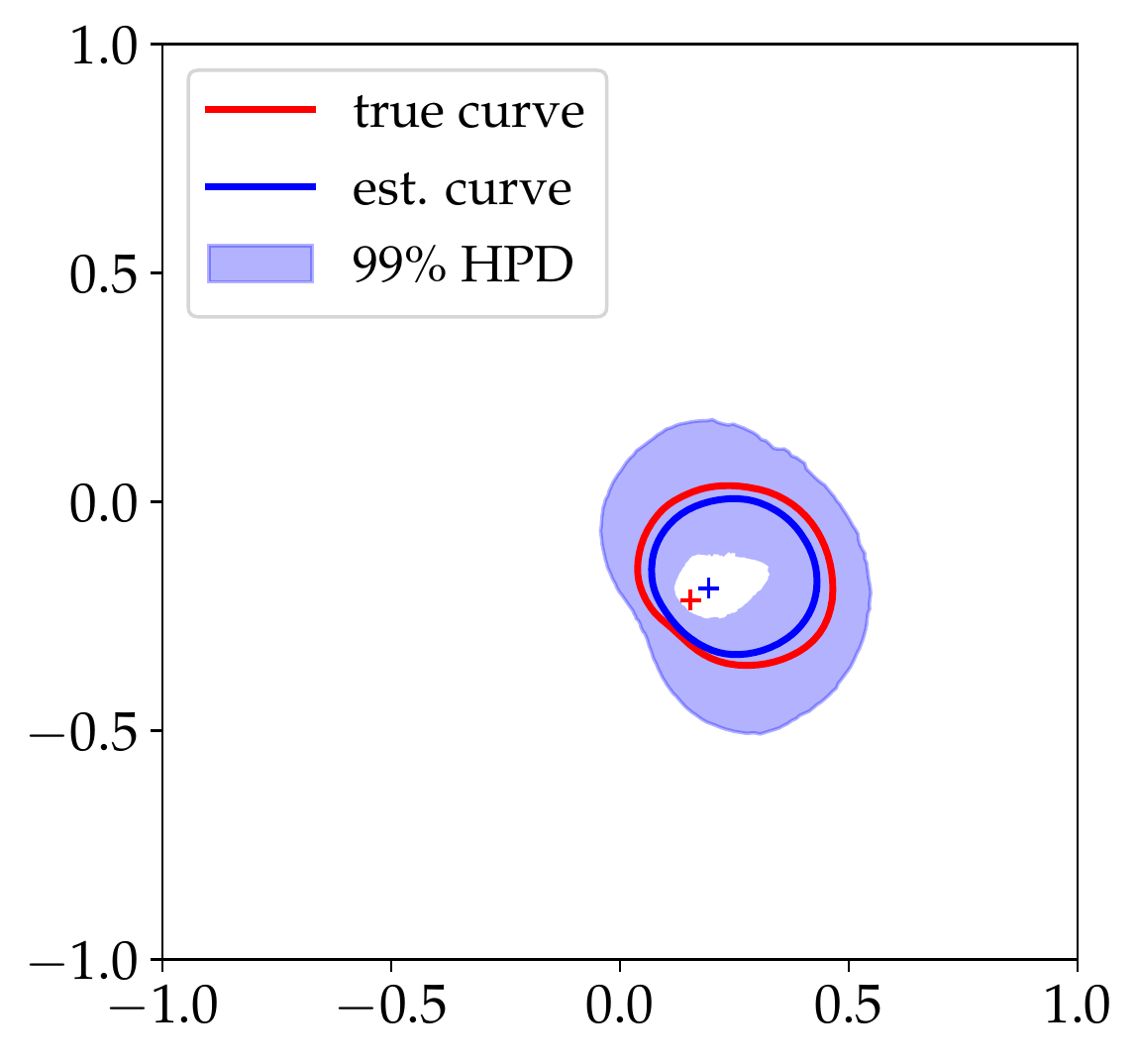} }\hspace{5mm}
    \subfloat[$5\%$ noise level]{\label{fig:4.1inc.e}\includegraphics[width=0.35\textwidth]{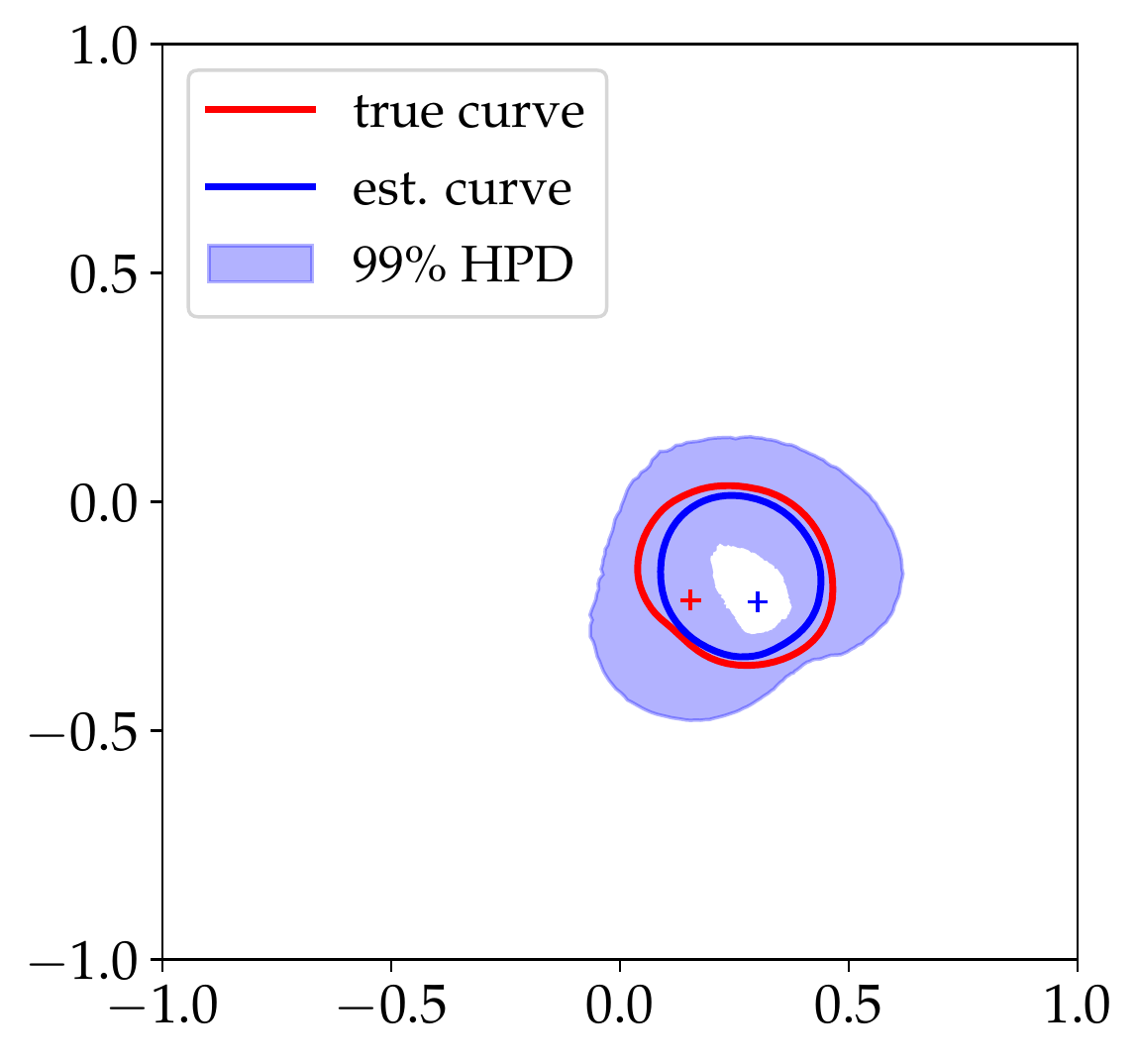} }
    \caption{\coyq{Illustration of the performance of our method with respect to different noise levels. The red cross represents the true star-shaped center, and the blue cross represents $\mathbb E c_1$ from Stage 1.}}
\label{fig:4.1inc}
\end{figure}





\coyq{To study the global uncertainty estimation, in \Cref{tab:4.1inc_var} we list the global variances $\mathbb E \| \xi_1 - \mathbb E \xi_1 \|^2_H$ for different noise levels. We notice that the relative difference in variance among the test cases confirms the uncertainty that is visualized in \Cref{fig:4.1inc}. This confirms our intuition that larger observation noise results in larger uncertainty in our estimation.}

\begin{table}[tbhp]
{\footnotesize
  \caption{Estimation of the global variances.} \label{tab:4.1inc_var}
\begin{center}
\setlength{\tabcolsep}{10pt} 
\renewcommand{\arraystretch}{1.5}
\begin{tabular}{|c|c|c|c|} \hline
 & \Cref{fig:4.10inc.b} & \Cref{fig:4.1inc.d}& \Cref{fig:4.1inc.e}  \\ \hline
$\mathbb E \| \xi_1 - \mathbb E \xi_1 \|^2_H$ & 0.048  &  0.323 & 0.334 \\ \hline
\end{tabular}
\end{center}
}
\end{table}

\coyq{In the previous tests, we assume that we know the correct value of $\gamma^{1}$. Now we will discuss the importance of the setting of $\gamma^{1}$.
We generate the ground truth image using $\gamma^1 = 2$, and compare the results by applying our method with $\gamma^{1}=2$ and $3$, which are shown in \Cref{fig:4.1inc_2}.}

\coyq{With the correct $\gamma^{1}$, we can see that} the estimated curve in \Cref{fig:4.1inc_2.a} accurately represents the true boundary. The location of some of the larger indentations in boundary is also represented in the estimated boundary. \correct{ For this example, we provide some of the prior and posterior samples in \Cref{fig:4.1inc_2.c,fig:4.1inc_2.d}, respectively.}

\coyq{\Cref{fig:4.1inc_2.b} shows the results by using $\gamma^{1}=3$ in our method. The reconstructed curve (the blue curve) is smoother than the true curve due to larger regularity parameter $\gamma^{1}$. However,} it fairly estimates the overall shape and orientation of the inclusion. We notice that the \coyq{HPD} band around the boundary is \coyq{wider than} the one in \Cref{fig:4.1inc_2.a}. \copc{A wrong choice of} \correct{the prior parameters may be the reason for larger uncertainty in the boundary.}


\coyq{In \Cref{fig:4.1inc_2} we also give the global variance in both cases.} The relative differences between the variances verifies the uncertainty bands indicated in the figures. We note that the global variance in \Cref{fig:4.1inc_2.b} is significantly larger than the one in \Cref{fig:4.1inc_2.a}. This can potentially be used to identify the correct regularity in case $\gamma^1$ is unknown. One approach can be to minimize the global variance over $\gamma^1$. This is left as future work.

\begin{figure}[tbhp]
    \centering
    \subfloat[prior samples]{\label{fig:4.1inc_2.c}\includegraphics[width=0.45\textwidth]{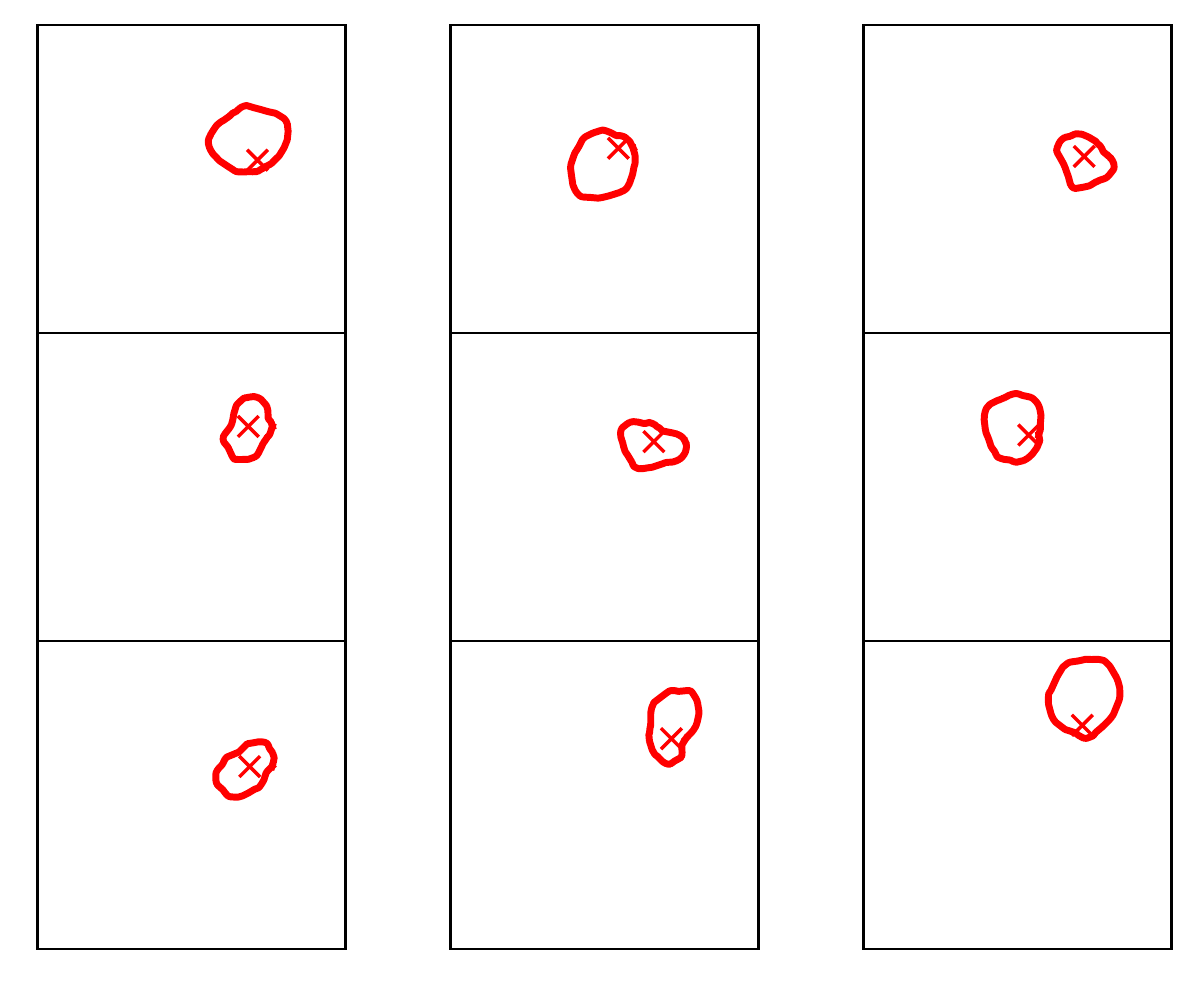} }
    \subfloat[posterior samples]{\label{fig:4.1inc_2.d}\includegraphics[width=0.45\textwidth]{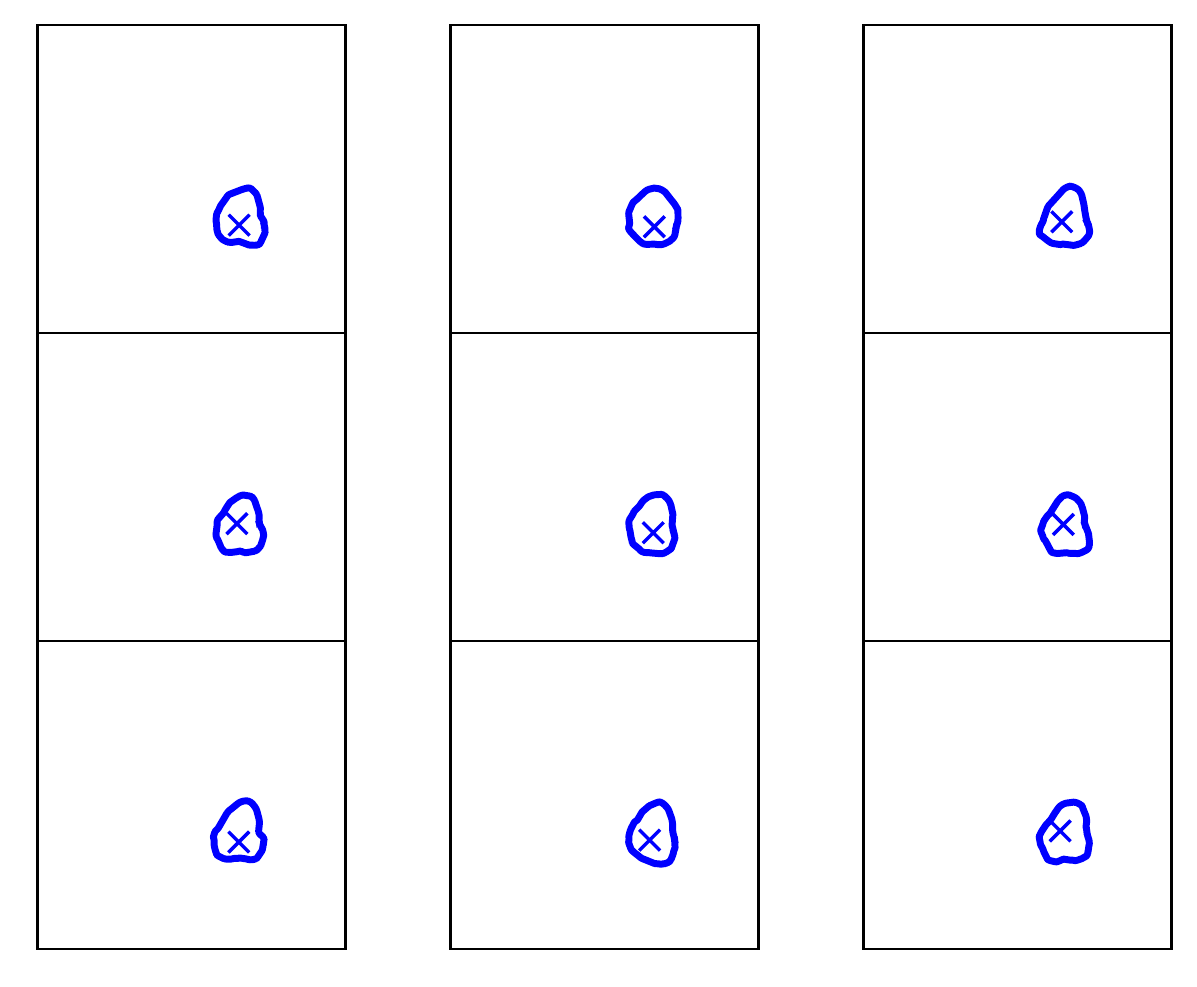} } \\
    \subfloat[using true $\gamma^{1}$ ($\gamma^{1}=2$)]{\label{fig:4.1inc_2.a}\includegraphics[width=0.35\textwidth]{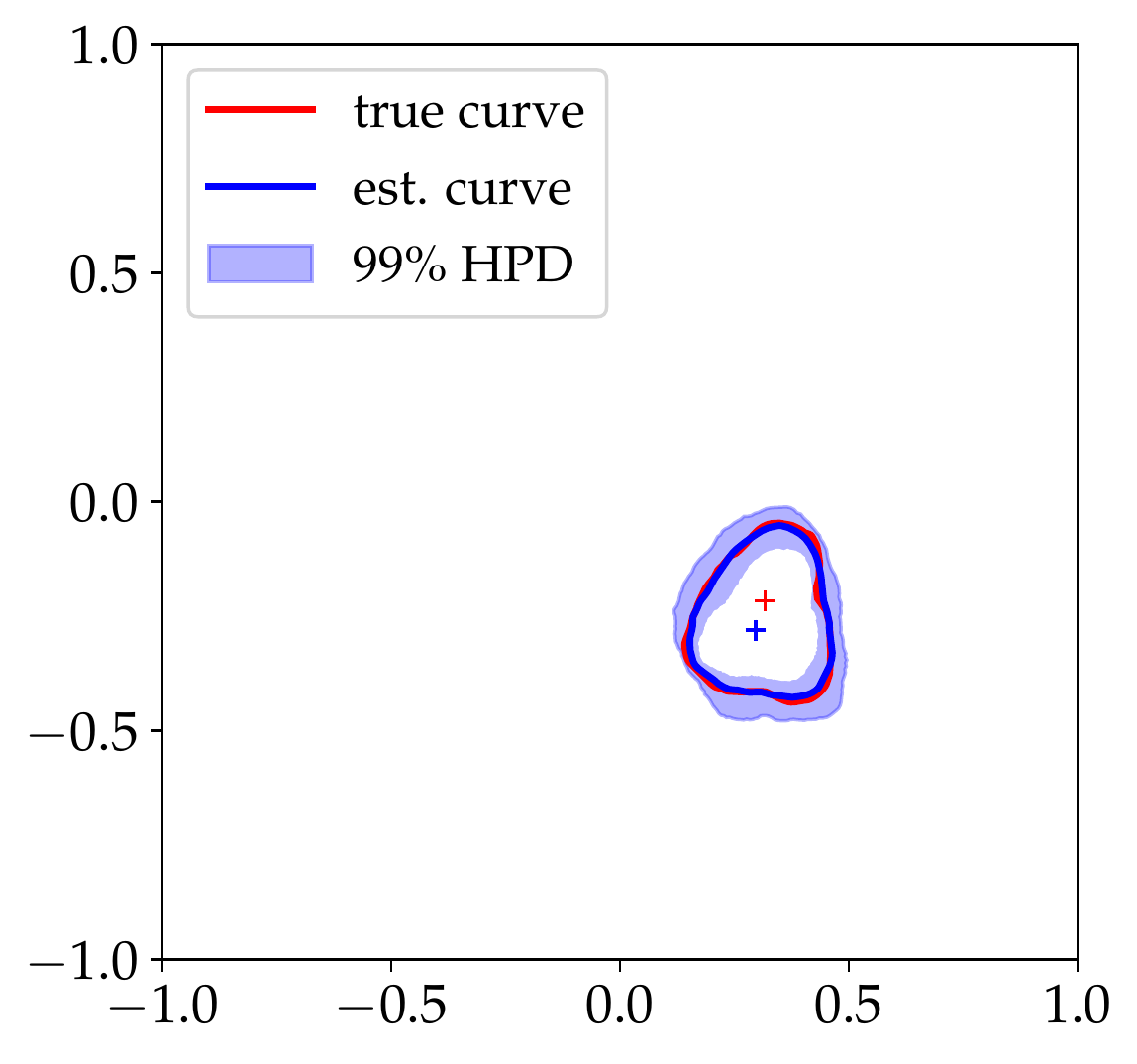} } \hspace{5mm}
    \subfloat[using wrong $\gamma^{1}$ ($\gamma^{1}=3$)]{\label{fig:4.1inc_2.b}\includegraphics[width=0.35\textwidth]{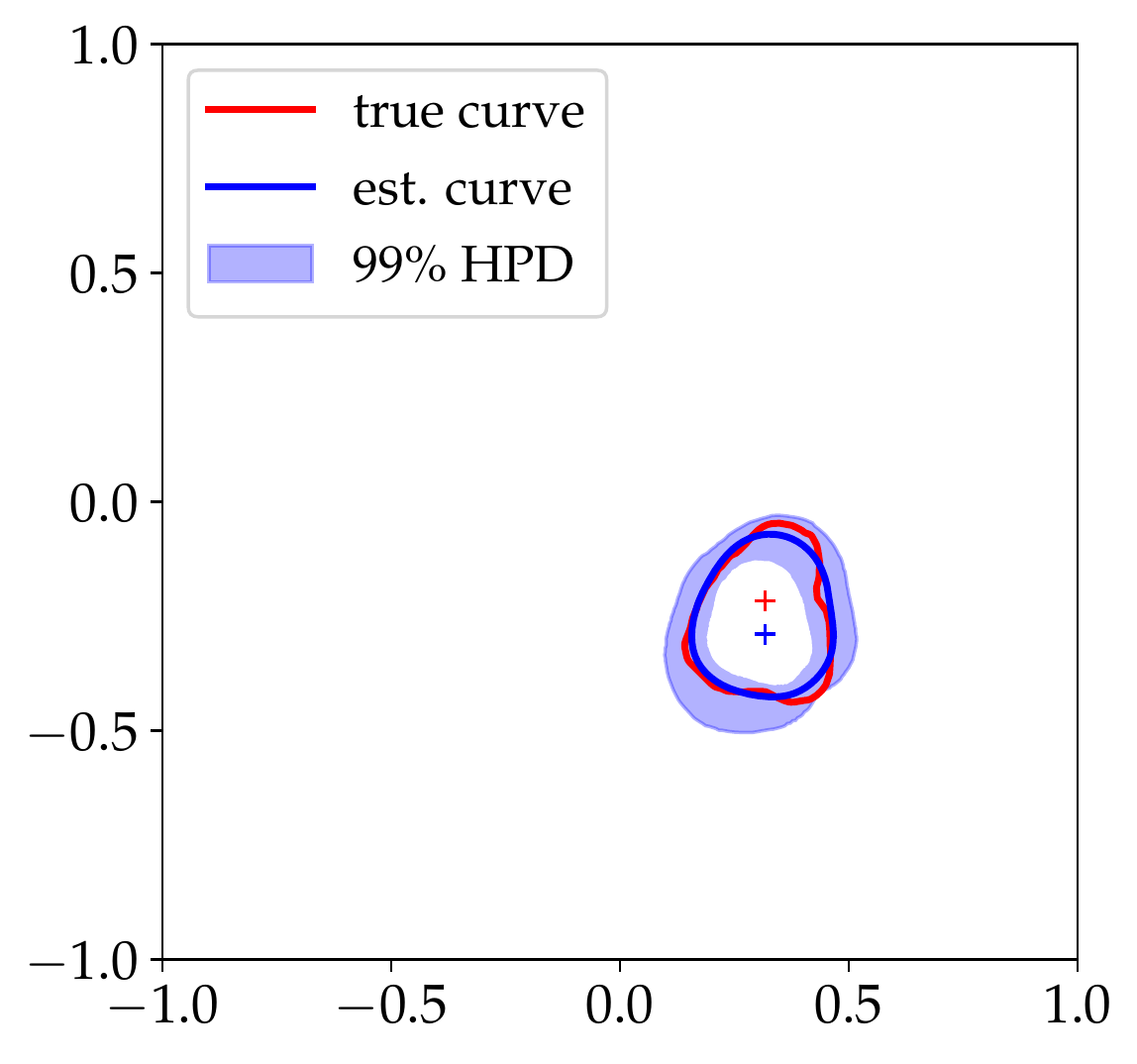} }
    \caption{\coyq{Illustration of the performance of our method for for different values of the regularity parameter $\gamma^{1}$. The global variances, $\mathbb E \| \xi_1 - \mathbb E \xi_1 \|^2_H$, in the estimations for \Cref{fig:4.1inc_2.a,fig:4.1inc_2.b} are 0.0561 and 0.198, respectively.}}
\label{fig:4.1inc_2}
\end{figure}



\subsubsection{Multiple Inclusions} \label{sec:4.1.2.multi}
In this section we consider \coyq{attenuation fields} $\alpha$ that contain multiple inclusions. We construct the ground truth density fields using the noisy star-shaped prior \coyq{given in \eqref{2.eq.noisy_star}
with $N_{\text{inc}}=3$. We let $\xi_0\sim \mathcal N(0,Q_{\gamma^0,\tau^0})$ with $\gamma^0 = 2.5$ and $\tau^0 = 50$, and set $\xi_i \sim \mathcal N(0,Q_{\gamma^i,\tau^i})$ with $\gamma^i = 3$ and $\tau^i = 1$ for $i=1,\dots,N_{\text{inc}}$. We} expect the inclusions to contain smooth boundaries. Drawing samples from $\alpha^{\text{noisy}} (x)$ requires sampling-and-elimination step to ensure that inclusions are inside the domain and they do not collide. We truncate the KL expansion for $\xi_0$ after 200 terms and for $\xi_i$, $i=1,\dots,N_{\text{inc}}$, after 100 terms. \coyq{In this test, we} \copc{use the noise level~$1\%$.} \coyq{The CT scan geometry is the same as in the test problem in \Cref{sec:4.1.gen}. In \Cref{fig:5.multi_inc}, we show a noisy sinogram as well as the true image. In our method, Stage 1 serves for finding the bound boxe and the center to each inclusion. Then, in Stage 2 the Gibbs samplers according to \Cref{alg:3.stage_2} are run parallel for each inclusion. We collect $10^4$ samples from the posterior distribution with an addition of $10^3$ samples in the burn-in stage.} 

\begin{figure}[tbhp]
    \centering
    \subfloat[sinogram $\boldsymbol{y}$]{\label{fig:5.muti_inc.a}\includegraphics[width=0.3\textwidth]{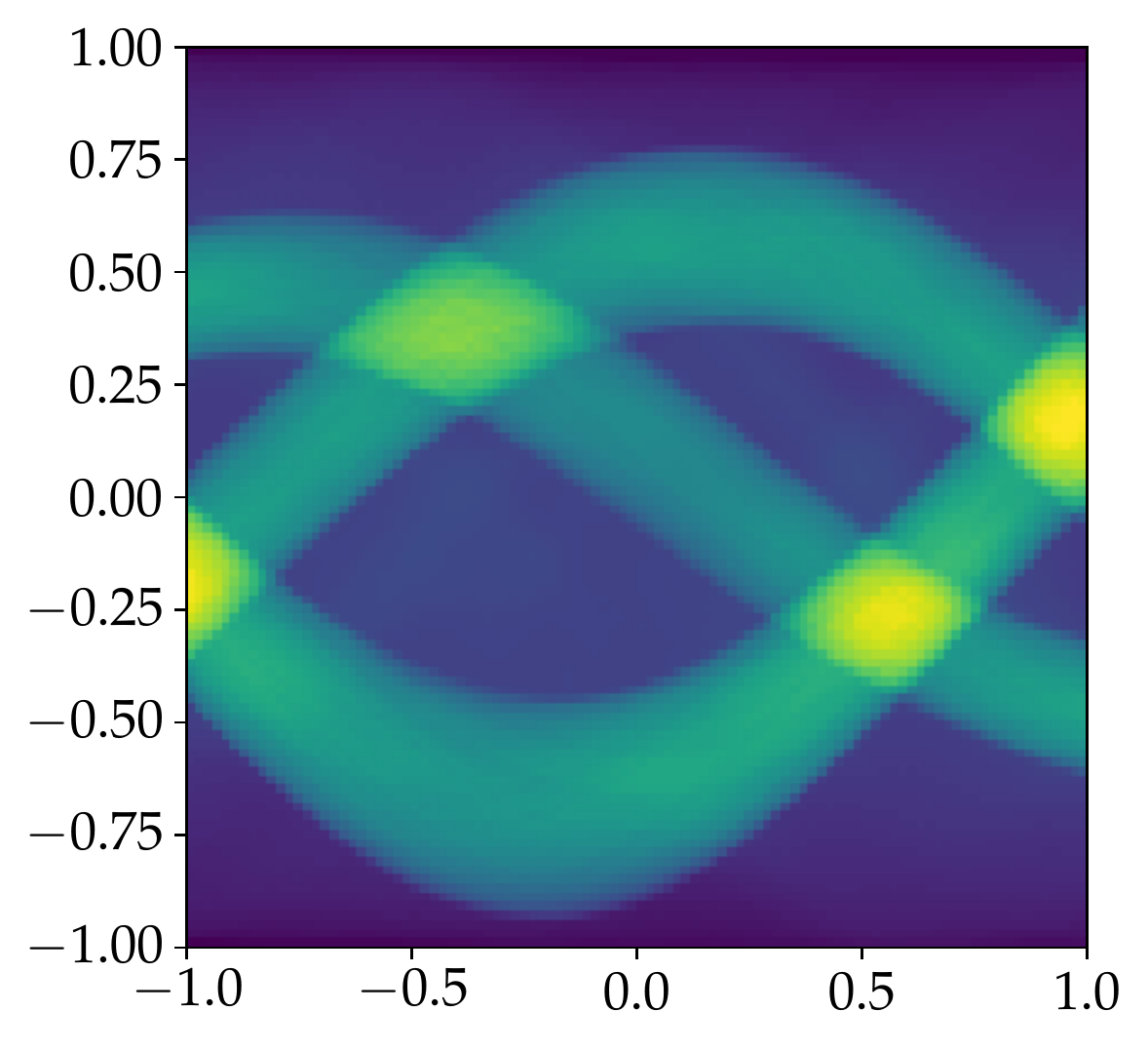} }
    \subfloat[true image $\alpha(\boldsymbol{\xi})$]{\label{fig:5.muti_inc.b}\includegraphics[width=0.3\textwidth]{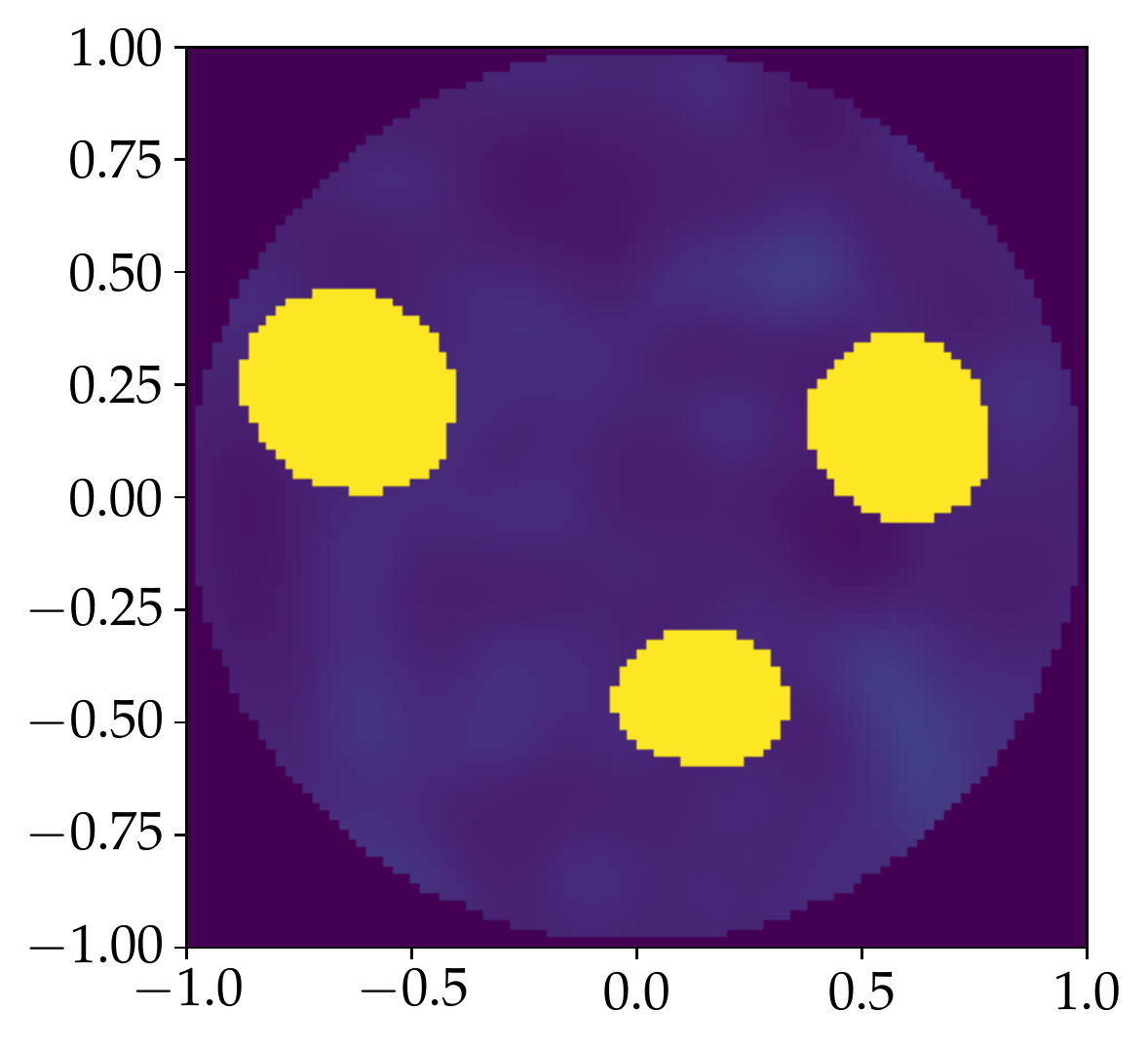} }
    \subfloat[estimated boundaries]{\label{fig:5.muti_inc.c}\includegraphics[width=0.3\textwidth]{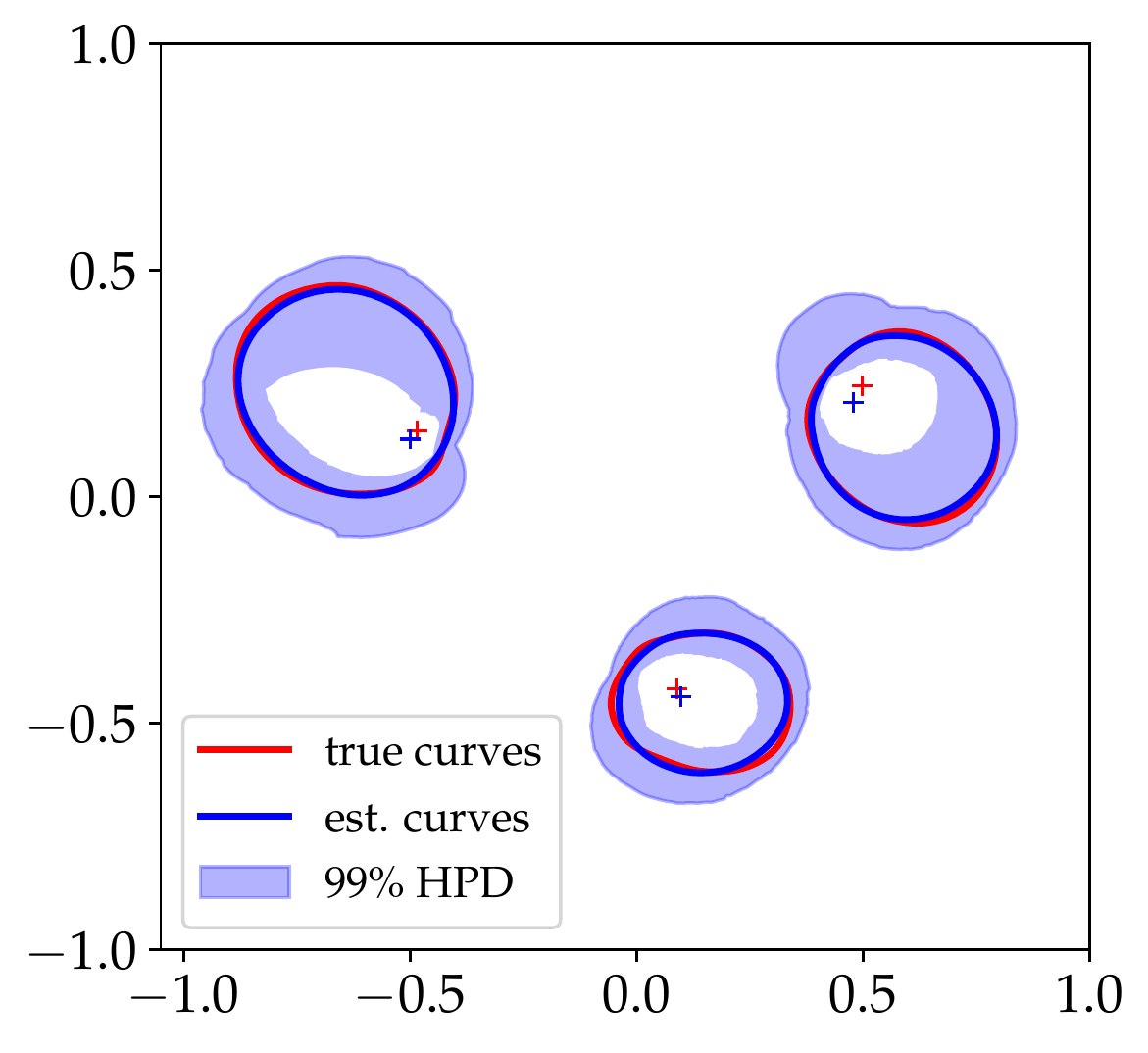} }
    \caption{Illustration of the uncertainty in detecting the boundary for an image with multiple inclusions.}
\label{fig:5.multi_inc}
\end{figure}

The estimated boundaries are presented in \Cref{fig:5.muti_inc.c}. We notice that the centers of all inclusions are estimated with high accuracy. Furthermore, the posterior mean for the boundary of the inclusion provides a precise estimate for the true boundaries. Note that the true inclusions have different size, shape and orientation (e.g., one of the inclusions is more elongated). Although the prior distribution is set identical for all inclusions, the method \copc{is able to} find such subtle differences \delete{automatically}.

We \copc{list the} global variance of $\xi_i$, $i=1,2,3$, in \Cref{tab:4.multi_inc_var}. The \coyq{HPD} band illustrated in \Cref{fig:5.muti_inc.c} is compatible with the estimated global variances. We report that \copc{we} find similar results for images with at most 5 inclusions. Given a good initial guess provided from Stage 1, we expect the method to perform as well for any number of inclusions.

\begin{table}[tbhp]
{\footnotesize
  \caption{Estimation of centers and variances.} \label{tab:4.multi_inc_var}
\begin{center}
\setlength{\tabcolsep}{10pt} 
\renewcommand{\arraystretch}{1.5}
\begin{tabular}{|c|c|c|c|c|c|c|} \hline
 & \multicolumn{6}{|c|}{\Cref{fig:5.muti_inc.c}} \\ \hline
Inclusion number & \multicolumn{2}{|c|}{$i=1$} & \multicolumn{2}{|c|}{ $i=2$} & \multicolumn{2}{|c|}{ $i=3$} \\ \hline
True centers & \multicolumn{2}{|c|}{$(-0.484 , 0.146)$} & \multicolumn{2}{|c|}{ $(0.499, 0.246)$} & \multicolumn{2}{|c|}{ $(0.090, -0.423)$} \\ \hline
Estimated centers & \multicolumn{2}{|c|}{$(-0.500, 0.128)$} & \multicolumn{2}{|c|}{ $(0.480, 0.209)$} & \multicolumn{2}{|c|}{ $(0.101, -0.439)$} \\ \hline
Variance in $\xi_i$ & \multicolumn{2}{|c|}{ 0.083 } & \multicolumn{2}{|c|}{ 0.050 } & \multicolumn{2}{|c|}{ 0.065 } \\ \hline
\end{tabular}
\end{center}
}
\end{table}

\subsubsection{Sparse and limited angle imaging} \label{sec:4.1.4.sparse}



In this section we estimate the boundary of a single inclusion in a sparse \coyq{and limited} angle imaging configuration. We compare \copc{measurement geometries} where the number of angles is $N_{\theta} = 10$ (compared to $N_{\theta} = 100$ in \Cref{sec:4.1.1.single,sec:4.1.2.multi}) and we set \copc{the angular range to be $[0,\theta_{\text{max}})$ with} $\theta_{\text{max}} \in \{180^{\circ} , 90^{\circ} , 45^{\circ} \}$. The interval $[0,\theta_{\text{max}})$ is uniformly discretized into imaging angles $\theta_j$, $j=1,\dots,N_{\theta}$. 
The regularity of the inclusion is chosen to be $\gamma^1=2.5$. The rest of the parameters for the forward problem and the sampling methods are \copc{identical to those in} \Cref{sec:4.1.1.single}.


\begin{figure}[tbhp]
    \centering
    \subfloat[$\theta_{\text{max}} = 180^{\circ}$]{\label{fig:4.sparse.g}\includegraphics[width=0.32\textwidth]{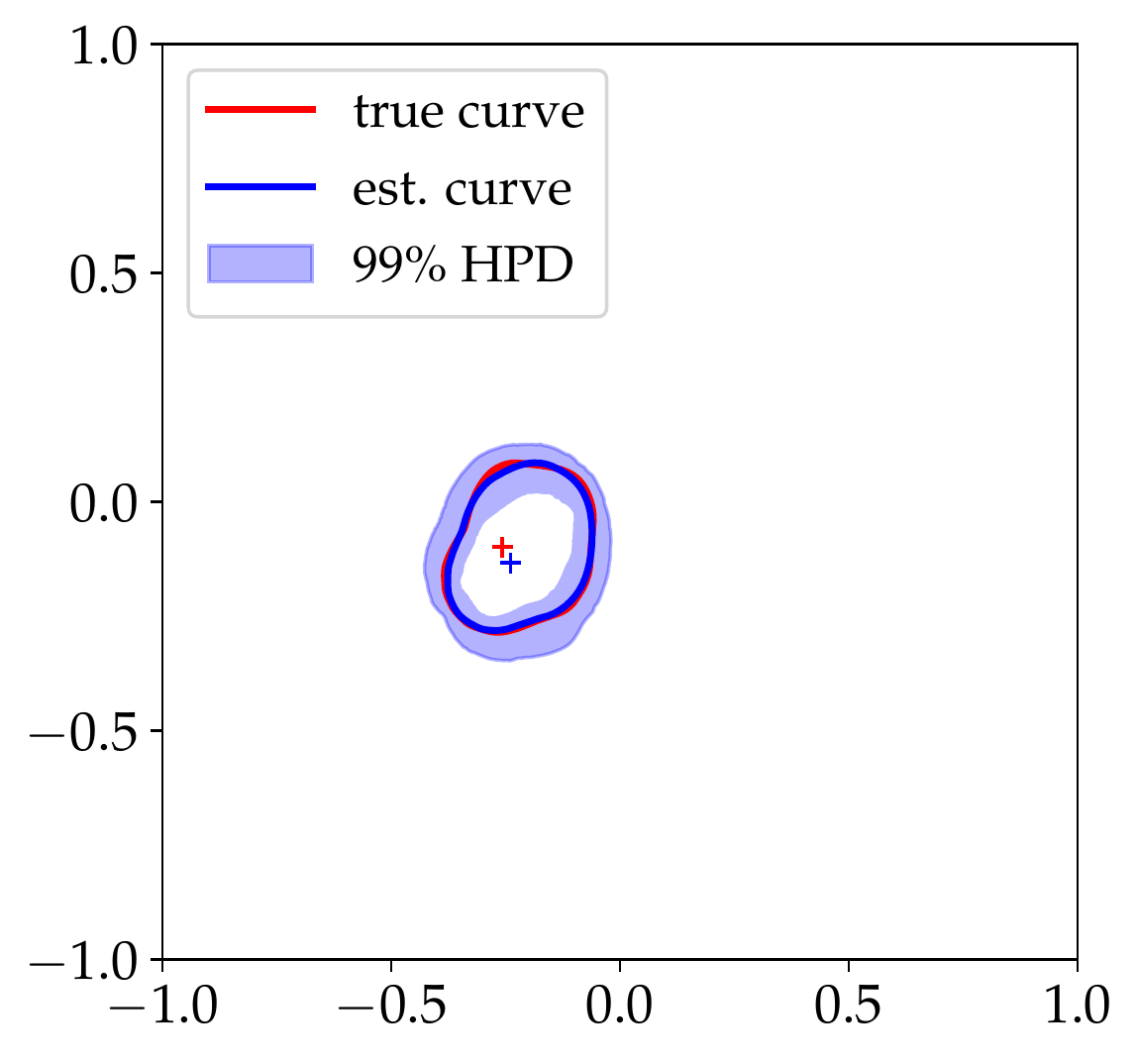} }
    \subfloat[$\theta_{\text{max}} = 90^{\circ}$]{\label{fig:4.sparse.h}\includegraphics[width=0.32\textwidth]{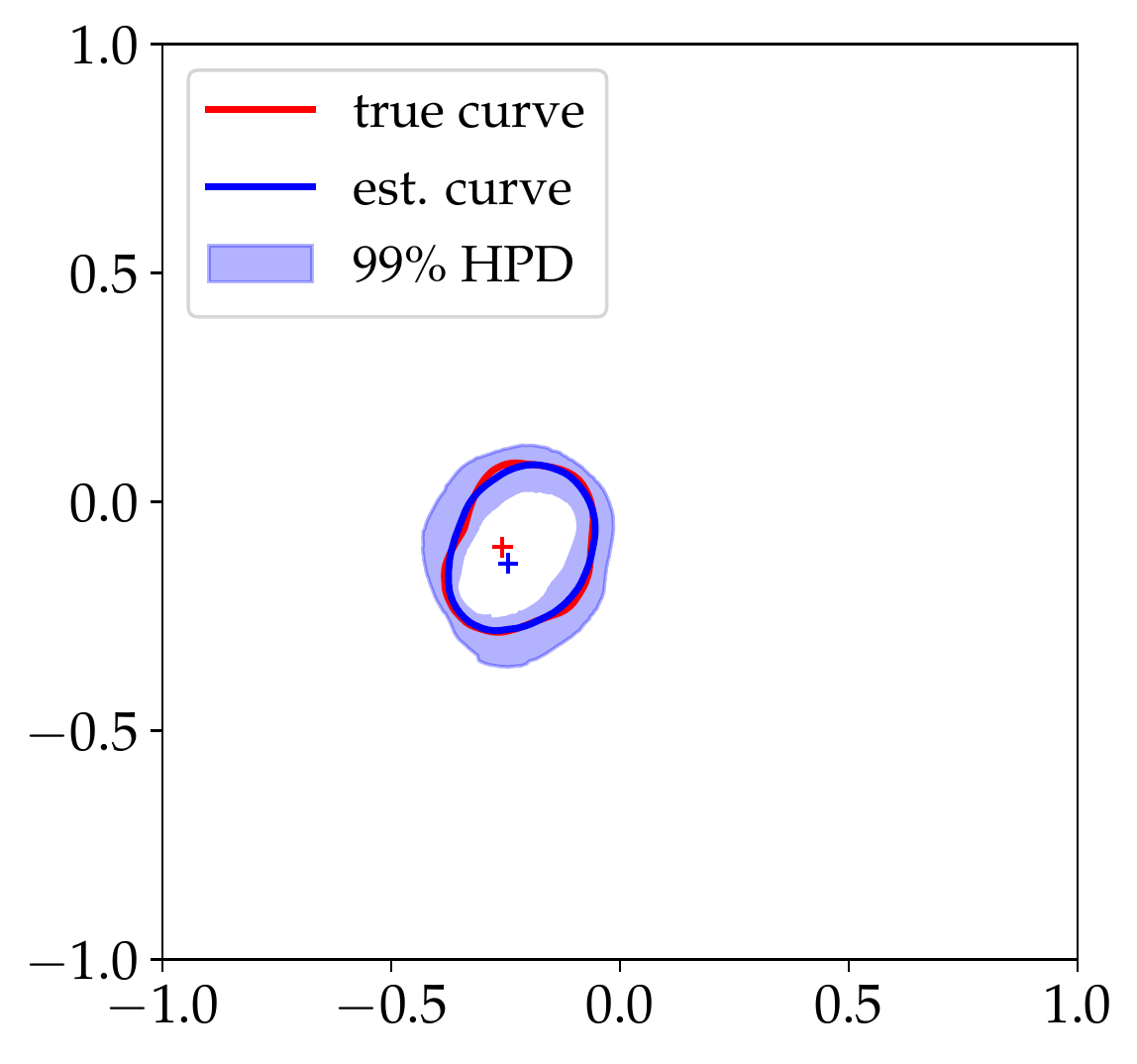} }
    \subfloat[$\theta_{\text{max}} = 45^{\circ}$]{\label{fig:4.sparse.i}\includegraphics[width=0.32\textwidth]{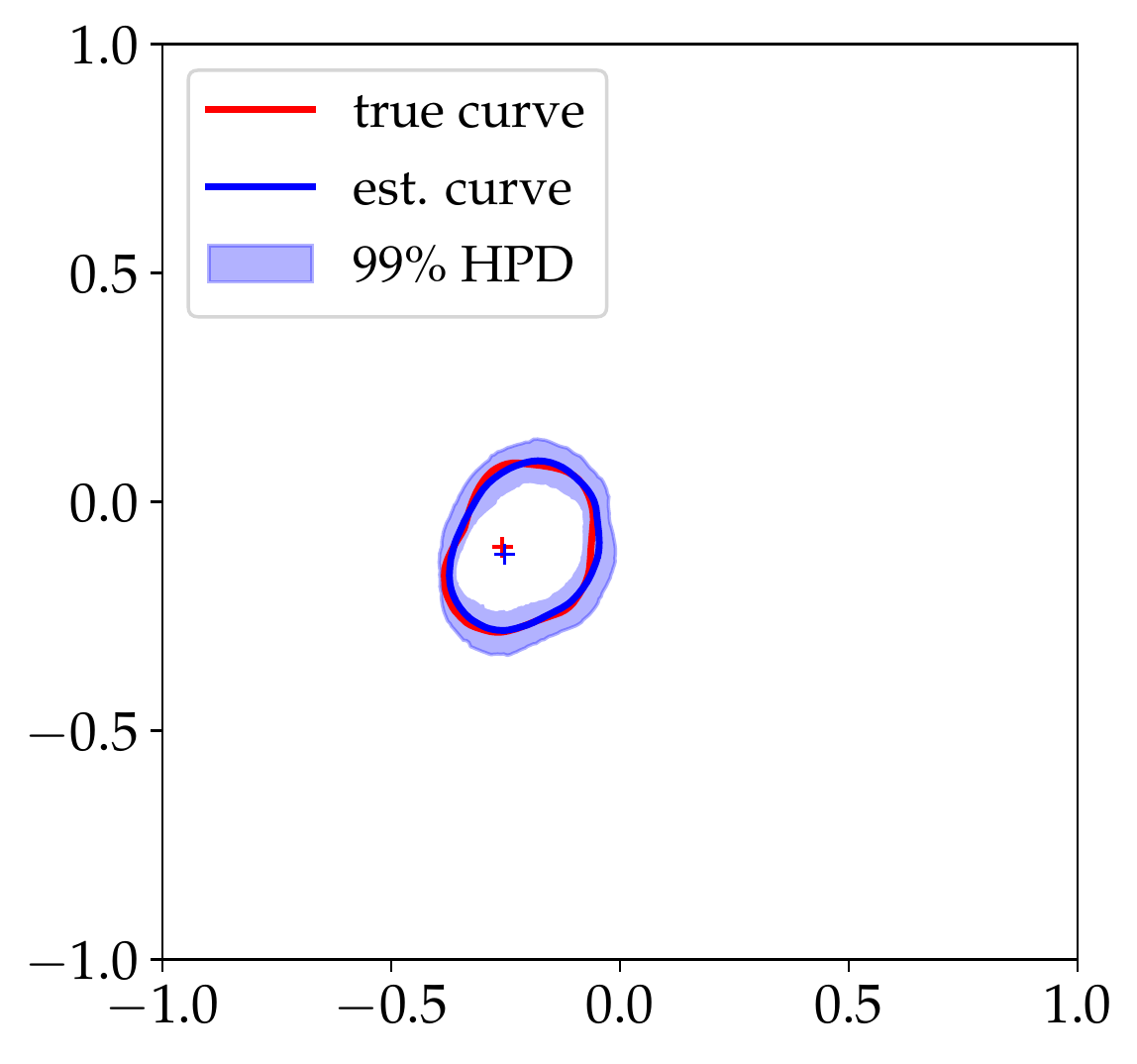} } \\
    \subfloat[$\theta_{\text{max}} = 180^{\circ}$]{\label{fig:4.sparse.j}\includegraphics[width=0.32\textwidth]{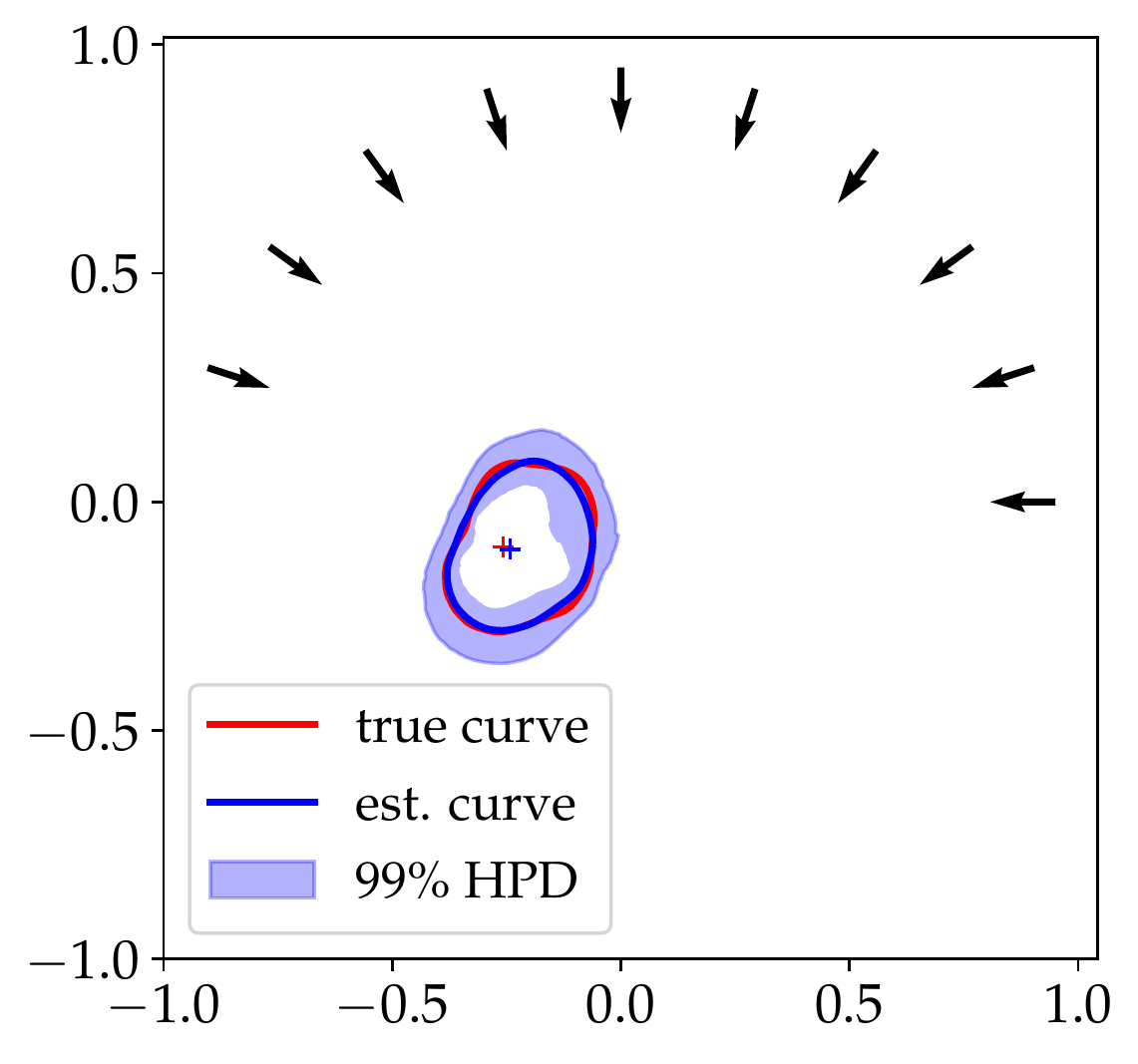} }
    \subfloat[$\theta_{\text{max}} = 90^{\circ}$]{\label{fig:4.sparse.k}\includegraphics[width=0.32\textwidth]{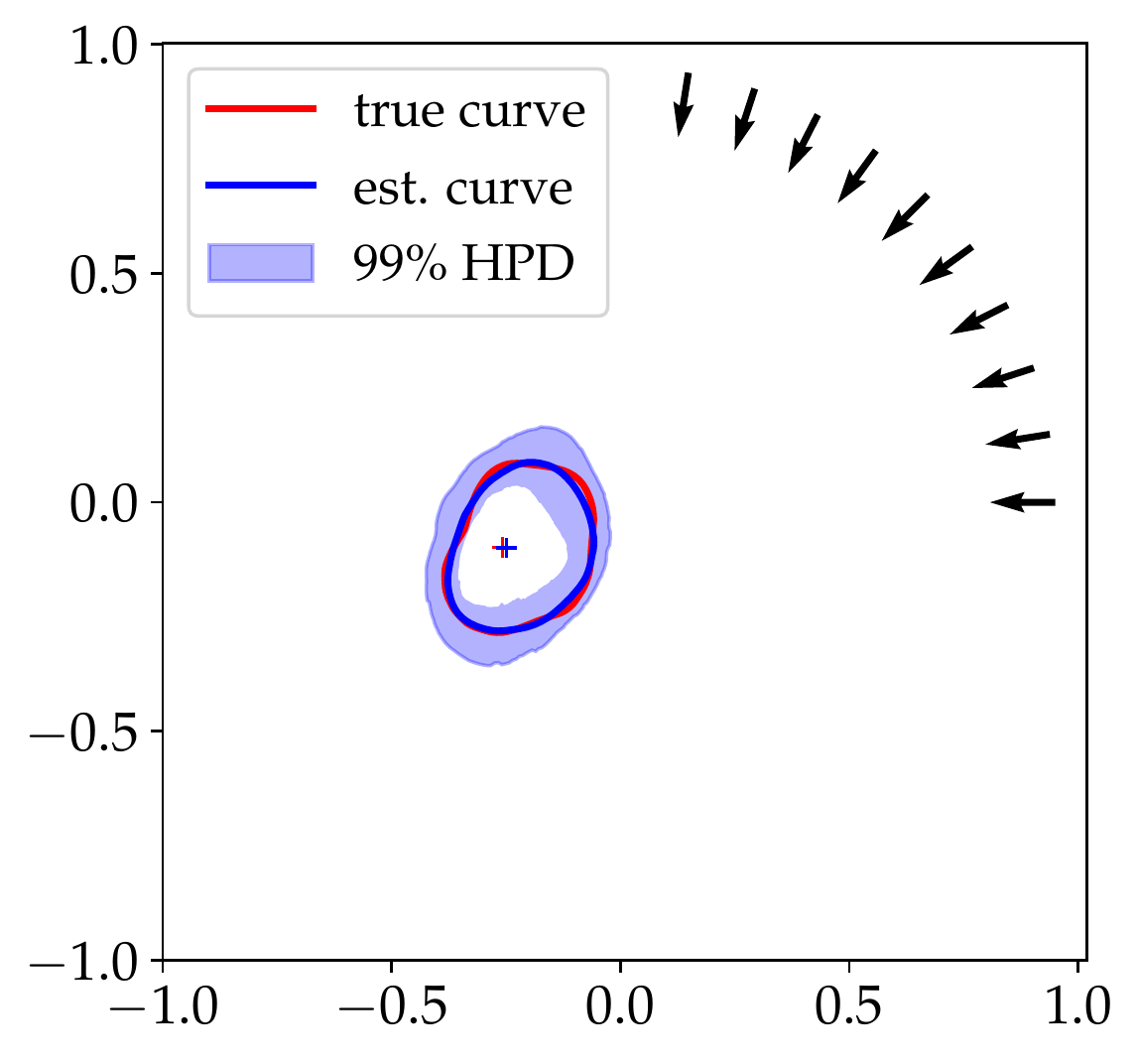} }
    \subfloat[$\theta_{\text{max}} = 45^{\circ}$]{\label{fig:4.sparse.l}\includegraphics[width=0.32\textwidth]{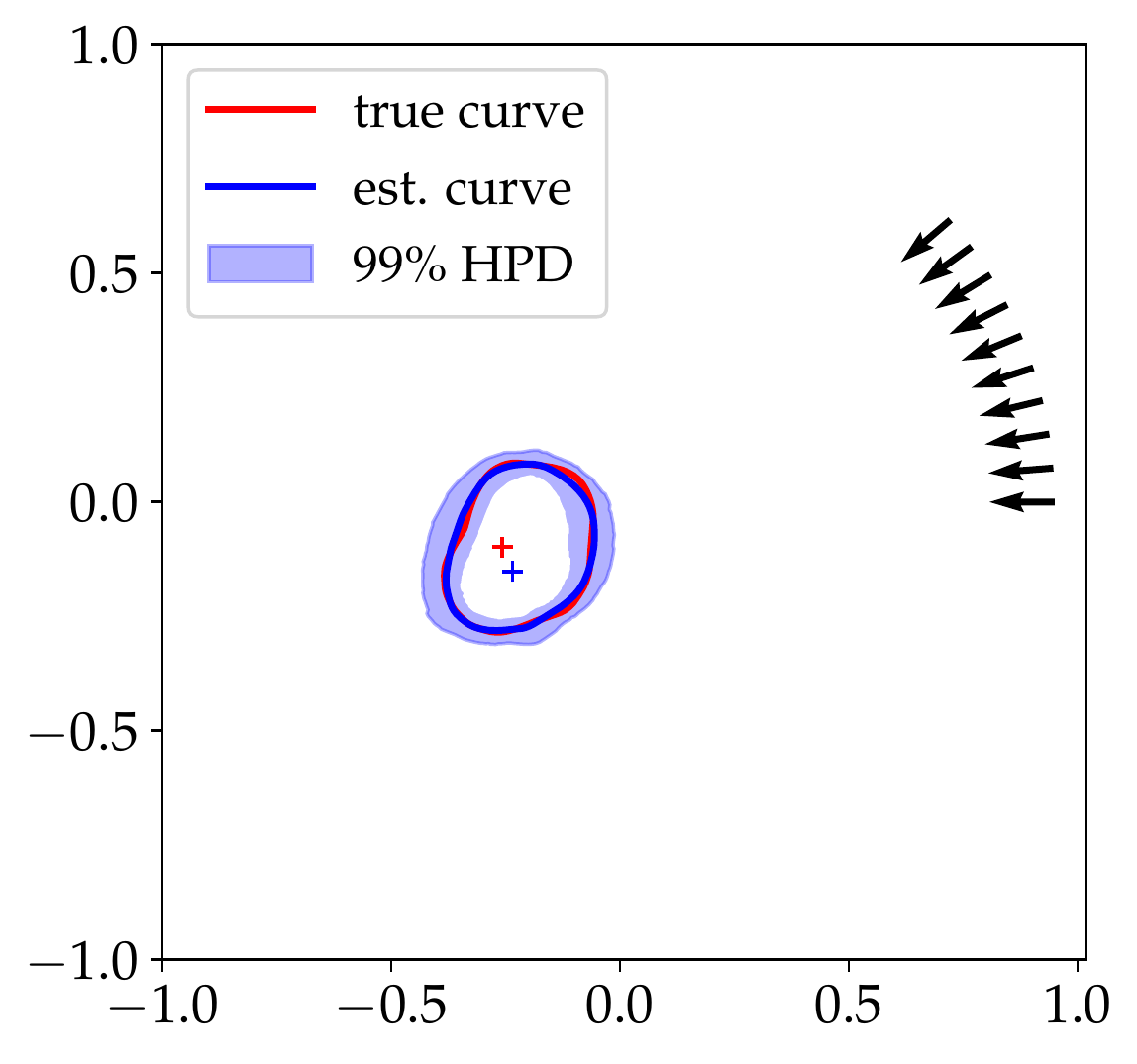} }
    \caption{Example of estimating the boundary of a single inclusion with sparse and limited angle imaging. \coyq{First row:\ a full set of angles with $N_{\theta}=100$; second row:\ sparse angles with $N_{\theta}=10$} \copc{as illustrated by the arrows.}}
\label{fig:4.sparse}
\end{figure}

To understand the effect of the sparsity and imaging angles on estimations, \coyq{we compare the results by using our method to the case where a full set of \coyq{angles} ($N_{\theta}=100$) is available with the ones from sparse and limited angle cases, see \Cref{fig:4.sparse}.} We notice that the uncertainty in the estimation of the boundary of the inclusion is significantly larger than the previous test cases. As we move from left to right, we cannot notice a qualitative difference in the amount of uncertainty in the estimated boundary. We notice the increase in the width of the uncertainty band as we move from
\copc{the top figures to the bottom figures.}
Therefore, the method is significantly more sensitive to the number of observations than the number of imaging angle.


We remark that in all test cases in this section \correct{the estimated boundary $(\bar \xi_i^1, \bar c_i^1)$} provides \copc{excellent approximations to the true boundary and the true center} of the star-shaped inclusion. Furthermore, we report that the overall behavior of the global variances were comparable with the uncertainty presented in \Cref{fig:4.sparse} as seen in \Cref{tab:4.sparse_var}.

\begin{table}[tbhp]
{\footnotesize
  \caption{Estimation of the global variance.} \label{tab:4.sparse_var}
\begin{center}
\renewcommand{\arraystretch}{1.5}
\begin{tabular}{|c|c|c|c|c|c|c|} \hline
 & \Cref{fig:4.sparse.g} & \Cref{fig:4.sparse.h}& \Cref{fig:4.sparse.i}& \Cref{fig:4.sparse.j} & \Cref{fig:4.sparse.k} & \Cref{fig:4.sparse.l} \\ \hline
$\mathbb E \| \xi_1 - \mathbb E \xi_1 \|^2_H$ & 0.0538  &  0.0485 & 0.0312 & 0.0571 & 0.0493 & 0.0320  \\ \hline
\end{tabular}
\end{center}
}
\end{table}

\subsection{Lotus Root} \label{sec:4.2}
In this section we apply our method \copc{to} tomographic X-ray data of a lotus root filled with attenuating objects \coyq{from the open data sets in \cite{bubba_tatiana_a_2016_1254204,1609.07299}. The data is taken over 360 angles around the object with the fan-beam geometry. We apply filtered back-projection (FBP) to reconstruct the image $\alpha$, which is shown in \Cref{fig:4.lotus.a} and considered as the ground truth. The target of the test is to reconstruct the boundary of a piece of circular chalk (made of calcium) placed inside the lotus root.}

 \begin{figure}[tbhp]
    \centering
    \subfloat[\coyq{the ground truth}]{\label{fig:4.lotus.a}\includegraphics[width=0.32\textwidth]{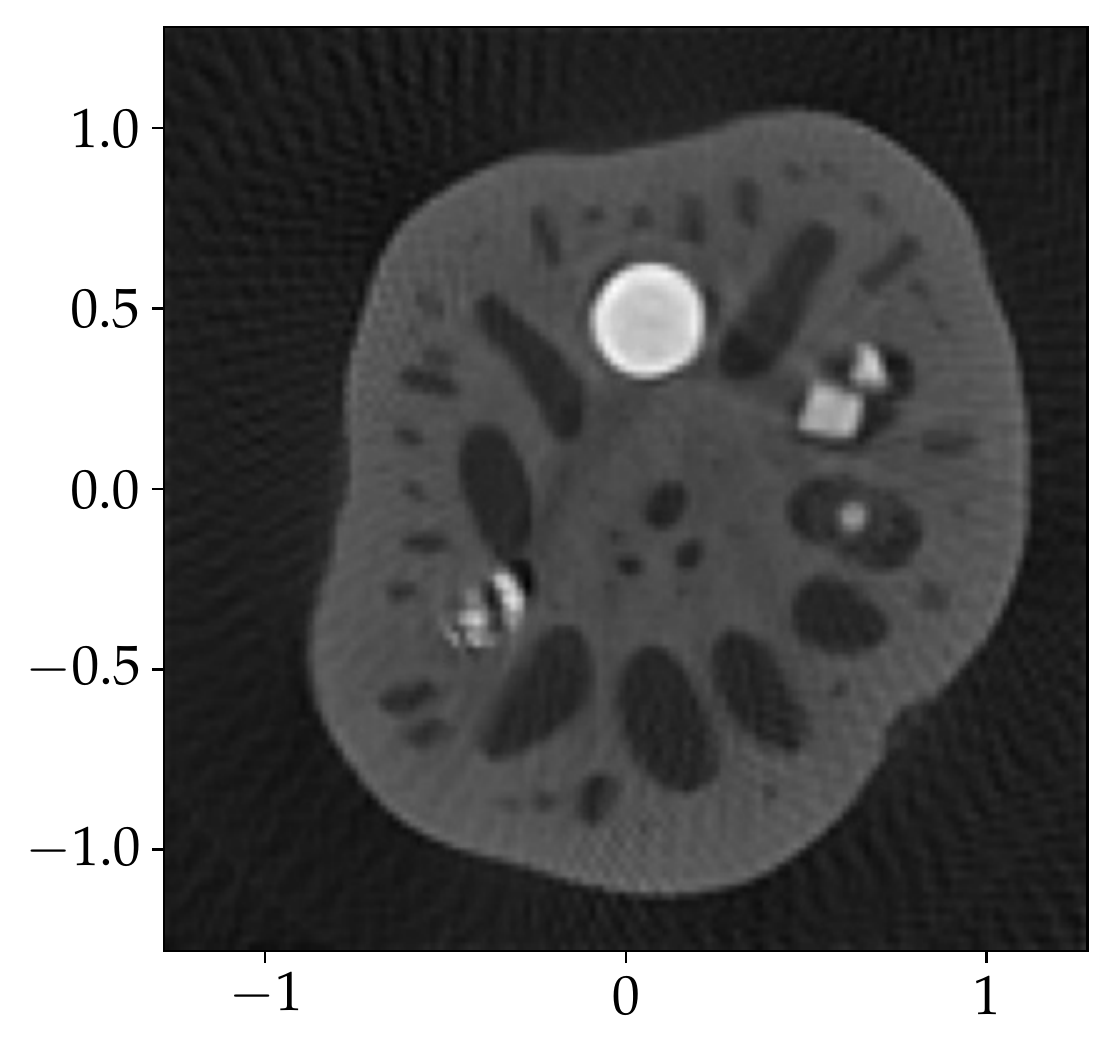} }\hspace{5mm}
    \subfloat[\coyq{the estimated boundary}]{\label{fig:4.lotus.b}\includegraphics[width=0.32\textwidth]{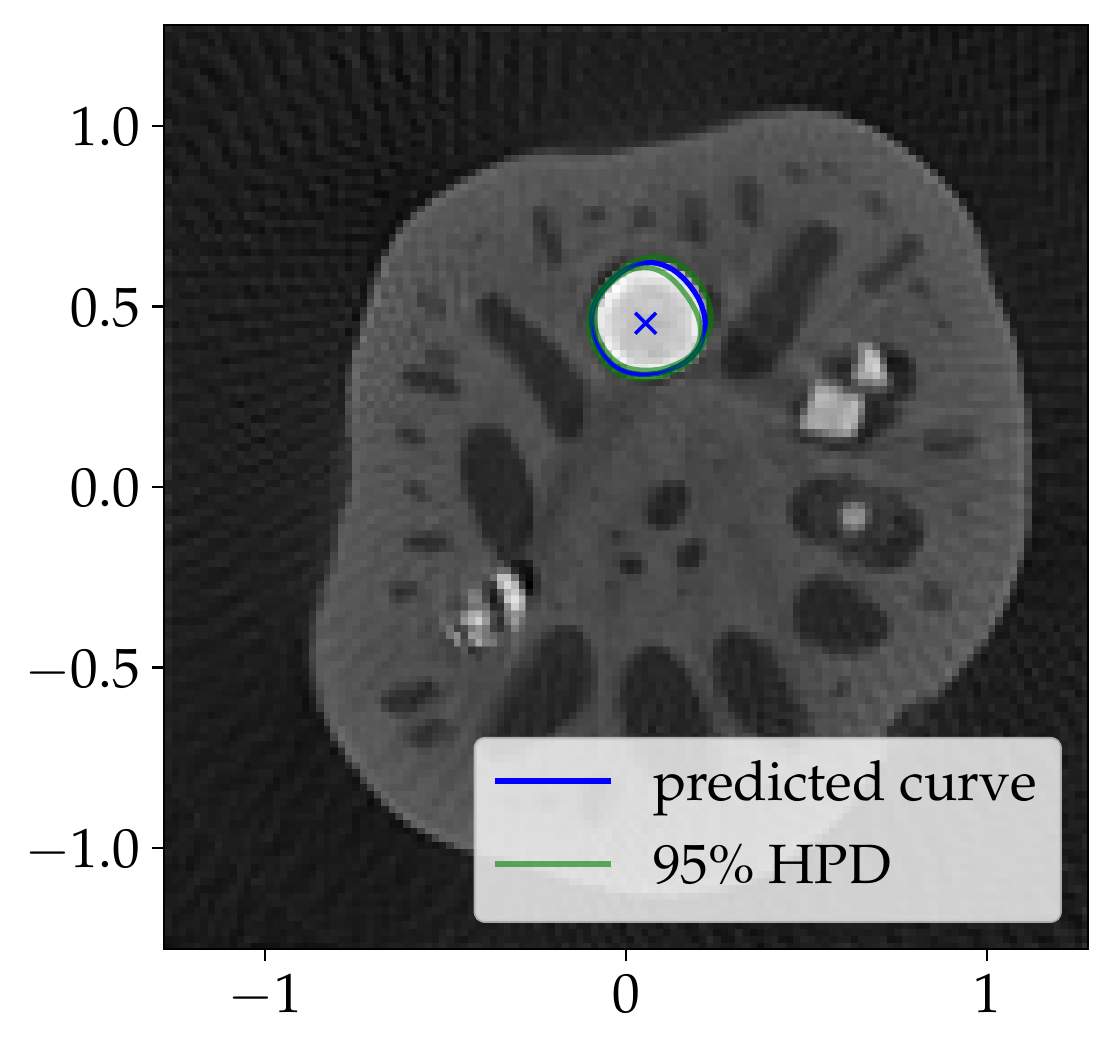} }
    \caption{Prediction of the boundary of a circular \copc{piece of} chalk in a 2D slice of a lotus root.}
\label{fig:4.lotus}
\end{figure}

\coyq{To test the performance of our method for real data, we down-sample the data with 100 equidistantly spaced angles in $[0,180^\circ)$ as $\boldsymbol{y}$. We assume that the noise level is $1\%$, and set the foreground attenuation to be $a^+ = 0.025$ and the background attenuation to be $a^- = 0.001$. In Stage 1 of our method we set the correlation length for the Mat\'ern-Whittle field to be $\ell = 0.02$ and the regularity parameter $\gamma = 3$. The rest of the modelling parameters are chosen identical to the previous tests.}

The estimated curve for the circular chalk is presented in \Cref{fig:4.lotus.b}. \coyq{Note that our method does not provide a reconstructed image for the whole domain. Here, we keep the ground truth behind the curve estimation as a visualization aid. It is clear that} the curve precisely follows the outline of the circular chalk. Furthermore, the estimated center appears approximately at the center of the inclusion, \coyq{which is consistent \copc{with the fact} that the center of a star-shaped inclusion with a perfect circular boundary is the same as the center of the mass. The HPD band around the boundary of the inclusion suggests that the reconstruction is very  accurate.} 

\section{Conclusions} \label{sec:5}

This work presents an infinite dimensional Bayesian framework for the X-ray CT problem for
\copc{goal-oriented}
estimation and uncertainty quantification of inclusion boundaries. \coyq{The proposed method reconstructs the boundaries of inclusions} with constant attenuation that can be represented as a star-shaped inclusion on a smoothly varying background. Furthermore, we provide \coyq{a HPD band around the boundary to quantify the uncertainty of the reconstruction.} The method is carried out in two stages. In Stage 1 we identify approximate locations of the inclusions by sampling the posterior constructed with the level set prior. This stage decomposes the image into regions with a single inclusion. Stage 2 comprises sampling from the posterior distribution constructed by using the star-shaped prior. The decomposition of the image \coyq{in Stage 1 guarantees the well-posedness of the reconstruction problem.}  

The numerical results show that \coyq{our method reconstructs the boundaries of the inclusions accurately and provides a reliable tool to quantify the uncertainty in the prediction.} 
Furthermore, the method consistently performs well in detecting inclusions in challenging X-ray CT scenarios (e.g., for sparse and limited angle imaging). \copc{Our results from
applying the method to a real data,
in the form of X-ray measurements} of a lotus root filled with a circular \copc{pieces of} chalk, suggest that this method can be extended to real world applications.


\section*{Acknowledgements} \correct{We thank Dr.\ Felipe Uribe for his help with the experiments in \Cref{sec:intro}. We would like to also thank the reviewers for their thoughtful comments and efforts towards improving our manuscript.} 

\appendix

\section{More on The Mat\'ern-Whittle Covariance} \label{sec:a.1}
Recall $D\subset \mathbb R^2$ is a bounded region with Lipschitz boundary. We take $H = L^2(D)$. The covariance function of two points $x,y\in D$ for the Mat\'ern-Whittle distribution \cite{nla.cat-vn2458246,rasmussen2003gaussian} is given by
\begin{equation} \label{eq:2.matern}
    \coyq{q_{\sigma,\nu,\ell}(x_1,x_2) = \sigma^2 \frac{2^{1-\nu}}{\Gamma(\nu)} \left( \frac{\|x_1-x_2\|_2}{\ell} \right)^\nu K_{\nu} \left( \frac{\|x_1-x_2\|_2}{\ell} \right).}
\end{equation}
Here, $\|\cdot\|_2$ is the Euclidean norm, $\ell>0$ is the spatial correlation length, $\nu>0$ is the smoothness parameter, $\sigma^2>0$ is the variance of the value of the field (amplitude scale). Furthermore, $K_{\nu}$ is the modified Bessel function of the second kind of order $\nu$ \cite{Lindgren2011,rasmussen2003gaussian}.

For $\nu = 1/2$ \eqref{eq:2.matern} reduces to the exponential covariance $q_{\sigma,1/2,\ell} = \sigma^2 \exp(-d/\ell)$, with $d=|x_1-x_2|$. For larger $\nu$ \copc{the} smoothness of $q_{\sigma,\nu,\ell}$ increases. One way to see this is that for $\nu = 1/2 + p$ with $p\in \mathbb N^+$ \eqref{eq:2.matern} can be written as a product of an exponential with a polynomial of order $p$ \cite{abramowitz1988handbook,rasmussen2003gaussian}. Therefore, larger $p$ contributes to higher regularity. As $p\to \infty$ the polynomial tends to an exponential function and $q_{\sigma,\nu,\ell}$ converges to the squared exponential (Gaussian) covariance function $q_{\sigma,\infty,\ell} = \sigma^2 \exp(-d^2/(2\ell^2))$.

We can construct a discrete density function by \coyq{discretizing} $D$ and computing \eqref{eq:2.matern} for each pair of points. However, this approach \copc{gives} in a full covariance matrix for large correlation length. Inversion of such covariance matrices is challenging for inverse problem applications \cite{Lindgren2011,Roininen2014}.

An alternative approach is to formulate the covariance as a differential operator of a stochastic partial differential equation (SPDE). A detailed discussion on how to construct this SPDE is beyond the scope of this paper. Below we present a brief sketch of this construction and refer the reader to \cite{Roininen2014,zhang2004inconsistent} for a detailed discussion.

The main idea is to construct an $H$-valued Gaussian random variable $\xi$ from its Fourier expansion. Let $\psi$ be white noise on $D$, i.e. $\psi$ has zero mean with the covariance operator being the Dirac's delta function $q(x_1,x_2) = \delta(\|x_1-x_2\|_2)$. Subsequently, all Fourier modes in $\widehat \psi$ will be present in the Fourier transform of $\psi$. Furthermore, let $S(w)$ be the Fourier transform of $q_{\sigma,\nu,\ell}/\sigma^2$, known as the \emph{power spectrum}. Now we can define $\xi$ through its Fourier transformation by rescaling $\widehat \psi$ with $S(w)$ as
\begin{equation} \label{eq:2.fourier_white}
    \widehat \xi := \sigma \sqrt{ S(w) } \widehat \psi.
\end{equation}
By definition, $\xi$ is a Gaussian random variable distributed according to the covariance function $q_{\sigma,\nu,\ell}$. We recover $\xi$ by applying the inverse Fourier transform to \eqref{eq:2.fourier_white} and obtain
\begin{equation} \label{eq:2.spde}
        \frac{1}{\sqrt{b(\nu)\ell^2}} (I-\ell^2 \Delta)^{(\nu + 1)/2} \xi = \psi, \qquad
        b(\nu) = \sigma^2 \frac{4 \pi \Gamma(\nu + 1)}{\Gamma(\nu)}.
\end{equation}
Here $\Delta$ is the Laplace operator and the covariance operator of $\xi$ corresponding to the covariance function $q_{\sigma,\nu,\ell}$ is given by
\begin{equation} \label{eq:2.cov_op}
    Q_{\sigma,\nu,\ell} = b\ell^2 (I-\ell^2 \Delta)^{-\nu - 1}.
\end{equation}
A simplification of \eqref{eq:2.cov_op} is presented in \cite{Dunlop2016,zhang2004inconsistent} which takes the form
\begin{equation}
    Q_{\gamma,\tau} = (\tau^2 I- \Delta)^{-\gamma}.
\end{equation}
Here, $\tau = 1/\ell >0$ controls the correlation length and $\gamma = \nu + 1$ is the smoothness parameter (see \cite{Lindgren2011} for more detail). For the covariance operator \eqref{eq:2.cov_simple} to be well defined we need to impose proper boundary conditions. See \cite{Roininen2014} for more details on types of boundary conditions.

\section{Existence and Well-posedness of the posterior measure} \label{sec:a.2}
\correct{
In this section we show that the CT problem introduced \coyq{in \Cref{sec:2.likefunc,sec:3}} is well-posed. We assume that the attenuation field $\alpha$ is bounded and strictly positive, i.e., there exist $\alpha^+, \alpha^{-} \in \mathbb R^+$ such that $ \alpha^-< \alpha < \alpha^+ $, and thus $\alpha \in L^\infty(D)$. \coyq{We use $\mathcal S (D)$ to denote the space of all such attenuation fields $\alpha$.} Recall that for an attenuation field with a single inclusion, \coyq{$F_{\text{star}}$ maps functions from $X$ to $\mathcal S(D)$. Here, $X= H\times \mathbb R^{2}$ forms a separable Hilbert space. The following proposition from \cite{1930-8337_2016_4_1007} shows that $F_{\text{star}}$ is a continuous map. We refer the reader to \cite{1930-8337_2016_4_1007} for the full proof.}

\begin{proposition} \cite{1930-8337_2016_4_1007}\label{B1}
Let $F_{\text{star}} : X\to \mathcal S(D)$ be the star-shaped map and let $\{ \xi_1^{\epsilon} \}_{\epsilon>0}$ and $\{ c_1^{\epsilon} \}_{\epsilon>0}$ be a sequence of functions in $H$ and a sequence of points in $D$, respectively, such that $\|\xi_1 - \xi_1^\epsilon \|_\infty \to 0$ and $\| c_1-c_1^{\epsilon} \|_2 \to 0$. \coyq{Then, we have} $F_{\text{star}}[ (\xi_1^\epsilon,c_1^\epsilon)] \to F_{\text{star}}[ (\xi_1,c_1) ]$ in measure.
\end{proposition}

Let $\mathcal A$ be the Borel $\sigma$-algebra constructed on $X$ with respect to the norm
\begin{equation}
    \coyq{ \| (\xi_1,c_1) \|_{X} := \| \xi_1 \|_{H} + \| c_1 \|_2.} 
\end{equation}
To define a probability measure on $X$ we assume that $\xi_1$ and $c_1$ are independent random variables such that their joint probability measure takes the form $\mu_0 = \mu^1_0 \otimes \mu^2_0$, with $\mu^1_0 = \mathcal N (0,\mathcal C)$ a Gaussian measure on $H$ and $\mu^2_0$ the Lebesgue measure on $D$. Now the triplet $(X,\mathcal A,\mu_0)$ forms a probability space. Furthermore, We define the probability measure on $\mathcal S(D)$ to be the push-forward measure $\mu \circ F_{\text{star}}^{-1}$.

To show Lipschitz-Hellinger well-posedness of the CT inverse problem with the star-shaped prior we follow the framework in \cite{stuart_2010}. The structure of the proofs follows \cite{1930-8337_2016_4_1007} closely where we modify the proofs for a forward model constructed with the Radon transform. \coyq{We first need to prove that the likelihood function constructed with the Radon transform is bounded, see \Cref{prop:continuous}. We provide the proof here.}

\begin{proof}[Proof of \Cref{prop:continuous}] $ $\\
\begin{enumerate}[(i)]
\item The Radon transform $G$ is a bounded linear operator, see Chapters 6.2 and 6.6 in \cite{Hansen2021}. Therefore, it is continuous and we can find \coyq{a constant} $C >0$ such that for $\alpha \in \mathcal S(D)$ we have
\begin{equation} \label{eq:radon_bound}
    \coyq{\| G\alpha \|_2} \leq C\| \alpha \|_{L^\infty(D)} = C\alpha^+.
\end{equation}
\coyq{Representing $\alpha$ by the star-shaped mapping defined in \Cref{eq:2.star_shape}, we have} $\| F_{\text{star}}( (\xi_1,c_1) ) \|_{L^{\infty}(D)} \leq \alpha^+$. It follows
\coyq{\begin{equation}
\begin{aligned}
    \Phi((\xi_1,c_1), \boldsymbol{y}) &\leq \frac12\left(\| \boldsymbol{y} \|^2_{\Sigma_n} + \| G\circ F_{\text{star}}(\xi_1,c_1) \|^2_{\Sigma_n}\right) \\
    &\leq \frac12 \left(\| \boldsymbol{y} \|^2_{\Sigma_n} + C \| F_{\text{star}}(\xi_1,c_1) \|_{L^{\infty}(D)}\right)  \\
    & \leq \frac12 \left(\rho^2 + C\alpha^+\right),
\end{aligned}
\end{equation}
where $\| \boldsymbol{y} \|_{\Sigma_n}\leq\rho$. We define $K(\rho) := \frac12\left(\rho^2 + C\alpha^+\right)$.}
\item \coyq{In \Cref{B1} we show} that $F_{\text{star}}$ is a continuous map. Furthermore, we discussed in part $(i)$ that $G$ is also a continuous map. Therefore, for a fixed $\boldsymbol{y}\in \mathbb R^N$, $\Phi(\cdot;\boldsymbol y)$ is a composition of continuous maps. This concludes the proof.
\item Let $(\xi_1,c_1) \in X$ and $\boldsymbol{y}_1,\boldsymbol{y}_2\in \mathbb R^N$ such that $\| \boldsymbol{y}_1 \|_{\Sigma_n},\|\boldsymbol{y}_2 \|_{\Sigma_n} \leq \rho $. We have
\coyq{\begin{equation}
    \begin{aligned}
         |\Phi((\xi_1,c_1),\boldsymbol{y}_1) - &\Phi((\xi_1,c_1),\boldsymbol{y}_2)|\\
         & = \frac12| \langle \boldsymbol{y}_1 - \boldsymbol{y}_2, \boldsymbol{y}_1 + \boldsymbol{y}_2 - 2\mathcal G(\xi_1,c_1) \rangle_{\Sigma_n} | \\
         &\leq \frac12\| \boldsymbol{y}_1 - \boldsymbol{y}_2 \|_{\Sigma_n}\left \| \boldsymbol{y}_1 + \boldsymbol{y}_2 - 2\mathcal G(\xi_1,c_1) \right\| \\
         &\leq\frac12 \| \boldsymbol{y}_1 - \boldsymbol{y}_2 \|_{\Sigma_n} \left(  \|\boldsymbol{y}_1\|_{\Sigma_n} + \|\boldsymbol{y}_2\|_{\Sigma_n} + 2 \|\mathcal G(\xi_1,c_1)\|_{\Sigma_n} \right) \\
         & \leq (\rho + C\alpha^+ ) \|\boldsymbol y_1 - \boldsymbol y_2 \|_{\Sigma_n}.
    \end{aligned}
\end{equation}}
We define $M(\rho) : = (\rho + C\alpha^+ )$.
\end{enumerate}
\end{proof}

Now we show that the posterior measure $\mu^{\boldsymbol{y}}$ is a well-defined measure on $X$.
\begin{theorem}
The posterior measure $\mu^{\boldsymbol{y}}$ \Cref{eq:2.posterior} is well-defined.
\end{theorem}
\begin{proof}
We showed in \Cref{prop:continuous} that $\Phi(\cdot;\boldsymbol{y})$ is $\mu_0$-a.s. continuous and $\Phi((\xi_1,c_1),\cdot)$ is locally Lipschitz. This is sufficient condition for $\Phi$ to be jointly continuous with respect to $\mu_0 = \mu_0^1 \otimes \mu_0^2$, see the proof of theorem 3.8 in \cite{1930-8337_2016_4_1007}. Therefore, it is left to show that the normalization constant $Z$ in \Cref{eq:2.posterior} is bounded away from zero. Define $B:= B_H \times D \subset X$, where $B_H$ is an open ball in $H$. \coyq{For $\boldsymbol{y}\in \mathbb R^N$, with $\| \boldsymbol{y} \|_{\Sigma_n}\leq\rho$} we have
\begin{equation}
\begin{aligned}
    \int_X \exp( -\Phi((\xi_1,c_1);\boldsymbol{y}) )\  \mu_0(d\xi) & \geq \int_B \exp( -\Phi((\xi_1,c_1);\boldsymbol{y}) )\  \mu_0(d\xi) \\
    & \geq \int_B \exp( -K(\rho) )\  \mu_0(d\xi) \\
    &\coyq{= \exp( -K(\rho) ) \mu_0(B) > 0.}
\end{aligned}
\end{equation}
Here we used positivity of the exponential function in the first inequality, and condition $(i)$ in \Cref{prop:continuous}. Note that since $\mu_0^1$ is Gaussian and $B_H$ is an open set then the $\mu_0^1(B_H)>0$. Therefore, $\mu_0(B) = \mu_0^1(B_H)\mu_0^2(D) >0$.
\end{proof}
}

Let $\mu_1$ and $\mu_2$ be two posterior probability measures defined on $(H,\mathcal B(H))$ such that they are both absolutely continuous with respect to the prior measure $\mu_0$. The Hellinger distance \cite{pardo2018statistical} between $\mu_1$ and $\mu_2$ \copc{is} given by
\begin{equation}
    d_{\text{Hell}}(\mu_1,\mu_2) = \sqrt{ \frac{1}{2}  \int_{H} \sqrt{ \frac{d\mu_1}{d\mu_0}(\xi) } - \sqrt{ \frac{d\mu_2}{d\mu_0} (\xi) } \ \mu_0( d\xi )  }.
\end{equation}

\correct{
\begin{theorem} \label{thm:wellposedness} \cite{Dashti2017}
Let $\Phi$ be the negative log-likelihood defined \eqref{eq:2.loglike} satisfying assumptions (i)-(iii) in \Cref{prop:continuous} and $\mu_0$ be the prior measure defined on $(X,\mathcal A)$. Furthermore, let $\boldsymbol{y},\boldsymbol{y}'\in \mathbb R^N$ be two observation vector such that $\| \boldsymbol{y}-\boldsymbol{y}' \|_{\Sigma_n} \leq \rho$ for a fixed $ 0\leq \rho<\infty$. Then we can find $C=C(\rho)>0$ such that
\begin{equation}
    d_{\text{Hell}}(\mu^{\boldsymbol{y}},\mu^{\boldsymbol{y}'}) \leq C \| \boldsymbol{y} - \boldsymbol{y}' \|_{\Sigma_n}.
\end{equation}
\end{theorem}

Bayesian well-posedness expressed in \Cref{thm:wellposedness} means that the posterior distribution remains bounded when the observation vector $\boldsymbol{y}$ is perturbed.}

\section{Markov Chain Monte Carlo (MCMC) Methods} \label{sec:2.mcmc}
In this section we briefly introduce the random walk Metropolis-Hastings (RWM) method \cite{robert2013monte}, the preconditioned Crank-Nicolson (pCN) method \cite{Cotter2013}, and the Gibbs sampling method \cite{robert2013monte}. These methods allow us to sample from the posterior measure $\boldsymbol{\mu}^{\boldsymbol{y}}$. We choose pCN to sample infinite dimensional Gaussian random variables as it is suitable for random variables on function spaces. Furthermore, the convergence rate of this method is independent of the cut off value $N_{\text{KL}}$ in \eqref{eq:2.KL}.

Let $f:H \to \mathbb R^{k}$, be a $\mu_0$-measurable function. We approximate $\mathbb E f$ as the ergodic average
\begin{equation} \label{eq:2.egodic}
    \mathbb E f := \int f \ \mu_0(d \xi) = \lim_{N_{\text{sample}}\to \infty} \frac 1 {N_{\text{sample}}} \sum_{i=1}^{N_{\text{sample}}} f(\xi^{(i)}) .
\end{equation}
Here, $\{ \xi^{(i)} \}_{i=1}^{N_{\text{sample}}}$ is a Markov chain with a transition kernel $\mathbb K_y$. We refer the reader to \cite{kaipio2006statistical} for properties of a transition kernel suitable for \eqref{eq:2.egodic}.


The aim of an MCMC method is to construct a transition function (a formula to go from $\xi^{(i)}$ to $\xi^{(i+1)}$ in \eqref{eq:2.egodic}) which result in an appropriate transition kernel $\mathbb K_y$. \coyq{As an MCMC method, pCN is suitable for drawing samples from an $H$-valued Gaussian random variables, and is robust under the discretization of the random variables.} 
\coyq{We first introduce pCN to sample from a single Gaussian random variable $\xi$, then discuss} how to generalize this to sample from multiple independent Gaussian random variables.

Let $\mu^{\boldsymbol{y}}$ be the posterior measure on a Hilbert space $(H,\langle \cdot , \cdot \rangle)$ and $\mu_0 \sim \mathcal N (0,Q)$ be a prior measure. Let $\mathbb K(\xi,\cdot)$ be the transition kernel on $H$ and $\eta$ denote a measure on $H\times H$ such that if $\xi \sim \mu^{\boldsymbol{y}}$ then $\zeta|\xi \sim \mathbb K_y(\xi,\cdot)$. We denote by $\eta ^ \perp$ the measure where the roles of $\xi$ and $\zeta$ are reversed. If $\eta^\perp$ is equivalent to $\eta$, in the sense of measure, then the Radon-Nikodym derivative $d\eta^\perp/d \eta$ is well defined and we can define the acceptance probability
\begin{equation} \label{eq:2.acceptance_pcn1}
    a(\xi,\zeta) = \min \left\{ 1, \frac{d\eta^\perp}{d \eta}(\xi,\zeta)\right\}.
\end{equation}
This means $\xi^{(i+1)} := \zeta$ with probability \coyq{$a(\xi^{(i+1)},\zeta)$} and $\xi^{(i+1)} := \xi$ otherwise.

The standard random walk proposal function results in $\eta^\perp$ that is singular with respect to $\eta$ \cite{Cotter2013} when $\xi$ is an $H$-valued function. This results in rejecting all proposed moves with probability 1. The pCN proposal function \cite{Cotter2013} is given by
\begin{equation} \label{eq:2.proposal_pcn}
   \coyq{ \zeta = (1 - b ^2)^{1/2}\xi + b \varrho,}
\end{equation}
where $\varrho \sim \mathcal N(0,Q)$, and \coyq{$b \in [0,1]$}. This choice of the proposal results in a well defined $d\eta^\perp/d \eta$ given by \cite{Cotter2013}
\begin{equation} \label{eq:2.acceptance_pcn1.2}
    \frac{d\eta^\perp}{d \eta}(\xi,\zeta) = \exp( \eta(\xi,\zeta) - \eta(\zeta, \xi) ) = \exp( \Phi(\xi;{\boldsymbol{y}}) - \Phi(\zeta;{\boldsymbol{y}}) ).
\end{equation}
We summarize the pCN sampling method in \Cref{alg:2.pCN}.

\begin{algorithm}
\caption{pCN for collecting $N_{\text{sample}}$ samples}
\label{alg:2.pCN}
\begin{algorithmic}[1]
\STATE{Set $j=0$ and take the initial sample $\xi^{(0)}$.}
\FOR{$j\leq N_{\text{sample}}$ }
\STATE{Propose \coyq{$\zeta = (1 - b ^2)^{1/2}\xi^{(k)} + b \varrho$}, $\varrho\sim \mathcal N(0,Q)$.}
\STATE{Set $\xi^{(j+1)} = \zeta$ with probability $a(\xi^{(j)},\zeta)$ defined in \eqref{eq:2.acceptance_pcn1} together with \eqref{eq:2.acceptance_pcn1.2}, otherwise $\xi^{(j+1)} = \xi^{(j)}$.}
\STATE{ $j \leftarrow j+1$.}
\ENDFOR
\end{algorithmic}
\end{algorithm}

Recall that star-shaped prior \coyq{for each inclusion is defined as} the joint random variable $(\xi_i,c_i)$. To construct an MCMC method to sample from \coyq{$\mu^{\boldsymbol{y}}$} we use a \emph{Gibbs}-type \cite{robert2013monte} sampling method. Such methods alternatively sample from the random variables $\xi_i|c_i,\boldsymbol{y}$ and $c_i|\xi_i,\boldsymbol{y}$ and constructed Markov chain $[\xi_i^{(j)},c_i^{(j)}]_{j=1}^{N_{\text{sample}}}$ contains ergodic properties as in \Cref{eq:2.egodic} \cite{robert2013monte}. In this paper we construct a Gibbs sampler following the structure:
\begin{equation}
    \xi_i^{(j+1)} \sim \mathbb K_{\boldsymbol{y}}^{c_i^{(j)}} (\xi_i^{(j)},\cdot), \qquad c_i^{(j+1)} \sim \mathbb L_{\boldsymbol{y}}^{\xi_i^{(j+1)}} (c_i^{(j)},\cdot).
\end{equation}
Here, $\mathbb K_{\boldsymbol{y}}^{c_i}$ and $\mathbb L_{\boldsymbol{y}}^{\xi_i}$ are Metropolis-Hastings Markov kernel reversible with respect to $\xi_i|c_i,\boldsymbol{y}$ and $c_i|\xi_i,\boldsymbol{y}$, respectively. For the random variable $\xi_i|c_i,\boldsymbol{y}$ we use a pCN proposal function with the acceptance probability of
\begin{equation} \label{eq:2.acceptance_pcn2.1}
    a(\xi_i,\zeta) = \min\{ 1, \exp( \Phi((\xi_i,c_i);\boldsymbol{y}) - \Phi(\zeta_0,(\zeta,c_i);\boldsymbol{y}) \}.
\end{equation}
For the random variable $c_i|\xi_i,\boldsymbol{y}$ we use the standard random walk MH proposal function with the acceptance probability of
\begin{equation} \label{eq:2.gibbs_geom}
    r(c_i,o) = \min\{ 1, \exp( \Phi((\xi_i,c_i);\boldsymbol{y}) - \Phi((\xi_i,o);\boldsymbol{y}) \}.
\end{equation}
Note that we dropped the uniform prior in \eqref{eq:2.gibbs_geom} as both $c_i$ and $o_i$ are restricted to the same interval. We summarize the Gibbs sampler for $\mu^{\boldsymbol{y}}$ constructed with the star-shaped prior in \Cref{alg:3.gibbs}. \correct{Here, $N_{\text{pCN}}$ and $N_{\text{MH}}$ represent the number of within-Gibbs samples for each component.}


\correct{
\begin{algorithm}
\caption{ The Gibbs sampling method for collecting $N_{\text{sample}}$ samples}
\label{alg:3.gibbs}
\begin{algorithmic}[1]
\STATE{Set $j=0$ and take the initial sample $( \xi_i^{(0)}, c_i^{(0)})$.}
\FOR{$j\leq N_{\text{sample}}$ }
\STATE{Set $\xi_i^{(j,0)} = \xi_i^{(j)}$.}
\FOR{$k\leq N_{\text{pCN}}$}
\STATE{Propose $\zeta = (1 - b_1 ^2)^{1/2}\xi^{(j,k)} + b_1 \varrho$, $\varrho\sim \mathcal N(0,Q)$.}
\STATE{Set $\xi_i^{(j,k+1)} = \zeta$ with probability $a(\xi_i^{(j,k)},\zeta)$ defined in \eqref{eq:2.acceptance_pcn2.1}, otherwise $\xi_i^{(j,k+1)} = \xi_i^{(j,k)}$.}
\STATE{ $k \leftarrow k+1$.}
\ENDFOR
\STATE{Set $\xi_i^{(j+1)} = \xi_i^{(j,k+1)}$.}

\STATE{ Set $c_i^{(j,0)} = c_i^{(j)}$.}
\FOR{$k\leq N_{\text{MH}}$}
\STATE{Propose $o = c_i^{(j,k)} + b_2 \rho$, $\rho\sim \mathcal N(0,I)$.}
\STATE{Set $c_i^{(j,k+1)} = o$ with probability $r(c_i^{(j,k)},o)$ defined in \eqref{eq:2.gibbs_geom}, otherwise $c_i^{(j,k+1)} = c_i^{(j,k)}$.}
\STATE{ $k \leftarrow k+1$.}
\ENDFOR
\STATE{ Set $c_i^{(j+1)} = c_i^{(j,k+1)}$. }
\STATE{ $j \leftarrow j+1$.}
\ENDFOR
\end{algorithmic}
\end{algorithm}
}


\bibliographystyle{siamplain}
\bibliography{references}
\end{document}

%% file: shared.tex
\usepackage{lipsum}
\usepackage{amsfonts}
\usepackage{graphicx}
\usepackage{epstopdf}
\usepackage{algorithmic}
\usepackage{enumerate}
\usepackage{bbm}
\usepackage{graphicx,epstopdf}
\usepackage{soul}
\usepackage[caption=false]{subfig}
\ifpdf
  \DeclareGraphicsExtensions{.eps,.pdf,.png,.jpg}
\else
  \DeclareGraphicsExtensions{.eps}
\fi

\usepackage{algorithmic}
\Crefname{ALC@unique}{Line}{Lines}

\newsiamremark{remark}{Remark}
\newsiamremark{hypothesis}{Hypothesis}
\crefname{hypothesis}{Hypothesis}{Hypotheses}
\newsiamthm{claim}{Claim}

\headers{UQ of Inclusion Boundaries in the Context of X-ray Tomography}{B. M. Afkham, Y. Dong, and P. C. Hansen}

\title{Uncertainty Quantification of Inclusion Boundaries in the Context of X-ray Tomography\thanks{Submitted to the editors DATE.\funding{\copc{This project was funded by a Villum Investigator grant (no.~25893) from The Villum Foundation.}}}
}

\author{Babak Maboudi Afkham\thanks{Department of Applied Mathematics and Computer Science, Technical University of Denmark, Kgs. Lyngby, Denmark
  (\email{bmaaf@dtu.dk}, \email{yido@dtu.dk}, \email{pcha@dtu.dk}).}
\and Yiqiu Dong\footnotemark[2] \and Per Christian Hansen\footnotemark[2]
}

\newcommand{\correct}[1]{{\color{black} #1}}
\newcommand{\cofinal}[1]{{\color{black} #1}}
\newcommand{\coyq}[1]{{\color{black} #1}}
\newcommand{\copc}[1]{{\color{black} #1}}
\newcommand{\delete}[1]{}

\usepackage{amsopn}
